\documentclass{article}
\usepackage[new]{old-arrows}
\usepackage{macros}	
\usepackage{tikz-cd}

\def\F{\mathbf{F}}
\def\PP{\mathbf{P}}

\DeclareMathOperator{\car}{{\rm char}}
\DeclareMathOperator{\Jac}{{\rm Jac}}

\DeclareMathOperator{\ord}{ord}

\def\A{\mathbf{A}}
\def\P{\mathbf{P}}
\newcommand{\class}[1]{\left< #1 \right>}
\newcommand{\Ino}{\operatorname{{Ino}}}
\newcommand{\II}{\mathrm{II}}
\newcommand{\III}{\mathrm{III}}
\newcommand{\I}{\mathrm{I}}
\newcommand{\IV}{\mathrm{IV}}
\def\K3{$K$3}

\title{Higher modularity of elliptic curves over function fields}
\author{Adam Logan and Jared Weinstein, with an interlude by Masato Kuwata}
\begin{document}
\maketitle

\abstract{We investigate a notion of ``higher modularity'' for elliptic curves over function fields.   Given such an elliptic curve $E$ and an integer $r\geq 1$, we say that $E$ is $r$-modular when there is an algebraic correspondence between a stack of $r$-legged shtukas, and the $r$-fold product of $E$ considered as an elliptic surface.  The (known) case $r=1$ is analogous to the notion of modularity for elliptic curves over $\Q$.   Our main theorem is that if $E/\F_q(t)$ is a nonisotrivial elliptic curve with tame fibers whose conductor has degree 4, then $E$ is 2-modular.   Ultimately, the proof uses properties of K3 surfaces.   Along the way we prove a result of independent interest:  A K3 surface admits a finite morphism to a Kummer surface attached to a product of elliptic curves if and only if its Picard lattice is rationally isometric to the Picard lattice of such a Kummer surface.}

\tableofcontents

\section{Introduction}\label{sec:introduction}

Let $K$ be a function field over a finite field $\F_q$, and let $E/K$ be an elliptic curve.   In this article we investigate a notion of ``higher modularity'' for $E/K$.

\subsection{Analytic modularity, geometric modularity}


Let $E$ be an elliptic curve over the rational numbers $\Q$, with conductor $N$.  We review what it means for $E/\Q$ to be {\em modular}.  Following \cite{UlmerAnalogies}, we might distinguish between two notions of modularity for $E/\Q$:
\begin{enumerate}
\item (Analytic modularity) There exists a cuspidal holomorphic form $f$ of weight $2$ and level $\Gamma_0(N)$ such that $L(f, \chi,s)=L(E,\chi,s)$ for all Dirichlet characters $\chi$.
\item (Geometric/motivic modularity) There exists a non-constant morphism $X_0(N)\to E$ defined over $\Q$, where $X_0(N)$ is the modular curve of level $\Gamma_0(N)$.
\end{enumerate}
The two notions are equivalent.  If $E$ is geometrically modular, the required cusp form $f$ may be found by pulling back a nonzero holomorphic differential form on $E$ through $X_0(N)\to E$.  Conversely, if $E$ is analytically modular, there exists an elliptic curve quotient $\Jac X_0(N)\to E'$ corresponding to its cusp form $f$.  The equality of $L$-functions implies that the representations of $\Gal(\overline{\Q}/\Q)$ on the rational $\ell$-adic Tate modules of $E$ and $E'$ are isomorphic.  A theorem of Faltings \cite{FaltingsMordell} implies that $E$ and $E'$ are isogenous, from which we deduce that $E$ is geometrically modular.

In any case, the conjecture of Shimura-Tanayama is a theorem \cite{BCDT};  every elliptic curve $E/\Q$ is analytically modular, hence geometrically modular.  An important consequence is that the Heegner point on $X_0(N)$ for an imaginary quadratic field $F$ satisfying the Heegner hypothesis relative to $N$ gives rise to a Heegner point $P_F\in E(F)$.  The celebrated theorem of Gross-Zagier \cite{GrossZagier} relates the height of $P_F$ to the derivative $L'(E,F,1)$.  The theorem of Kolyvagin \cite{Kolyvagin} shows that if $P_F$ is not torsion, then it generates a finite-index subgroup of $E(F)$, and furthermore one gets finiteness of the Shafarevich-Tate group.  

Now suppose instead $E$ is an elliptic curve over a function field $K$ with field of scalars $\F_q$.   We assume that $E$ is nonisotrivial, meaning that $j(E)\not\in\F_q$.   Then $L(E,\chi,s)$ is entire for every Hecke character $\chi$.  (In fact $L(E,\chi,s)$ is a polynomial in $q^{-s}$.)  The analogue of analytic modularity for~$E$ is the condition that there exists a cuspidal automorphic representation $\pi$ of $\GL_2$ over $K$, such that $L(\pi,\chi,s)=L(E,\chi,s)$ for all $\chi$.  The analytic modularity of $E$ follows from a theorem of Deligne together with Weil's converse theorem (see \cite{UlmerAnalogies} for a discussion).

The analogue of geometric modularity for $E$ is the subject of this paper.  The interest here hinges on the correct analogue for the modular curve.  Let $X/\F_q$ be the nonsingular projective curve with function field $K$.  
Let $N$ be the conductor of $E$, with support $\Sigma\subset \abs{X}$.  As a natural starting point, we may consider the Drinfeld modular curve $\DrMod(\Gamma_0(N_f);\infty)$.  Here $\infty\in\abs{X}$ is a place of split multiplicative reduction for $E$;  thus $N=N_f+(\infty)$ for a divisor $N_f$ which is prime to $\infty$.  If $A$ is the ring of functions on $X$ regular away from $\infty$, then $\DrMod(\Gamma_0(N_f);\infty)$ parametrizes Drinfeld $A$-modules of rank 2 endowed with $\Gamma_0(N_f)$-structure.  There is a morphism $\DrMod(\Gamma_0(N_f);\infty)\to X\smallsetminus \Sigma$ (sending a Drinfeld module to its ``characteristic''), whose fibers are smooth curves.  Let $\DrMod(\Gamma_0(N_f);\infty)_K$ be the generic fiber of this morphism, and let $\DrMod(\Gamma_0(N_f);\infty)_{\overline{K}}$ be its base change to an algebraic closure $\overline{K}/K$.

Drinfeld's computation \cite{DrinfeldEllipticModules} of the cohomology of Drinfeld modular curves shows that $H^1(E_{\overline{K}})$ appears as a summand of $H^1(\DrMod(\Gamma_0(N_f);\infty)_{\overline{K}})$.   (We mean \'etale cohomology with $\Q_\ell$-coefficients, considered as a representation of the Galois group of $K$.) We may once again apply Faltings' theorem to conclude that there is a nonconstant morphism $\DrMod(\Gamma_0(N_f);\infty)_K\to E$.  If $K'/K$ is a separable quadratic extension satisfying the Heegner hypothesis relative to $E$, there is a Heegner point $P_{K'}\in E(K')$, whose height is related to the derivative $L'(E/K',1)$, see~\cite[Remark 1.5]{YunZhangShtukasII}.  If $L'(E/K',1)\neq 0$, then $P_{K'}$ is not a torsion point.  But also we have that the rank of $E(K')$ is $\leq 1$,  since the ``easy inequality'' of the Birch and Swinnerton-Dyer conjecture is a theorem for function fields.  We conclude immediately that the rank is exactly 1, and so the conjecture of Birch and Swinnerton-Dyer holds for $E'/K$.

\subsection{Stacks of shtukas, and the definition of higher modularity}

The notion of geometric modularity in the function field setting may be generalized far beyond the Drinfeld modular curve.  Namely, for $r=1,2,\dots$ we have the moduli stack of $r$-legged shtukas of rank 2, which lies over the $r$-fold product $X^r$ over $\F_q$.  These were introduced by Drinfeld \cite{DrinfeldLanglandsConjectureForGL2} in the case $r=2$ to establish the Langlands correspondence for $\GL_2$.  For general $r$, these are examples of the spaces used by V. Lafforgue \cite{VincentLafforgueChtoucas} to prove the automorphic-to-Galois direction of the Langlands correspondence for general reductive groups.  

The spaces of shtukas relevant to us are relative to the group $G=\PGL_2$.   To define them, we need an effective divisor $N$ of $X$, with support $\Sigma$.  We also need a subset $\Sigma_\infty\subset \Sigma$, such that each $v\in \Sigma_\infty$ appears with multiplicity 1 in $N$.  
Finally, we need an integer $r\geq 1$ satisfying the parity condition $\#\Sigma_\infty\equiv r \pmod{2}$.  Following \cite{YunZhangShtukasII}
 we let
\[ \lambda_r\from \Sht^r_G(\Gamma_0(N);\Sigma_\infty)\to (X\smallsetminus \Sigma)^r  \]
denote the moduli stack of $G$-shtukas with $\Gamma_0(N)$-level structure,  with $r$ ``moving legs'' in $X\smallsetminus \Sigma$, and ``fixed legs'' at the places in $\Sigma_\infty$ (these are analogous to the archimedean places in the number field setting).   These legs are ``minuscule''; i.e., they are the type considered by Drinfeld in his original definition.  We save the precise definition of $\Sht^r_G(\Gamma_0(N);\Sigma_\infty)$ for \S\ref{sec:shtukareview}.  It is a Deligne-Mumford stack, generically smooth of relative dimension $r$ over $(X\smallsetminus\Sigma)^r$.  When $r$ is even and $\Sigma_\infty=\emptyset$ we simply write $\Sht_G^r(\Gamma_0(N))$.  

For example,  when $r=0$, $\Sht^0_G(\Gamma_0(N))$ is the discrete set $G(K)\backslash G(\mathbf{A}_K)/\Gamma_0(N)$.   When $r=1$ and $\Sigma_\infty=\set{\infty}$ is a singleton, one of the connected components of $\Sht^1_G(\Gamma_0(N);\set{\infty})$ is isomorphic over $X\smallsetminus \Sigma$ to a quotient of $\DrMod(\Gamma_0(N_f);\infty)$;  this is the shtuka correspondence \cite{MumfordAnAlgebroGeometricConstruction}.  


We can now start to describe our notion of ``higher modularity'' for our elliptic curve $E/K$.   Let $U=X\smallsetminus \Sigma$, and let $\E\to U$ be the family of elliptic curves with generic fiber $E$.  The idea goes like this:  Consider the family $h^1(\E)$ of motives over $U$.  Let $h^1(\E)^{\boxtimes r}$ be its $r$th external tensor power, meaning the family of motives over $U^r$ whose fiber over $(s_1,\dots,s_r)$ is $h^1(\E_{s_1})\otimes \cdots \otimes h^1(\E_{s_r})$.  For $E$ to be $r$-modular should mean that $h^1(\E)^{\boxtimes r}$ is a quotient of the (compactly supported) cohomology of a stack of $r$-legged shtukas.   
 
The precise condition involves the existence of a degree 0 algebraic correspondence between $\Sht^r:=\Sht^r_G(\Gamma_0(N);\Sigma_\infty)$ and $\E^r$.  By a degree 0 algebraic correspondence, we mean a $\Q$-linear combination of irreducible closed substacks
\[ Z\subset \Sht^r \times_{U^r} \E^r\]
such that $Z\to U^r$ is proper with $r$-dimensional fibers.  

Let $\eta$ (resp., $\eta_r$) be the generic point of $X$ (resp., $X^r$), and let $\overline{\eta}_r\to \eta_r$ be an algebraic closure lying over $\overline{\eta}^r\to \eta^r$. 
Let $\ell$ be a prime not dividing $q$.  The $\ell$-adic cycle class of $Z_{\overline{\eta}_r}$ gives rise to a map
\begin{equation}
\label{EqPZ}
 p_Z\from H^r_c(\Sht^r_{\overline{\eta}_r},\Q_\ell)\to \bigotimes_{i=1}^r H^1(\E_{\overline{\eta}},\Q_\ell)
 \end{equation}
(Explanation:  the $\ell$-adic cycle class of $Z_{\overline{\eta}_r}$ belongs to $H^{2r}_c(\Sht^r_{\overline{\eta}_r}\times (\E_{\overline{\eta}})^r,\Q_\ell)(r)$, which maps via the K\"unneth isomorphism to $H^r_c(\Sht^r_{\overline{\eta}_r},\Q_\ell)(r)\otimes \bigotimes_{i=1}^r H^1(\E_{{\overline{\eta}}},\Q_\ell)$.  Poincar\'e duality identifies the latter with the space of maps as in \eqref{EqPZ}.)  For $Z=\sum_j a_j Z_j$ a formal $\Q$-linear combination of correspondences $Z_j$ as above, we define $p_Z=\sum_j a_j p_{Z_j}$. 

\begin{defn}\label{DefnRModularity} Let $K/\F_q$ be a function field, and let $E/K$ be a nonisotrivial elliptic curve whose conductor $N$
has support $\Sigma$.  Let $U=X\smallsetminus \Sigma$, and let $\E\to U$ be the family of elliptic curves with generic fiber $E$.  We say that $E/K$ is {\em $r$-modular} if for some $\Sigma_\infty\subset \Sigma$ there exists a degree 0 algebraic correspondence $Z$ between $\Sht^r_G(\Gamma_0(N);\Sigma_\infty)$ and $\E^r$ such that the map $p_Z$ of \eqref{EqPZ} is surjective.
\end{defn}

%
%

\begin{rmk} When $r$ is odd, we must assume that $E/K$ has at least one place of multiplicative reduction in order to have an interesting notion of $r$-modularity.  
\end{rmk} 


\begin{conj}  \label{ConjMain} A nonisotrivial elliptic curve $E/K$ is $r$-modular for all $r\geq 1$.
\end{conj} 

The main goal of this article is to give the first examples of 2-modular and 3-modular elliptic curves.  Our results apply to what is arguably the simplest class of elliptic fibrations, namely the {\em extremal rational elliptic fibrations} $\E\to X$.
Here ``rational'' means that $\E$ is birational to $\P^2$ (and consequently $X\isom \P^1$), and ``extremal'' means that $\E\to X$ has Mordell-Weil rank 0.   (The generic rational elliptic fibration has Mordell-Weil rank 8.)   For such a fibration we necessarily have $\deg N=4$ and $L(E,s)=1$ identically.  Conversely, if an elliptic fibration $\E\to \P^1$ has degree 4 conductor, then it is extremal rational.  

The extremal rational elliptic fibrations $\E\to X$ over a field of characteristic $0$ are classified in \cite{MirandaPersson}.   The number of singular geometric fibers of a extremal rational elliptic fibration is either two, three, or four, with two occurring only if $\E\to X$ is isotrivial or if it has wild fibers (which happens only in characteristics 2 and 3).   A nonisotrivial extremal rational elliptic fibration always has at least one multiplicative fiber, and so there is always a nontrivial notion of higher modularity for $\E$.

\begin{thm}\label{ThmMain2Modularity} Let $\E\to X$ be a nonisotrivial tame extremal rational elliptic fibration with generic fiber $E$.   Then $E$ is 2-modular. 
\end{thm}

In the next two subsections we will (a) explain why one should believe Conjecture \ref{ConjMain}, and (b) sketch the proof of Theorem \ref{ThmMain2Modularity}.

\subsection{Relation to the Tate conjecture}

Here we motivate the definition of higher modularity.  Using analytic modularity (which is a theorem) combined with prior results on the cohomology of stacks of shtukas, we will reduce Conjecture \ref{ConjMain} to a sufficiently strong version of the Tate conjecture.   Let $\eta=\Spec K$ be the generic point of $X$, and let $\overline{\eta}=\Spec \overline{K}$ be a geometric generic point.

We use the normalization of the global Langlands correspondence found in \cite{DrinfeldLanglandsConjectureForGL2}, adapted to the group $G=\PGL_2$.  This is a bijection between isomorphism classes:
\begin{itemize}
\item Cuspidal automorphic representations $\pi$ of $G$ with coefficients in $\overline{\Q}_\ell$, and
\item Irreducible representations $\sigma\from \Gal(\overline{K}/K)\to \GL_2\overline{\Q}_\ell$, such that $\det \sigma=\overline{\Q}_\ell(-1)$.  
\end{itemize}
Suppose that $\pi$ and $\sigma$ correspond.   Let $N$ be the conductor of $\pi$, in the sense that $\dim \pi^{\Gamma_0(N)}=1$, and let $\Sigma$ be the support of $N$.   Let $U=X\smallsetminus \Sigma$.  Then $\sigma$ factors through a representation of $\pi_1(U,\overline{\eta})$ of conductor $N$.  The correspondence is characterized by the property that for all $v\in U$ with residue field $\F_{q_v}$ and Frobenius $\Frob_v$, we have $L(s-1/2,\pi_v)=\det(1-\sigma(\Frob_v) q_v^{-s})^{-1}$.  

\begin{rmk} We have chosen this normalization of the global Langlands correspondence so that the $H^1$ of a nonisotrivial elliptic curve over $K$ corresponds to a cuspidal automorphic representation of $G$.  Under the ``usual'' normalization of the correspondence, the Langlands parameter of an automorphic representation of $G$ is a homomorphism $\Gal(\overline{K}/K)\to \SL_2\overline{\Q}_\ell$.  This is largely an aesthetic choice;  working with $G$ rather than $\GL_2$ saves us certain notational headaches having to do with the center of $\GL_2$.  
\end{rmk}

We now reference some facts about the relation between the cohomology of stacks of shtukas and the global Langlands correspondence.  Let $r\geq 1$, write $\eta^r$ for the $r$-fold product of $\eta$ over $\F_q$, and write $\eta_r$ for the generic point of $X^r$.   Let $\overline{\eta}_r\to\eta_r$ be an algebraic closure lying over $\overline{\eta}^r\to \eta^r$. Then there is a homomorphism 
\begin{equation}
\label{EqGaloisMap}
\pi_1(U^r,\overline{\eta}_r) \to \pi_1(U,\overline{\eta})^r
\end{equation}
The cohomology $H^r_c(\Sht^r_G(\Gamma_0(N);\Sigma_\infty)_{\overline{\eta}_r},\Q_\ell)$ admits commuting actions of $\pi_1(U^r,\overline{\eta}_r)$ and the Hecke algebra $T_0(N)$ for $G$ with level $\Gamma_0(N)$.   The following proposition is along the lines of the Kottwitz conjecture for Shimura varieties.   It appears later as Proposition \ref{PropCohoOfShtukaSpace}.  We warn that it is conditional on the extension of the main results of Xue in \cite{Xue} and \cite{XueFiniteness} to the situation $\Sigma_\infty\neq \emptyset$.  
\begin{prop}  
\label{PropCohoOfShtukaSpaceIntro}  Assume that $\#\Sigma_\infty\leq 1$.  Let $\sigma\from \pi_1(U,\overline{\eta})\to \GL_2\overline{\Q}_\ell$ be an irreducible representation of conductor $N$ with $\det \sigma=\overline{\Q}_\ell(-1)$.   Then as representations of $\pi_1(U^r,\overline{\eta}_r)$, the external tensor power $\sigma^{\boxtimes r}$ appears as a 
subspace of $H^r_c(\Sht^r_G(\Gamma_0(N);\Sigma_\infty)_{\overline{\eta}_r},\overline{\Q}_{\ell})$.   (Here $\pi_1(U^r,\overline{\eta}_r)$ acts on $\sigma^{\boxtimes r}$ by means of the homomorphism \eqref{EqGaloisMap}.)
\end{prop}

 Now suppose $E/K$ is a nonisotrivial elliptic curve of conductor $N$.  Let $\E\to U$ be the family of elliptic curves with generic fiber $E$.    Let $\sigma$ be the the 2-dimensional representation of $\pi_1(U,\overline{\eta})$ on $H^1(\E_{\overline{\eta}},\Q_\ell)$.   Then $\sigma$ is irreducible, and (owing to the Weil pairing on $E$) we have $\det \sigma=\Q_\ell(-1)$.  By Proposition \ref{PropCohoOfShtukaSpaceIntro}, there exists a $\pi_1(U^r,\overline{\eta}_r)$-equivariant embedding $\sigma^{\boxtimes r}\injects H^r_c(\Sht^r_{\overline{\eta}_r},\Q_\ell)$, where $\Sht^r=\Sht_G^r(\Gamma_0(N);\Sigma_\infty)$.   Therefore we have embeddings:
 \[ (\sigma\otimes \sigma(1))^{\boxtimes r} \injects 
 H^r_c(\Sht^r_{\overline{\eta}_r},\Q_\ell) \otimes H^1(\E_{\overline{\eta}},\Q_\ell)^{\boxtimes r}(r) 
\injects
H^{2r}_c(\Sht^r_{\overline{\eta}_r}\times \E^r_{\overline{\eta}_r},\Q_\ell(r))
\]
The representation $(\sigma\otimes\sigma(1))^{\boxtimes r}$ contains a $\pi_1(U^r,\overline{\eta}_r)$-invariant vector (again due to the Weil pairing), and therefore so does $H^{2r}_c(\Sht^r_{\overline{\eta}_r}\times \E^r_{\overline{\eta}_r},\Q_\ell(r))$. If the stack of shtukas were a projective variety, the Tate conjecture would predict that $w$ is the class of an algebraic correspondence $Z$.  This is why we have claimed that a sufficiently strong version of the Tate conjecture predicts the algebraic correspondence $Z$ appearing in Definition \ref{DefnRModularity}.


\subsection{Strategy of proof of 2-modularity}  

Let $N$ be an effective divisor of $X=\P^1_{\F_q}$ of degree 4 with support $\Sigma$.   As before, let $U=X\smallsetminus \Sigma$.   Let $\Sigma_\infty\subset \Sigma$ be a set of places appearing with multiplicity 1 in $N$, such that $\#\Sigma_\infty$ is even.   The stack of shtukas in this situation can be described with simple equations, at least up to birational equivalence.  More precisely, each fiber of 
\[ \Sht^2_G(\Gamma_0(N);\Sigma_\infty) \to U^2 \]
is birational to an elliptic surface (Theorem \ref{ThmPresentationOfShtukas}).  
For $q$ large these elliptic surfaces are of general type, which makes it hard to imagine finding the desired algebraic correspondence directly.  Instead, we found an unexpected relation between shtuka spaces with different numbers of legs, which we have called the ``coincidence map''.  The coincidence map is a rational map
\[ c\from \Sht^2_G(\Gamma_0(N);\Sigma_\infty) \dashrightarrow \Sht^1_G(\Gamma_0(N);\Sigma_\infty').\]
between stacks of shtukas with 2 and 1 legs, respectively; its domain of definition meets every fiber over $U^2$.   Here $\Sigma'_\infty\subset \Sigma_\infty$ is another set of places appearing with multiplicity 1 in $N$, such that $\#\Sigma_\infty'$ is odd.  The coincidence map fits into a cartesian diagram of varieties with rational maps:
\begin{equation}
\label{EqCartesianIntro}
\xymatrix{
\Sht_G^2(\Gamma_0(N);\Sigma_\infty)\ar@{-->}[rr]^{c\times\lambda_2} \ar@{-->}[d] &&   \Sht^1_G(\Gamma_0(N);\Sigma_\infty') \times U^2  \ar[d]^{\lambda_1\times \id}  \\
\Coinc^3_G(\Gamma_0(N);\Sigma_\infty'') \ar[rr]_{\lambda_3} && U^3
}
\end{equation}
Here, $\Coinc^3_G(\Gamma_0(N);\Sigma_\infty'')$ is a moduli space of what we have called ``3-legged coincidences'':   these are certain modifications of vector bundles on the projective line, but Frobenius is not involved (and indeed the space of coincidences can be defined in any characteristic).  The set of archimedean places $\Sigma''$ is the symmetric difference between $\Sigma$ and $\Sigma'$.  The leg map $\lambda_3$ in \eqref{EqCartesianIntro} is a double cover of $U^3$, branched over a divisor of degree $(2,2,2)$.  

Now suppose $E/K$ is a nonisotrivial elliptic curve of conductor $N$, corresponding to a family of elliptic curves $\E\to U$.  Let $\infty$ be a place of multiplicity 1 in $N$.  We apply the fact that $E$ is 1-modular, which is to say there is a finite morphism
\begin{equation}
\label{Eq1Modular}
\Sht^1_G(\Gamma_0(N),\set{\infty})\to \E  \\
 \end{equation}
commuting with the maps to $U$.  Combining \eqref{EqCartesianIntro} with \eqref{Eq1Modular}, we obtain a dominant rational map:
 \begin{equation}
 \label{EqMapFromSht2ToZ2}
 \Sht^2_G(\Gamma_0(N)) \dashrightarrow \caZ^2(\E)
 \end{equation}
where $\caZ^2(\E)$ is defined as the cartesian product:
 \begin{equation}
 \label{EqDefnZ}
 \xymatrix{
\caZ^2(\E) \ar[r] \ar[d] & \E\times U^2 \ar[d] \\
\Coinc^3_G(\Gamma_0(N);\set{\infty}) \ar[r] & U^3
} 
 \end{equation}
We have a map $\caZ^2(\E)\to U^2$ obtained by composing the upper row of \eqref{EqDefnZ} with the projection; with respect to this, the morphism \eqref{EqMapFromSht2ToZ2} commutes with the maps to $U^2$.  By replacing $\E\to U$ with the complete elliptic fibration over $\P^1$ in \eqref{EqDefnZ}, it is possible to redefine $\caZ^2(\E)$ so that the morphism $\caZ^2(\E)\to U^2$ is proper and smooth; we do this.  Then each fiber of $\caZ^2(\E)\to U^2$ is the base change of a rational elliptic fibration by a double cover of $\P^1$ ramified at 2 points.   Generically, such a base change is a K3 surface of Picard rank 18.  Our plan of attack now shifts to the study of K3 surfaces.

Recall that the Kummer surface $\Km(A)$ of an abelian surface $A$ is the K3 surface obtained by resolving singularities on the quotient $A/[-1]$.  If $A$ is the product of generic nonisogenous elliptic curves, then $\Km(A)$ has Picard rank 18.

We prove the following theorem:
\begin{thm}  \label{ThmIsogenyBetweenK3s} Let $S$ be a K3 surface over an algebraically closed field.  Assume there is an isometry $\Pic S\otimes\Q\isom \Pic K \otimes\Q$, where $K$ is a Kummer surface of the form $\Km(E_1\times E_2)$, where $E_1,E_2$ are nonisogenous elliptic curves.  Then there exists a finite morphism from $S$ to a Kummer surface of this form.
\end{thm}

\begin{rmk}  Theorem \ref{ThmIsogenyBetweenK3s} is in the spirit of a result of Buskin \cite[Theorem 1.1]{Buskin} which states that if $S$ and $S'$ are two complex K3 surfaces, then a Hodge isometry $H^2(S,\Q)\isom H^2(S',\Q)$ is necessarily induced from an algebraic correspondence between $S$ and $S'$.   
\end{rmk}

Recall our family of K3 surfaces $\caZ^2(\E)\to U^2$.  The idea now is to apply Theorem \ref{ThmIsogenyBetweenK3s} (suitably modified to work in families) to conclude the existence of a family of abelian surfaces $\mathcal{A}\to U^2$, which factors as a product of elliptic curves \'etale-locally on $U^2$, together with an isogeny $\caZ^2(\E)\to \Km(\mathcal{A})$.  

The next phase of the operation is to find an isogeny between $\Km(\mathcal{A})$ and $\Km(\E^2)$ over $U^2$, where $\E\to U$ is the given family of elliptic curves.  This implies that $\E$ is 2-modular:  there is then a dominant rational map $\Sht^2\dashrightarrow \Km(\E^2)$ over $U^2$, which lifts over $\E^2\dashrightarrow \Km(\E^2)$ to produce (after taking Zariski closures) the desired correspondence between $\Sht^2$ and $\E^2$.  

We can leverage prior knowledge of the cohomology of $\Sht^2$ to force an isogeny between $\Km(\mathcal{A})$ and $\Km((\E')^2)$, where $\E'\to U$ is some family of elliptic curves of conductor $N$, a posteriori 2-modular (Theorem \ref{ThmASeparates}).  In the case that $\E$ is semistable (meaning $N$ is multiplicity-free), this is enough to force an isogeny between $\E$ and $\E'$, and so $\E$ is 2-modular as well.  In the other cases, explicit calculations were necessary to find the isogeny between $\mathcal{Z}^2(\mathcal{E})$ and $\Km(\E^2)$.  

The phenomenon of the coincidence map applies to spaces of shtukas with arbitrarily many legs, and opens up an avenue of attack to prove $r$-modularity for any $r$.

\begin{thm} \label{ThmMain3Modularity} Assume $q$ is odd.  Let $E/\F_q(t)$ be the Legendre elliptic curve, with Weierstrass equation $y^2=x(x-1)(x-t)$.    Then $E$ is 3-modular.
\end{thm}

This $E$ has conductor $N=(0)+(1)+2(\infty)$.    Theorem \ref{ThmMain3Modularity} was proved by finding a birational map between two families of Calabi-Yau threefolds $\caZ^3(\E)$ and $\Km(\E^3)$ over $U^3$, where $U=\mathbf{P}^1\smallsetminus\set{0,1,\infty}$.  Here $\caZ^3(\E)$ is defined as a fiber product analogous to \eqref{EqDefnZ}, and $\Km(\E^3)$ is a generalized Kummer variety.

\subsection{Application:  Heegner-Drinfeld cycles on $\mathcal{E}^r$}
\label{sec:application}

This article was inspired by the theorem of Yun-Zhang \cite{YunZhangShtukasII} relating {\em Heegner-Drinfeld cycles} on stacks of shtukas, to the Taylor expansion of $L$-functions of automorphic forms.    We explain here the contact between higher modularity and the conjecture of Birch and Swinnerton-Dyer for elliptic curves over function fields.

As before, we let $X$ be a curve over $\F_q$ with function field $K$.   Let $f\from \E\to X$ be an elliptic fibration with generic fiber $E/K$.  Let $\ell$ be a prime not dividing $q$.  Recall \cite{TateBSD} that the inequality $\ord_{s=1}L(E/K,s)\geq \rank E(K)$ is known unconditionally, and that the following are equivalent:

\begin{enumerate}
\item The conjecture of Birch and Swinnerton-Dyer (BSD) holds for $E/K$;  i.e., $\ord_{s=1} L(E/K,s)=\rank E(K)$.  
\item The Tate conjecture holds for the surface $\E/\F_q$.  By this,  we mean that $H^2(\E_{\overline{\F}_q},\Q_\ell(1))^{\Fr_q=1}$ is spanned by classes of divisors.
\item The Tate-Shafarevich group $\Sha(E/K)$ is finite.
\end{enumerate}
We briefly review the connection between these statements.   Let $\NS(\E)$ be the N\'eron-Severi group of $\E$ (divisors modulo algebraic equivalence).  Define a decreasing two-step filtration $\Fil^i \NS(\E)$, with $\Fil^1\NS(\E)$ the subgroup of divisors whose intersection pairing with a fiber of $f$ is zero, and with $\Fil^2 \NS(\E)$ the subgroup generated by irreducible components of fibers.    Then $\NS(\E)/\Fil^1\NS(\E)\isom \Z$ and $\Fil^1\NS(\E)/\Fil^2\NS(\E)\isom E(K)$, the Mordell-Weil group.   The identity section splits off the first filtrant:  $\NS(\E)\isom \Fil^1\NS(\E)\oplus \Z$.

Meanwhile, the Leray spectral sequence for the composition $\E\to X\to \Spec \F_q$ degenerates on the second page.  As a result there is a decreasing two-step filtration $\Fil^i H^2(\E_{\overline{\F}_q},\Q_\ell(1))$ with graded pieces $H^2(X_{{\overline{\F}_q}},f_*\Q_\ell(1))\isom \Q_\ell$, $H^1(X_{\overline{\F}_q},R^1f_*\Q_\ell(1))$, $H^0(X_{{\overline{\F}_q}},R^2f_*\Q_\ell(1))$.   This filtration is preserved by the action of Frobenius $\Fr_q$.  Once again, the identity section splits off the first filtrant:  $H^2(\E_{\overline{\F}_q},\Q_\ell(1))\isom \Fil^1 H^2(\E_{\overline{\F}_q},\Q_\ell(1))\oplus \Q_\ell$.  

The cycle class map $\NS(\E)\otimes\Q_\ell \to H^2(\E_{\overline{\F}_q},\Q_\ell(1))$ respects the filtrations on either side, and induces an isomorphism on the first and third graded pieces.   On the second graded pieces, we get an injective map
\[ E(K)\otimes\Q_\ell \to V:=H^1(X_{\overline{\F}_q},R^1f_*\Q_\ell(1))\]
which lands in the Frobenius-fixed part $V^{\Fr_q=1}$.   The Tate conjecture for $\E/\F_q$ reduces to the statement that $E(K)\otimes\Q_\ell\to V^{\Fr_q=1}$ is an isomorphism.   Note that $V^{\Fr_q=1}$ is the Selmer group for $E/K$ (with $\Q_\ell$-coefficients), and the statement that $E(K)\otimes \Q_\ell\to V^{\Fr_q=1}$ is an isomorphism is equivalent to the statement that the $\ell$-primary part of $\Sha(E/K)$ is finite. 

The vector space $V$ is related to the $L$-function via:
\[ L(E/K,s) = \det (1-q^{1-s}\Fr_q \vert V ).\]
Thus $\ord_{s=1} L(E/K,s)$ is the dimension of the generalized $1$-eigenspace for $\Fr_q$ acting on $V$, which a priori may be larger than $V^{\Fr_q=1}$.   Thus we have inequalities:
\[ \rank E(K) \leq \dim V^{\Fr_q = 1} \leq \ord_{s=1} L(E/K,s) \]
If BSD holds for $E/K$, then the leftmost and rightmost quantities are equal, and consequently the Tate conjecture holds for $\E$, and as a bonus we also find that the generalized $1$-eigenspace for $\Fr_q$ on $V$ coincides with the eigenspace.  Conversely, if the Tate conjecture holds for $\E$, Tate shows that the generalized $1$-eigenspace for $\Fr_q$ on $V$ coincides with the eigenspace, and therefore BSD holds for $E/K$.  Note that the statement that the generalized $1$-eigenspace for $\Fr_q$ on $V$ coincides with the eigenspace is a consequence of the general conjecture that $\Fr_q$ acts semisimply on all $H^i(\E_{\overline{\F}_q},\Q_\ell)$;  this statement is sometimes packaged along with the Tate conjecture.

Let $N$ be the conductor of $E$, and let $\Sigma$ be the support of $N$.  Assume that $N$ is multiplicity-free.  Let $\Sigma_\infty\subset \Sigma$ be a subset.   For each $r\geq 0$ with the same parity as $\#\Sigma_\infty$, we have a space of $r$-legged shtukas $\lambda_r\from \Sht^r_G(\Gamma_0(N);\Sigma_\infty)\to X^r$.   (The map is really defined over all of $X^r$, without removing $\Sigma$.  In the following discussion we are working with the iterated shtukas as defined in \cite{YunZhangShtukasII}.  In particular $\Sht^r_G(\Gamma_0(N);\Sigma_\infty)$ is a smooth Deligne-Mumford stack.)   Abbreviate this as $\Sht^r$.  

Let $K'/K$ be a quadratic extension corresponding to a (possibly branched) double cover $X'\to X$.  We assume that $X'$ is geometrically connected over $\F_q$, and that the branch points of $X'\to X$ are disjoint from $\Sigma$.  Let us write $\Sigma=\Sigma_f\cup \Sigma_\infty$.  Assume the Heegner hypothesis:  all places in $\Sigma_f$ are split in $X'$, and all places in $\Sigma_\infty$ are inert in $X'$.   Let $\E'=\E\times_X X'$.  Then (under an assumption on the expected behavior of the cohomology of the stacks of shtukas, see \cite[Equation 1.9]{YunZhangShtukasII}), Yun-Zhang construct a {\em Heegner-Drinfeld class}
\[ HD_{E,r} \in ((V')^{\otimes r})^{\Fr_q=1} \]
where $V'$ is the analogue of $V$ for $E'\to X'$.    

The following theorem is claimed by Yun-Zhang (see \cite[Section 1.3.1]{YunZhangShtukasII}):

\begin{thm} \label{ThmGeneratesLine} Let $r_0\geq 0$ be the smallest integer $r$ such that $HD_{E,r}\neq 0$.  Then the Selmer rank $\dim (V')^{\Fr_q=1}$ is $r_0$, and $HD_{E,r_0}$ spans the line $\bigwedge^{r_0} (V')^{\Fr_q=1}$.  
\end{thm}


On the other hand,  the main result \cite[Corollary 1.4]{YunZhangShtukasII} applied to this situation reads:

\begin{equation}
\label{EqSelfIntersectionOfHD}
 (HD_{E,r},HD_{E,r}) \dot{=} L^{(r)}(E/K',1), 
 \end{equation}
where $(\;,\;)\from V\times V\to\Q_\ell$ is the natural symmetric pairing on $V$, and $\dot{=}$ means equality up to an explicit nonzero constant.   If one knew definiteness of the restriction of the pairing to the Heegner-Drinfeld classes, one could conclude that the $r_0$ of Theorem \ref{ThmGeneratesLine} is the analytic rank of $E/K'$, and therefore that Selmer rank = analytic rank.   (This is true unconditionally in the case of analytic rank 3, using the Hodge index theorem for surfaces.) The upshot is that Heegner-Drinfeld classes can be used to construct a basis for the determinant of the Selmer group of $E/K'$, but (unless the analytic rank is 1) one cannot use them to construct elements of the Mordell-Weil group.

Under the assumption that $E$ is $r$-modular, the Heegner-Drinfeld classes actually arise from algebraic cycles on $(\E')^r$, in the following sense.  Define the base change
\[ (\Sht^r)'=\Sht^r \times_{X^r} X'. \]
The {\em Heegner-Drinfeld cycle} is a codimension $r$ cycle with proper support:
\[ \mathcal{HD}_r\in Z^r_c((\Sht^r)',\Q). \]
It is obtained as the locus of rank 2 $X$-shtukas which arise from rank 1 $X'$-shtukas.  (To define $\mathcal{HD}_r$ one needs to make some auxiliary choices, see \cite[1.1.3]{YunZhangShtukasII}.  In particular one has to choose a place of $X'$ above each $v\in\Sigma_{f}$.)

Let $Z$ be an algebraic correspondence between $\Sht^r$ and $\E^r$ as in Definition \ref{DefnRModularity} (extended over all of $X^r$).   Such a correspondence induces a map on algebraic cycles:
\[ p_Z\from Z^r_c((\Sht^r)',\Q)\to Z^r((\E')^r,\Q). \]
Define
\[ \mathcal{HD}_{E,r}: = p_Z(\mathcal{HD}_r) \in Z^r((\E')^r,\Q).\]
Its cohomology class is an element
\[ \cl(\mathcal{HD}_{E,r})\in H^{2r}((\E')^r_{\overline{\F}_q},\Q_\ell(r)) \]
which is $\Fr_q$-invariant.  

Owing to the K\"unneth formula and the splitting of $H^2(\E'_{\overline{\F}_q},\Q_\ell(1))$ noted above, we have a surjective map
\begin{equation}
\label{EqMapFromHrToVr}
H^{2r}((\E')^r_{\overline{\F}_q},\Q_\ell(r)) \to (V')^{\otimes r}
\end{equation}
We expect that the image of $\cl(\mathcal{HD}_{E,r})$ in $(V')^{\otimes r}$ is $HD_{E,r}$.    The moral of our story is this:  suppose $r$ is the Selmer rank of $E/K'$.  Under the assumption of $r$-modularity, one still doesn't know that the Selmer group $(V')^{\Fr_q=1}$ is spanned by classes of algebraic cycles in $\E'$, but one does know that the determinant of the Selmer group is spanned by the class of an algebraic cycle in $(\E')^r$.

\subsection{Acknowledgments}

The authors would like to thank Noam Elkies, Mar\'ia In\'es de Frutos Fern\'andez, Colin Ingalls, Masato Kuwata, Will Sawin,  Ari Shnidman, John Voight, Ziquan Yang, Wei Zhang, and Xinwen Zhu for many helpful discussions.  We are also grateful to the anonymous referee,
whose meticulous review improved the paper in many ways.  A.L. would like to thank the Tutte Institute for its support of his external research.   J.W. was supported by NSF grant DMS-1902148 and an award from the Simons Foundation.   

\section{The stacks of shtukas for $\PGL_2$}
\label{sec:shtukareview}

In this section we review the construction of the stack of $r$-legged shtukas $\Sht^r_G(\Gamma_0(N);\Sigma_\infty)$ for the group $G=\PGL_2$.  We have tried to keep our notation consistent with \cite{YunZhangShtukasII}.   

From \S\ref{SectionShtukasOverP1} on we specialize to the case that the base curve $X$ is $\P^1_{\F_q}$.   In that case the shtuka space $\Sht^r(\Gamma;\Sigma_\infty)$ is amenable to concrete calculations.

\subsection{Vector bundles of rank 2, fractional twists, Atkin-Lehner automorphisms, and passage to $G=\PGL_2$}

Let $F$ be a field, and let $X/F$ be a smooth projective curve with fraction field $K$.  For each closed point $x\in \abs{X}$, let $K_x$ be the completion of $K$ at $x$.  We have the stack $\Bun_2$, which assigns to an $F$-scheme $S$ the groupoid of rank 2 vector bundles $\caF$ on $X\times S$.   If $D\in \Divi X$, we have the twist $\caF(D):=\caF\otimes \OO_X(D)$.  If $N\subset X$ is an effective divisor and $\caF$ is a rank 2 vector bundle on $X\times S$, then a $\Gamma_0(N)$-structure on $\caF$ is a rank 1 subbundle $\mathcal{L}_N\subset \caF\vert_{N\times S}$.


In the special case that $N$ is multiplicity-free with support $\Sigma$, a $\Gamma_0(N)$-structure may alternatively be described as a system of {\em fractional twists} $\caF(\half P)$ for each $P\in \Sigma$, which lies in between $\caF$ and $\caF(P)$:
\[ \caF\subset \caF(\half P) \subset \caF(P)\]
The quotients $\caF(\half P)/\caF$ and $\caF(P)/\caF(\half P)$ are required to be rank 1 vector bundles on $P\times S$.  Given the data of the $\caF(\half P)$, one can define for each $D\leq N$ the twist $\caF(\half D)$, defined as the subbundle of $\caF(D)$ generated by $\caF$ and the $\caF(\half P)$ for each $P\in \supp D$.  The line bundle $\L_N$ can be recovered as the kernel of $\caF\vert_{N\times S} \to \caF(\half N)\vert_{N\times S}$.  

 We can go further than this and define $\caF(D)$ for any $D$ belonging to $\half \Z \Sigma$, in such a way that $\caF(D+D')=\caF(D)(D')$ for all $D,D'\in\half\Z N$. Then whenever $D\leq D'$ for $D,D'\in \half \Z N$, we have an inclusion $\caF(D)\subset \caF(D')$ with cokernel supported on $D'-D$.  

Keeping the hypothesis that $N$ is multiplicity-free, we may write objects of $\Bun_2(\Gamma_0(N))$ as pairs $\caF^\dagger = (\caF,\set{\caF(D)})$, where $\set{\caF(D)}$ is a system of fractional twists for $D\in \half\Z \Sigma$.  For an object $\caF^\dagger$ of $\Bun_2(\Gamma_0(N))$ and $D'\in \half\Z \Sigma$ we can form the fractional twist $\caF^\dagger(D')$, simply by replacing each $\caF(D)$ with $\caF(D+D')$.  The resulting automorphism of $\Bun_2(\Gamma_0(N))$ is called the {\em Atkin-Lehner automorphism} $\AL(D')$.  These satisfy $\AL(D)\circ \AL(D')=\AL(D+D')$.\footnote{To be pedantic:  a certain diagram of stacks 2-commutes.}  
%

A rank 2 vector bundle $\caF$ on a curve is semistable if $2\deg \L\leq \deg \caF$ for all rank 1 subbundles $\L\subset \caF$.  The index of instability of a vector bundle is 
\[ \inst(\caF) = \max_{\caL}(2\deg \L - \deg \caF) \]
 where $\caL$ runs through rank 1 subbundles of $\caF$.  Note that $\inst(\caF)$ is invariant under (integral) twists by line bundles.   This index allows us to define a stratification on $\Bun_2(\Gamma_0(N))$.
 For $i\geq 0$, let 
 \[ \Bun_2(\Gamma_0(N))^{\leq i} \subset \Bun_2(\Gamma_0(N)) \]
 denote the open substack determined by the condition that $\inst \caF\leq i$ on all geometric points.
 Similarly define $\Bun_2(\Gamma_0(N))^{d,\leq i}$ as the degree $d$ component.  Then each $\Bun_2(\Gamma_0(N))^{d,\leq i}$ is an Artin stack of finite type.
%

As in \cite{YunZhangShtukas}, our main focus is on shtukas for the group
\[ G = \PGL_2.\]
Let $\Bun_G$ be the stack which assigns to an $F$-scheme $S$ the groupoid of $G$-torsors $\caF$ on $X\times S$.  Then 
\[ \Bun_G \isom \Bun_2/\Pic_X,\]
where the Picard stack $\Pic_X$ acts on $\Bun_2$ by twisting.   We may also define the stack $\Bun_G(\Gamma_0(N))$, classifing $G$-torsors $\caF$ over $X\times S$ together with a reduction of $\caF\vert_{N\times S}$ to the Borel subgroup of $G$.

The degree of an object of $\Bun_G$ is valued in $\Z/2\Z$, and we write $\Bun_G(\Gamma_0(N))^d$ ($d=0,1$) for the appropriate component.   Similarly as above we have $\Bun_G(\Gamma_0(N))^{d,\leq i}$, an Artin stack of finite type.  The Atkin-Lehner automorphisms described above descend to $\Bun_G(\Gamma_0(N))$:  for any subset $\Sigma_0\subset \Sigma$ supported on places appearing with multiplicity 1 in $N$, we have the Atkin-Lehner automorphism $\AL(\half \Sigma_0)$, which is in fact an involution.  For two such subsets $\Sigma_0,\Sigma_0'\subset \Sigma$, we have the relation $\AL(\half \Sigma_0)\circ \AL(\half\Sigma_0') = \AL(\half \Sigma_0'')$, where $\Sigma_0''\subset \Sigma$ is the symmetric difference of $\Sigma_0$ and $\Sigma_0'$.

\subsection{Stacks of $G$-shtukas}
Now let $\F_q$ be a finite field, and let $X/\F_q$ be a curve.   The definition of shtukas we use coincides with that appearing in Varshavsky \cite{Varshavsky} for general reductive groups $G$, except that our shtukas feature a set of ``archimedean places'' $\Sigma_\infty$, as in \cite{YunZhangShtukasII}.   

\begin{defn} \label{DefnShtuka} Let $S$ be an $\F_q$-scheme.  Let $G$ be $\GL_2$ or $\PGL_2$.  Suppose we are given the following data:
\begin{itemize}
\item An effective divisor $N\subset X$, with support $\Sigma$.
\item A subset $\Sigma_\infty \subset \Sigma$ of degree 1 places\footnote{The assumption that all places of $\Sigma_\infty$ are degree 1 is just for notational convenience.} appearing in $N$ with multiplicity 1.
\item An integer $r\geq 0$.
\item Morphisms $x_i\from S\to X\smallsetminus \Sigma$ for $i=1,\dots,r$.  Let $\gamma_i\subset X\times_{\F_q} S$ be the graph of $x_i$;  this is a divisor of $X\times_{\F_q} S$.
\end{itemize}  
A {\em $G$-shtuka with legs $x_1,\dots,x_r$, archimedean places $\Sigma_\infty$, and $\Gamma_0(N)$-structure} is an isomorphism of $G$-torsors with $\Gamma_0(N)$-structure:
\[ f\from \caF^\dagger\vert_{(X\times_{\F_q} S) \smallsetminus \bigcup_{i=1}^r \gamma_i } \isomto \Fr_S^*\caF^\dagger(\half \Sigma_\infty)\vert_{(X\times_{\F_q} S) \smallsetminus \bigcup_{i=1}^r \gamma_i }. \]
Here the ``$\Fr_S$'' is shorthand for the endomorphism $\id_X\times \Fr_q$ on $X\times S$, where $\Fr_q\from S\to S$ is the $q$th power Frobenius morphism.
\end{defn}  
One can measure the failure of $f$ to be an isomorphism along $\gamma_1,\dots,\gamma_r$ by means of an algebraic representation $W$ of the $r$-fold product $\hat{G}^r$, where $\hat{G}$ is the Langlands dual group.   We write $\Sht^W_G(\Gamma_0(N);\Sigma_\infty)$ for the moduli stack of $G$-shtukas with $\Gamma_0(N)$-structure and archimedean places $\Sigma_\infty$ which are bounded by $W$.


We explain here in explicit terms the condition that a shtuka as above is ``bounded by $W$'', in the case that $G=\GL_2$ (so that $\hat{G}=\GL_2$ as well) and $W$ is the representation
\[ W_\mu = \mathrm{std}^{\mu_1} \boxtimes \cdots \boxtimes \mathrm{std}^{\mu_r} \]
of $\hat{G}^r$, where $\mathrm{std}$ is the standard representation of $\GL_2$ and $\mu=(\mu_1,\dots,\mu_r)\in\set{1,-1}^r$ is a vector of signs (with $\mathrm{std}^{-1}$ meaning the dual of $\mathrm{std}$).  
Consider the following data:
\begin{itemize}
\item Vector bundles $\caG_0,\caG_1,\cdots,\caG_r$ on $X\times_{\F_q} S$,
\item For each $i=1,\dots,r$, an isomorphism $f_i\from \caG_{i-1}\vert_{(X\times_{\F_q} S)\backslash \gamma_i} \to \caG_i\vert_{(X\times_{\F_q} S)\backslash \gamma_i}$.   If $\mu_i=1$, then $f_i$ is required to extend to a morphism of vector bundles $\caG_{i-1}\to \caG_i$ whose cokernel is length 1 supported on $\gamma_i$.  If $\mu_i=-1$, then $f_i^{-1}$ is required to have this property,
\end{itemize}
An isomorphism between $\caG_0$ and $\caG_r$ away from $\bigcup_i^r \gamma_i$ is bounded by $W_\mu$ if it factors into $f_r\circ\cdots \circ f_1$, where $f_1,\dots,f_r$ are as above.   Comparing degrees of vector bundles, we find that such an isomorphism can exist only if $\sum_i \mu_i = \deg \caG_r-\deg\caG_0$.  

From now on we fix
\[ G = \PGL_2,\] 
so that $\hat{G}=\SL_2$.  Note that $\mathrm{std}$ is self-dual when considered as a representation of $\hat{G}$.  Let us simply write
\[ \Sht_G^r(\Gamma_0(N);\Sigma_\infty) \]
for the moduli stack of $G$-shtukas with $\Gamma_0(N)$-structure and archimedean places $\Sigma_\infty$ which are bounded by
\[ W_r: = \mathrm{std}^{\boxtimes r}. \]
We have an isomorphism
\[ \Sht_G^r(\Gamma_0(N);\Sigma_\infty) \isom \Sht_{\GL_2}^{W_\mu}(\Gamma_0(N);\Sigma_\infty)/\Pic_X(\F_q), \]
 where $\mu$ is any vector of signs satisfying $\sum_i \mu_i=\#\Sigma_\infty$.  Therefore a necessary condition for $\Sht^r_G(\Gamma_0(N);\Sigma_\infty)$ to be nonempty is the congruence
\[ r\equiv \#\Sigma_\infty\pmod{2}.\]

The stack $\Sht_G^r(\Gamma_0(N);\Sigma_\infty)$ is a Deligne-Mumford stack, and the morphism
\[ \lambda^r\from \Sht_G^r(\Gamma_0(N);\Sigma_\infty) \to (X\smallsetminus \Sigma)^r \]
sending a shtuka to its legs $x_1,\dots,x_r$ is generically smooth of relative dimension $r$.  

\begin{rmk} Keeping the intermediate $\caG_i$ as above, one obtains the notion of an iterated shtuka.  The moduli space of iterated $\GL_2$-shtukas bounded by $\mu=(1,-1)$ was originally considered by Drinfeld \cite{DrinfeldLanglandsConjectureForGL2}.  If we had worked with iterated shtukas, the morphism $\lambda^r$ would be smooth over $(X\smallsetminus \Sigma)^r$.  The spaces of iterated and non-iterated shtukas coincide over the locus in $X^r$ where the legs are pairwise distinct.  
\end{rmk}


\begin{ex}[Relation to Drinfeld modular curves]
In the special case that $r=1$ and $\Sigma_\infty=\set{\infty}$ is a singleton, the stack $\Sht^1_G(\Gamma_0(N);\infty)$ is related to a Drinfeld modular curve.  Namely,  decompose $N$ as $N=N_f+(\infty)$, and let $\mathrm{DrMod}(\Gamma_0(N_f);\infty)$ classify Drinfeld $A$-modules, where $A$ is the ring of functions on $X$ regular away from $\infty$, equipped with a $\Gamma_0(N_f)$-structure.     The precise relationship is that there is an isomorphism
\[ \Sht^1_G(\Gamma_0(N);\infty) \isom 
\mathrm{DrMod}(\Gamma_0(N_f);\infty)/\Pic^\circ_X(\F_q).\]
\end{ex}

For $r>1$ the stack $\Sht_G^r(\Gamma_0(N);\Sigma_\infty)$ is generally not of finite type.  To produce finite-type stacks, we need to pass to truncations.   There is a map 
\[ \Sht_G^r(\Gamma_0(N);\Sigma_\infty) \to \Bun_G(\Gamma_0(N))\]
sending a shtuka as in Definition \ref{DefnShtuka} to its $G$-torsor $\caF$.    For each $i\geq 0$, let $\Sht_G^r(\Gamma_0(N);\Sigma_\infty)^{\leq i}$ be the pullback under this map of $\Bun_G(\Gamma_0(N))^{\leq i}$.  Then $\Sht_G^r(\Gamma_0(N);\Sigma_\infty)^{\leq i}$ is of finite type.

\subsection{Cohomology of stacks of shtukas}
\label{sec:CohoShtukas}

Here we prove Proposition \ref{PropCohoOfShtukaSpaceIntro}, which states that irreducible Langlands parameters for $G$ appear as subspaces of the cohomology of stacks of shtukas.  We will need the main results of  \cite{Xue} and \cite{XueFiniteness}, which we now review.

Recall our conventions on generic points:  $\eta$ (resp., $\eta_r$) is the generic point of $X$ (resp., $X^r$), and $\overline{\eta}_r\to \eta_r$ is an algebraic closure lying over $\overline{\eta}^r\to \eta^r$.   Let $\Sigma$ be the support of $N$, and let $U=X\smallsetminus \Sigma$.

The truncated stack of shtukas $\Sht_G^r(\Gamma_0(N);\Sigma_\infty)^{\leq i}$ is of finite type over $U^r$, so it makes sense to define the cohomology with compact supports 
\[H^r_c(\Sht_G^r(\Gamma_0(N);\Sigma_\infty)^{\leq i}_{\overline{\eta}_r},\overline{\Q}_\ell),\]
a finite-dimensional $\overline{\Q}_\ell$-vector space admitting an action of $\pi_1(U^r,\overline{\eta}_r)$.  Define
\[ H^r_c(\Sht_G^r(\Gamma_0(N);\Sigma_\infty)_{\overline{\eta}_r},\overline{\Q}_\ell):=\varinjlim_{i\geq 0} H^r_c(\Sht_G^r(\Gamma_0(N);\Sigma_\infty)^{\leq i}_{\overline{\eta}_r},\overline{\Q}_\ell),\]
and abbreviate
\[ H^r_c:=H^r_c(\Sht_G^r(\Gamma_0(N);\Sigma_\infty)_{\overline{\eta}_r},\overline{\Q}_\ell). \]
Then $H^r_c$ admits an action of the Hecke algebra $T_0(N)$ which commutes with the action of $\pi_1(U^r,\overline{\eta}_r)$.  Let $(H^r_c)^{\text{cusp}}\subset H^r_c$ be the cuspidal part.  Assuming that $\Sigma_\infty=\emptyset$, a theorem of Xue  \cite[Theorem 5.0.1]{Xue} shows that $(H^r_c)^{\text{cusp}}$ is finite-dimensional, and then Drinfeld's lemma extends the action of $\pi_1(U^r,\overline{\eta}_r)$ on $(H^r_c)^{\text{cusp}}$ to an action of $\pi_1(U,\overline{\eta})^r$.

\begin{prop}  
\label{PropCohoOfShtukaSpace} Assume that $\#\Sigma_\infty\leq 1$.  When $r$ is odd,  assume that Xue's theorem \cite[Theorem 5.0.1]{Xue} extends to the case of $(r,\Sigma_\infty)$.  Then as a representation of $T_0(N)\times \pi_1(U,\overline{\eta})^r$, there is a decomposition
\[ (H^r_c)^{\mathrm{cusp}} = \bigoplus_\pi \pi^{\Gamma_0(N)} \otimes \sigma^{\boxtimes r}_\pi, \]
where $\pi$ runs over cuspidal representations of $G$ containing a $\Gamma_0(N)$-fixed vector, and $\sigma_\pi$ is the 2-dimensional representation of $\pi_1(U,\overline{\eta})$ corresponding to $\pi$ under the global Langlands correspondence.
\end{prop}

\begin{rmk} When $r=2$ and $\Sigma_\infty=\emptyset$, \eqref{PropCohoOfShtukaSpace} is a special case of a theorem of Drinfeld \cite{DrinfeldLanglandsConjectureForGL2}. 
\end{rmk}

\begin{proof}  The formalism of excursion operators developed in \cite{VincentLafforgueChtoucas} furnishes a decomposition of $T_0(N)\times \pi_1(U,\overline{\eta})^r$-modules $(H^r_c)^{\mathrm{cusp}}=\bigoplus_{\sigma} (H^r_c)^{\mathrm{cusp}}_\sigma$ indexed by isomorphism classes of Langlands parameters $\sigma\from \pi_1(U,\overline{\eta}) \to \GL_2\overline{\Q}_\ell$. 

Let $\sigma$ be an irreducible Langlands parameter for $G$.   In the case that $r$ is even and $\Sigma_\infty=\emptyset$,  Corollaire 2.1 to  \cite[Proposition 1.2]{LafforgueZhu} identifies $(H^r_c)^{\mathrm{cusp}}_{\sigma}$ with $\pi^{\Gamma_0(N)}\otimes \sigma^{\boxtimes r}$, where $\pi$ corresponds to $\sigma$ under the Langlands correspondence.   (We remark here that our $(H^r_c)^{\mathrm{cusp}}$ is the space denoted $H^{\text{cusp}}_{I,W}$ in \cite{LafforgueZhu}, where $I$ is a set of cardinality $r$ and $W=\mathrm{std}^{\boxtimes I}$ is the external tensor product of $I$ copies of the standard representation std of $\hat{G}=\SL_2$.)  The proposition follows.


In the case that $r$ is odd and $\Sigma_\infty=\set{\infty}$ is a singleton,
the same technique as in \cite{LafforgueZhu} applies (conditionally on the extension of Xue's theorem to this case) to give a finite-dimensional $T_0(N)$-module $A_\sigma$ such that for all odd $r$, there is an isomorphism of $T_0(N)\times\pi_1(U,\overline{\eta})^r$-modules:
\[ (H^r_c)^{\mathrm{cusp}}_{\sigma}\isom A_\sigma \otimes \sigma^{\boxtimes r}. \]
Setting $r=1$, we have that $(H^1_c)^{\mathrm{cusp}}_{\sigma}$ is the cohomology of a Drinfeld modular curve, namely the curve which parametrizes Drinfeld modules over the ring $H^0(X\backslash \set{\infty},\OO_X)$.   
The behavior of the cohomology of Drinfeld modular curves was known to Drinfeld:  $A_\sigma=\pi^{\Gamma_0(N)}$.  
We proceed as above.
 \end{proof}
 
We can now complete the proof of Proposition \ref{PropCohoOfShtukaSpaceIntro}.  Suppose $\sigma$ is an irreducible Langlands parameter of conductor $N$.  Proposition \ref{PropCohoOfShtukaSpace} shows that $\sigma^{\boxtimes r}$ appears as a summand of $(H^r_c)^{\mathrm{cusp}}$.  Therefore it is a subspace of $H^r_c$.

\subsection{Vector bundles with level structures on $\mathbf{P}^1$}
\label{SectionShtukasOverP1}
At this point we specialize our study of shtuka spaces to the case that the base curve $X$ is the projective line.  We make heavy use of the fact that the degree $d$ component $\Bun_2^d$ admits a dense open substack containing only one isomorphism class, namely $\OO(d/2)^{\oplus 2}$ or $\OO((d-1)/2)\oplus \OO((d+1)/2)$ as $d$ is even or odd, respectively;  these are exactly the objects with instability index $\leq 1$.   As a result, it becomes possible to give explicit equations for dense open substacks of the moduli stacks of shtukas.  

The following lemmas, concerning $\Bun_G(\Gamma_0(N))$ for divisors $N$ of degree 3 and 4, exemplify the simplifications that will occur.

\begin{lm}\label{LemmaDescriptionOfBunGGamma0deg3} Let $F$ be a field, and let $X=\mathbf{P}^1_F$.  Let $N$ be an effective divisor of degree 3.  There is a dense substack $\Bun_G(\Gamma_0(N))^{\omega}$ of $\Bun_G(\Gamma_0(N))$, such that each connected component $\Bun_G(\Gamma_0(N))^{d,\omega}$ ($d=0,1$) is isomorphic to a single point $\Spec F$.
\end{lm}

\begin{proof}  We given an explicit description of $\Bun_G(\Gamma_0(N))^{\omega}$ in the case that $N=P_1+P_2+P_3$ where each $P_i$ is $F$-rational, the other cases being similar.   It is the locus of $\caF^\dagger$ where the instability of each fractional twist $\caF(\half D)$ is $\leq 1$, where $D$ runs over divisors supported on $N_1,N_2,N_3$.  One sees that each of its components $\Bun_G(\Gamma_0(N))^{d,\omega}$ is a single point.  In degree 0, for instance, let $\caF^\dagger\in \Bun_G^0(\Gamma_0(N))(F)$ be the object with $\caF$ the trivial $G$-torsor.  The level structure, represented by a triple of pairwise distinct elements of $\P^1(F)$, determines an object of $\Bun_G(\Gamma_0(N))^{\omega}$ if and only if the three points in the triple are pairwise distinct.  (Translated into vector bundles, this is the observation that elementary modifications of $\OO^{\oplus 2}$ at points $P_1,P_2$ along the lines $L_1,L_2\subset F^2$ produce $\OO(1)\oplus\OO(1)$ generically, but will produce $\OO(2)\oplus\OO$ if $L_1=L_2$.)  As $G$ acts simply transitively on such triples, we find an isomorphism $\Bun_G(\Gamma_0(N))^{0,\omega}\isom \Spec F$.

\end{proof}

\begin{lm}\label{LemmaDescriptionOfBunGGamma0}
Let $F$ be a field, and let $X=\PP^1_F$.  
\begin{enumerate}
\item Let $N=P_1+P_2+P_3+P_4$ be a multiplicity-free divisor of degree $4$.  There exists a dense open substack $\Bun_G(\Gamma_0(N))^{\omega}$, stable under the Atkin-Lehner operators, each of whose connected components is isomorphic to $\PP^1_F\smallsetminus N$.  With respect to these isomorphisms, the Atkin-Lehner involution $\AL(\half (P_1+P_2)))$ is the unique involution which exchanges $P_1$ with $P_2$ and $P_3$ with $P_4$.
\item Let $N=P_1+P_2+2P_3$, with $P_1,P_2,P_3$ pairwise distinct.  There exists a dense open substack $\Bun_G(\Gamma_0(N))^{\omega}$, stable under the Atkin-Lehner operators (supported only on $P_1,P_2$), each of whose connected components is isomorphic to $\mathbf{A}^1_F$.  
\end{enumerate}
\end{lm}

\begin{proof}  (1) Let $\Bun_G(\Gamma_0(N))^{\omega}$ be the locus of $\caF^\dagger$ where (once again) all Atkin-Lehner twists $\caF(\half \Sigma)$ have instability index $\leq 1$.  This time, the $\Gamma_0(N)$-level structures $\caF^\dagger$ on the trivial $G$-torsor correspond to quadruples of points of $\P^1$.  One computes that $\caF^\dagger$ lies in $\Bun_G(\Gamma_0(N))^{\omega}$ if and only if the quadruple is pairwise distinct and is not in the $G$-orbit of $(P_1,P_2,P_3,P_4)$.  The set of $G$-orbits of such quadruples is $\P^1_F\smallsetminus N$.  (The claim about Atkin-Lehner involutions is not used in the sequel and we omit its proof.)

(2) Let $\Bun_2(\Gamma_0(N))^{\omega}$ be the preimage of the substack described in Lemma \ref{LemmaDescriptionOfBunGGamma0deg3}.    The parameter space for $\Gamma_0(2P_3)$-structures lying over a given $\Gamma_0(P_3)$-structure is described by the fibers of $\Res^{2P_3}_{P_3} \PP^1_F\to \PP^1_F$, which are all $\mathbf{A}^1_F$.

\end{proof}

\subsection{Explicit presentations of shtuka spaces with $\Gamma_0(N)$-structure}

The following theorem and its proof give an explicit presentation of the space of shtukas with $\Gamma_0(N)$-structure, where $N$ is a divisor of degree 3 or 4.  
It is a special case of the main results of \cite{Fernandez}.  Define $\Sht^r_G(\Gamma_0(N);\Sigma_\infty)^{\omega}$ to be the dense open substack where the underlying $G$-bundle of the shtuka lies in the substack $\Bun_G(\Gamma_0(N))^{\omega}$ described by Lemmas \ref{LemmaDescriptionOfBunGGamma0deg3} and \ref{LemmaDescriptionOfBunGGamma0}.  

\begin{thm} \label{ThmPresentationOfShtukas} Let $X=\mathbf{P}^1_{\F_q}$,  let $r\geq 1$ be an integer, and let $\Sigma_\infty\subset \abs{X}$ be a finite set of degree 1 points satisfying $\#\Sigma_\infty\equiv r\pmod{2}$.   Let $\eta_r= \Spec \F_q(t_1,\dots,t_r)$ be the generic point of $X^r$.  
\begin{enumerate}
\item Let $N$ be an effective divisor of degree 3, such that each point of $\Sigma_\infty$ has multiplicity 1 in $N$.    The stack $\Sht^r_G(\Gamma_0(N);\Sigma_\infty)^{\omega}_{\eta_r}$ is a quasiprojective variety, each of whose connected components is rational; more precisely, there is a birational morphism from each component to $(\P^1_{\eta_r})^r$.  
\item Let $N$ be a multiplicity-free effective divisor of degree 4 whose support contains $\Sigma_\infty$.  The stack $\Sht^r_G(\Gamma_0(N);\Sigma_\infty)^{\omega}_{\eta_r}$ is a quasiprojective variety.  Each of its connected components is birational to a hypersurface in $(\P^1_{\eta_r})^{r+1}$ of degree $(2,2,\dots,2,q+1)$.  
\item Let $N=P_1+P_2+2P_3$ be an effective divisor of degree 4, such that $P_1,P_2,P_3$ are pairwise distinct and $\Sigma_\infty\subset \set{P_1,P_2}$. The stack $\Sht^r_G(\Gamma_0(N);\Sigma_\infty)^{\omega}_{\eta_r}$ is a quasiprojective variety.  Each of its connected components is birational to a hypersurface in $(\P^1_{\eta_r})^{r+1}$ of degree $(2,2,\dots,2,q)$.  
\end{enumerate}
\end{thm}

\begin{rmk}  There are no cusp forms on $\PGL_2$ of conductor $N$ if $\deg N=3$, whereas if $\deg N=4$ and $N$ is multiplicity-free, then the dimension of the space of cusp forms of conductor $N$ is exactly $q$.  (See \cite[\S 7]{DeligneFlicker} for these results and related formulas.) So it is not surprising that non-rational varieties appear when passing from degree 3 to degree 4.
\end{rmk}

\begin{rmk} It is possible to extend Theorem \ref{ThmPresentationOfShtukas} to include the cases $N=2P_1+2P_2$, $N=P_1+3P_2$, and $N=4P$.   We refer the reader to \cite{Fernandez} for details.
\end{rmk}

\begin{proof} 
(1) By Lemma \ref{LemmaDescriptionOfBunGGamma0deg3}, $\Bun_G(\Gamma_0(N))^{d,\omega}$ is a point, corresponding to an object $\caF^\dagger$ of $\Bun_G(\Gamma_0(N))(\F_q)$.  Thus for a test scheme $S$, the shtukas in $\Sht_G^r(\Gamma_0(N);\Sigma_\infty)^{d,\omega}(S)$ classify isomorphisms of $G$-torsors with $\Gamma_0(N)$-structure:
\[ f\from \caF^\dagger\vert_{(X\times_{\F_q} S) \backslash \cup_i \gamma_i}  \isomto \caF^\dagger(\half \Sigma_\infty) \vert_{(X\times_{\F_q} S) \backslash \cup_i \gamma_i}
\]
which are bounded by $W_r$.  

With Frobenius out of the picture, all that remains is linear algebra.   Let $\caF_1^\dagger$ and $\caF_2^\dagger$ be any two lifts of $\caF$ and $\caF(\half \Sigma_\infty)$ to $\Bun_2(\F_q)$, such that $\deg \caF_2=\deg \caF_1+r$.  (This is possible due to the parity condition on $r$.)  Note that the dimension of $\Hom(\caF_1^\dagger,\caF_2^\dagger)$ (meaning, morphisms of vector bundles, not necessarily injective, which preserve level structures) is $2r+1$.  
The stack of shtukas $\Sht_G^r(\Gamma_0(N);\Sigma_\infty)^{d,\omega}$ is isomorphic to the quotient by $\mathbf{G}_m$ of the moduli stack of morphisms of vector bundles with $\Gamma_0(N)$-structure $\caF_1^\dagger\to \caF_2^\dagger$ which are bounded by $\mathrm{std}^{\boxtimes r}$ at the legs.  This is a quasi-projective variety;  indeed it is the locally closed subvariety of 
\[ (\P^1)^r\times \P\Hom(\caF_1^\dagger,\caF_2^\dagger)\isom (\P^1)^r\times \P^{2r}, \]
consisting of data $(x_1,\dots,x_r,f)$, where the divisor of $\det f$ is $\gamma_1+\dots+\gamma_r$.  

For each $i=1,\dots,r$, the local behavior of $f$ at $\gamma_i$ gives an $S$-point of the affine Grassmannian $\mathrm{Gr}_G$: this is the period morphism at the $i$th leg.  If the legs are pairwise disjoint, all the period morphisms land in the minuscule Schubert cell $\mathrm{Gr}_G^{\mathrm{std}}=\P^1$.  We claim that over the generic point $\eta_r$, the product of the period morphisms
\[\Pi\from \Sht^r_G(\Gamma_0(N);\Sigma_\infty)^{d,\omega}_{\eta_r}\to (\P^1_{\eta_r})^r \]
is a birational map.  Indeed, let $x_1,\dots,x_r$ be generic legs, and consider the fiber of $\Pi$ over $(u_1,\dots,u_r)\in (\P^1)^r$.  This is the subspace of $f\in \P\Hom(\caF_1^\dagger,\caF_2^\dagger)\isom \P^{2r}$ such that $f$ is generically an isomorphism, and such that $f(x_i)$ has rank 1, with kernel $u_i$.  This is $2r$ linear conditions; for generic $(u_1,\dots,u_r)$ there is a unique solution.

A computation involving minors shows that the rational map $\Pi^{-1}\from (\P^1_{\eta_r})^r\to \P_{\eta_r}^{2r}$ has degree $(2,\dots,2)$.

(2) Now suppose $N$ is multiplicity-free of degree 4.  Let $N=N_0+P$ for a point $P$ disjoint from $N_0$.  The quasi-projective variety $\Sht^r_G(\Gamma_0(N);\Sigma_\infty)_{\eta_r}^{d,\omega}$ can be modeled on a subvariety of $\P^{2r}\times\P^1$, where the $\P^{2r}$ is as in part (1) and the $\P^1$ records the additional level structure at $P$.  Namely, it is the space of pairs $(f,v)$, where the divisor of $\det f$ is $\gamma_1+\dots+\gamma_r$, and $f$ preserves the level structure at $P$ in the sense that $f(P)(v)=\Fr_q(v)$.  Part (1) shows that this is birational via $\Pi$ to the subvariety $Z$ of $(\P^1)^r\times \P^1$ consisting of pairs $(u,v)$, where $f=\Pi^{-1}(u)$ satisfies $f(P)(v)=\Fr_q(v)$.  We saw that $\Pi^{-1}$ has degree $(2,\dots,2)$, which is to say that $f(t)$ has coordinates which are degree $(2,\dots,2)$ in $u=(u_1,\dots,u_r)$.  Finally we observe that the equation $f(P)(v)=\Fr_q(v)$ has degree $q+1$ in $v$. 

(3) Finally, suppose $N=P_1+P_2+2P_3=N_0+P_3$.  The calculation is similar to (2), except that the additional level structure at $P_3$ has parameter space $\mathbf{A}^1$, this being the fiber of $\Res^{2P_3}_{P_3} \P^1\to \P^1$.  The operator $f(P)$ is an automorphism of this $\mathbf{A}^1$; i.e., an affine transformation.  Thus $f(P)(v)=\Fr_q(v)$ has degree $q$ in $v$.
\end{proof}

%

\begin{rmk}  Consider the case $r=2$.   Let $N$ be a divisor of degree 4 as in part (2) or (3) of Theorem \ref{ThmPresentationOfShtukas}.  Let $\eta_2\isom \Spec \F_q(t_1,t_2)$ be the generic point of $X^2$.  The theorem shows that each connected component of $\Sht_G^2(\Gamma_0(N);\Sigma_\infty)^{\omega}_{\eta_2}$ is birational to a hypersurface in $(\P^1_{\eta_2})^3$ of degree $(2,2,q)$ or $(2,2,q+1)$.  Since $(2,2)$-curves in $(\P^1)^2$ are of genus 1, there is a desingularization, call it $Z$, which is a genus 1 fibration over $\P^1_{\eta_2}$.    In fact $Z\to\P^1_{\eta_2}$ admits a section, so that $Z$ is an elliptic surface.  The nature of the singular fibers of $Z\to \P^1_{\eta_2}$ was studied in \cite{Fernandez}.  

Specializing even further to the case $q=2$ and $N=(0)+(1)+2(\infty)$, one obtains a hypersurface $Z$ defined over $\F_2(t_1,t_2)$ of degree $(2,2,2)$, which is a K3 elliptic surface.  This project originated in 2019, when the second author sent the equation for $Z$ to Noam Elkies, who (a) recognized $Z$ as the universal K3 elliptic surface with a 6-torsion section, and (b) computed a finite morphism from $Z$ to the Kummer surface $\Km(\E_{\eta_2}^2)$, where $\E\to\P^1_{\F_2}$ is an elliptic fibration of conductor $N$.  (More precisely, $\E\to\P^1_{\F_2}$ is the elliptic fibration with $\I_2,\I_6,\IV$ fibers at $0,1,\infty$, with Mordell-Weil group $\Z/6\Z$.)   This $\E$ was the first example of a elliptic fibration we knew to be 2-modular.  
\end{rmk}

\subsection{The coincidence map}

Theorem \ref{ThmPresentationOfShtukas} allows us to derive equations for shtuka spaces of level $\Gamma_0(N)$, where $N$
is a degree 4 effective divisor of $\P^1_{\F_q}$ supported on at least one place $\infty$ of multiplicity 1.   In an early phase of this line of research, we dealt directly with those equations in an attempt to prove 2-modularity of some elliptic fibrations of conductor $N$.   To our surprise, a change of variables turned the defining equation for $\Sht^2_G(\Gamma_0(N))^{d,\omega}_{\eta_2}$ into one in which we could recognize an equation for the Drinfeld modular curve $\Sht^1_G(\Gamma_0(N);\infty)$.  Eventually we found a generalization of this phenomenon, in the form of the coincidence map described in the introduction.

Let $r\geq 2$, and let $\Sigma_\infty,\Sigma'_\infty\subset \Supp N$ be sets of degree 1 places with multiplicity 1 in $N$, such that $\#\Sigma_\infty\equiv r\pmod{2}$ and $\#\Sigma_\infty'\equiv 1\pmod{2}$.  (Under our hypotheses on $N$, such sets exist for any $r$.)  Let $\Sigma_\infty''$ be the symmetric difference of these sets, so that $\#\Sigma_\infty''\equiv r+1\pmod{2}$.  Our goal is to give a factorization of the variety of $r$-legged shtukas into a cartesian square:

\begin{equation}
\label{EqBigCoincidenceDiagram}
\xymatrix{
\Sht_G^r(\Gamma_0(N);\Sigma_\infty) \ar@{-->}[rr]^{c\times\lambda_r} \ar@{-->}[d]_{\beta} &&   \Sht^1_G(\Gamma_0(N);\Sigma_\infty') \times U^r  \ar[d]^{\lambda_1\times \id}  \\
\Coinc^{r+1}_G(\Gamma_0(N);\Sigma_\infty'') \ar[rr]_{\lambda_{r+1}} && U^{r+1}
}
\end{equation}

We have called the variety $\Coinc^{r+1}(\Gamma_0(N);\Sigma''_\infty)$ the moduli space of coincidences, and $c$ the coincidence map.  The surprising thing (to us) about this state of affairs was that the family of varieties $\Coinc^{r+1}(\Gamma_0(N);\Sigma''_\infty)\to U^{r+1}$ was symmetric not just in the first $r$ copies of $U$ as we would expect, but also in the final copy.

First we define the variety of coincidences.  The ``coincidences'' classified by these spaces are modifications of vector bundles, but Frobenius is not involved and they may be defined in any characteristic.  
Therefore let $F$ be any field, and let $N$ be an effective divisor on $\P^1$ of degree 4.   Let $\Sigma_\infty\subset \Supp N$ be a set of places appearing with multiplicity 1 in $N$, such that $\#\Sigma_\infty\equiv r\pmod{2}$.  

Let $\caH^{\dagger}\in \Bun_G(\Gamma_0(N))(F)$ be the following ``reducible object'' of instability index 2.  This will be the image of an object of $\Bun_2(\Gamma_0(N))(F)$ represented by the data $\caH=\OO(2)\oplus \OO$ together with the $\Gamma_0(N)$-level structure which is just $\OO_N$ embedded into $\caH\vert_N$ along the second factor.  Note that $\Aut\caH^{\dagger}$ (considered as a group scheme) is $\mathbf{T}$, the diagonal torus in $G$.   Thus $\Bun_G(\Gamma_0(N))^0$ contains a locally closed substack isomorphic to $B\mathbf{T}$, whose $S$-points classify those $\caH_0^\dagger\in \Bun_G(\Gamma_0(N))$ which are locally isomorphic to $\caH^\dagger$.

We need $\caH^{\dagger}$ for the following property:

\begin{lm}  \label{LemmaMapFromH}  Consider the stack $\mathcal{U}$ whose $S$-points classify tuples $(\caH^\dagger_0,\caF^\dagger,\tau,h)$, where $\caH^\dagger_0$ is an object of $\Bun_G(\Gamma_0(N))^0(S)$ which is locally isomorphic to $\caH^\dagger$, where $\caF^\dagger\in \Bun_G(\Gamma_0(N))^{1,\omega}(S)$, where $\tau\in\P^1(S)$, and where 
\[h\from\caH_0^\dagger\vert_{\P^1\smallsetminus\gamma_{\tau}}\isom \caF^\dagger\vert_{\P^1\smallsetminus\gamma_{\tau}}\]
is an isomorphism bounded by $\mathrm{std}$ on $\tau$.  Then the projection $\mathcal{U}\mapsto \Bun_G(\Gamma_0(N))^{1,\omega}$ is an isomorphism.  That is, $\caF^\dagger$ determines the data $(\tau,\caH^\dagger,h)$.
\end{lm}

\begin{proof}  A restatement of the lemma is that, given $\caF^\dagger$, the space classifying isomorphisms $h$ between $\caH^\dagger$ and $\caF^\dagger$ away from an unspecified leg and bounded by $\mathrm{std}$ there is a $\mathbf{T}$-torsor.  The space of such $h$ can be modeled in an ambient space $\P\Hom(\OO(2)\oplus \OO, \OO(2)\oplus\OO(1))\isom \P^5$.  The fact that $h$ preserves $\Gamma_0(N)$-level structures (recall that $\deg N = 4$) means that it belongs to a $\P^1$, but then two of the points on the $\P^1$ correspond to $h$ with $\det h=0$.  The complement of these points is a $\mathbf{T}$-torsor. 
\end{proof}

\begin{defn}[The stack of $r$-legged coincidences]\label{DefnGCoincidences} A {\em $G$-coincidence} with legs $x_i\from S\to \P^1$ (for $i=1,\dots,r$), archimedean places $\Sigma_\infty$, and $\Gamma_0(N)$-structure is an isomorphism of $G$-torsors with $\Gamma_0(N)$-structure:
\[ f\from \caH_1^{\dagger}\vert_{\PP^1_S\smallsetminus \bigcup_i \gamma_i } \isomto \caH_2^{\dagger}\vert_{\PP^1_S\smallsetminus \bigcup_i \gamma_i} \]
such that locally $\caH_1^\dagger \isom \caH^\dagger$ and $\caH_2^\dagger \isom \caH^\dagger(\half \Sigma_\infty)$, and such that $f$ is bounded by $\mathrm{std}^{\boxtimes r}$ at the legs.  Let $\Coinc^{r}_G(\Gamma_0(N);\Sigma_\infty)$ be the stack classifying such $G$-coincidences, and let
\[ \Coinc^{r}_G(\Gamma_0(N);\Sigma_\infty)\to (\P^1)^r \]
be the morphism sending a $G$-coincidence to its legs.
\end{defn}

In other words, $\Coinc^{r}_G(\Gamma_0(N);\Sigma_\infty)$ classifies such isomorphisms involving $\caH^\dagger$ and $\caH^\dagger(\half \Sigma_\infty)$, modulo the action of $\mathbf{T}\times\mathbf{T}$.

Now we can construct the coincidence map $c$.  Return to the setting in the overview, so that the base field is $\F_q$ and we have $\Sigma_\infty,\Sigma_\infty',\Sigma_\infty''\subset \Sigma$, whose parities are $r,1,r+1$, respectively.  Consider the stack $\mathcal{V}$ whose $S$-points classifying the following data:
\begin{itemize}
\item $\caF^\dagger$ is an object of $\Bun_G(\Gamma_0(N))^{0,\omega}(S)$,
\item $\caH_1^\dagger,\caH_2^\dagger$ are objects of $\Bun_G(\Gamma_0(N))^0(S)$ locally isomorphic to $\caH^\dagger$,
\item $f\from \caF^\dagger\vert_{X_S\smallsetminus \bigcup_{i=1}^r \gamma_i } \isomto \Fr^\ast\caF^\dagger(\half \Sigma_\infty)\vert_{X_S\smallsetminus \bigcup_{i=1}^r \gamma_i}$ is a shtuka bounded by $\mathrm{std}^{\boxtimes r}$ at the legs $x_1,\dots,x_r\in\P^1(S)$,
\item $f'\from \caF^\dagger\vert_{X_S\smallsetminus \gamma_{r+1}}\isomto \Fr^\ast\caF^\dagger(\half \Sigma_\infty')\vert_{X_S\smallsetminus \gamma_{r+1}}$ is a shtuka bounded by $\mathrm{std}$ at the leg $x_{r+1}\in\P^1(S)$,
\item For $i=1,2$, $h_i\from \caH_i^{\dagger}\vert_{X_S\smallsetminus \gamma_\tau} \isomto \caF^\dagger(\half \Sigma_\infty')\vert_{X_S\smallsetminus \gamma_\tau} $ is an isomorphism bounded by $\mathrm{std}$ at $\tau\in\P^1(S)$.
\end{itemize}
These are subject to a condition ($\ast$) that the composition 
\begin{equation}
\label{EqFormulaForG}
 g:=h_2(\half \Sigma_\infty'')^{-1}\circ f'(\half \Sigma_\infty)^{-1}\circ f(\half \Sigma_\infty') \circ h_1,
\end{equation}
a priori an isomorphism from $\caH_1^\dagger$ to $\caH_2(\half \Sigma_\infty'')$ away from $x_1,\dots,x_{r+1},\tau$ and bounded by $\mathrm{std}^{\boxtimes r+1} \boxtimes \mathrm{std}^{\otimes 2}$, extends over $\tau$ (more precisely, is bounded by the subrepresentation $\mathrm{std}^{\boxtimes r+1} \boxtimes \mathrm{triv}$).  

Now, by Lemma \ref{LemmaMapFromH}, $\caF^\dagger$ completely determines the data of $\tau$ and the $h_i$, so we see that $\mathcal{V}$ is a closed substack of $\Sht_G^r(\Gamma_0(N);\Sigma_\infty)^{0,\omega}\times \Sht_G^1(\Gamma_0(N);\Sigma_\infty')^{0,\omega}$.  We claim that $\mathcal{V}$ is birational to the graph of a dominant rational map $c\from \Sht_G^r(\Gamma_0(N);\Sigma_\infty)^0\dashrightarrow\Sht_G^1(\Gamma_0(N);\Sigma_\infty')^0$.

Proof of claim:  Given an $r$-legged shtuka $f$ involving $\caF^\dagger$, the space of shtukas $f'$ with one indeterminate leg lives in an ambient space $\P\Hom(\OO^2,\OO(1)\oplus\OO)\isom \P^5$, subject to 4 level conditions plus the additional linear condition $(\ast)$, which determines $f'$ uniquely, at least for generic $f$.  

We have defined a rational map
\[ c\from \Sht^r_G(\Gamma_0(N);\Sigma_\infty)^0 \dashrightarrow \Sht_G^1(\Gamma_0(N);\Sigma_\infty')^0\]
sending $f$ to $f'$.  There is also a morphism $\mathcal{V}\to \Coinc_G^r(\Gamma_0(N);\Sigma_\infty'')$ whose output is the composition $g$ above, inducing a rational map 
\[ \beta\from \Sht^r_G(\Gamma_0(N);\Sigma_\infty)^{0,\omega}\dashrightarrow \Coinc^{r+1}_G(\Gamma_0(N);\Sigma_\infty''). \]
We have defined a commutative diagram of rational maps \eqref{EqBigCoincidenceDiagram} at least on the degree 0 component.  To see that it is cartesian, suppose $g$ and $f'$ are given.  By Lemma \ref{LemmaMapFromH} we get an $h_1$ and an $h_2$, from which it is possible to solve for $f$ in \eqref{EqFormulaForG}. Finally, we may extend $c$ and $\beta$ to the degree 1 component using an Atkin-Lehner involution.

\subsection{Explicit equations for the spaces of $G$-coincidences} 

Let $F$ be a field, let $N\subset \P^1_F$ be an effective divisor of degree 4 with support $\Sigma$,  let $U=\P^1_F\smallsetminus \Sigma$, and let $\Sigma_\infty\subset \Sigma$ be a subset consisting of points appearing in $N$ with multiplicity 1.  Then we have defined the space of $G$-coincidences $\Coinc^r_G(\Gamma_0(N);\Sigma_\infty)\to U^r$ in Definition \ref{DefnGCoincidences}.  The fiber of this morphism over the generic point $\eta=\Spec F(t_1,\dots,t_r)$ is, as we will see, a quasi-projective variety over $\eta$ of dimension $r-3$. We record here some explicit equations for $\Coinc^r_G(\Gamma_0(N);\Sigma_\infty)_{\eta_r}$.  The case
\[ N = (0)+(1)+2(\infty) \]
being especially simple, we assume this.  Fix $\Sigma_\infty$ in the following way: $\emptyset$ if $r$ is even, and $\set{1}$ if $r$ is odd.   Let $\caH_0^\dagger$ be the object of $\OO(2)\oplus \OO$ of $\Bun_2(\Gamma_0(N))$, where the $\Gamma_0(N)$-level structure is $\OO_N$ embedded in the second factor.  Let $\caH_1^\dagger$ be a twist of $\caH_0^\dagger(\half \Sigma_\infty)$ which has degree $r+2$:  thus $\caH_1$ is isomorphic to $\OO((r+3)/2)\oplus \OO((r+1)/2)$ if $r$ is odd and $\OO(r/2+2)\oplus\OO(r/2)$ if $r$ is even.  
We can model $\Coinc^r_G(\Gamma_0(N);\Sigma_\infty)$ on the quotient by $\mathbf{T}\times\mathbf{T}$ of the subvariety $A\subset \P\Hom(\caH_0^\dagger,\caH^\dagger_1)\isom \P^{2r+3}$ describing those $f$ for which $\det f$ is nonzero.  Explicitly, $A$ classifies matrices with polynomial entries
\[ 
f(T)=\begin{pmatrix}
a(T) & b(T) \\
c(T) & d(T)
\end{pmatrix}
\]
taken up to nonzero scalar, where the entries $a,b,c,d$ have degrees bounded by $r/2$,$r/2+2$,$r/2-2$,$r/2$ (for $r$ even) or 
$(r-1)/2,(r+3)/2,(r-3)/2,(r+1)/2$ (for $r$ odd), such that $f(T)$ preserves the $\Gamma_0(N)$-level structure, and such that $\det f(T)=\prod_{i=1}^r (T-t_i)$ (again up to nonzero scalar).  The preservation of level structure means:
\begin{itemize}
\item $b(0)=b(1)=0$, and $\deg b(t)\leq r/2$, if $r$ even,
\item $b(0)=d(1)=0$, and $\deg b(t)\leq (r-1)/2$, if $r$ odd.
\end{itemize}
The action of $\mathbf{T}\times\mathbf{T}$ is by left and right translation on $f(T)$.    This is a quasi-projective variety over $\eta_r$ of dimension $r-3$. 

\begin{ex}\label{ex:coinc-3}[The case $r=3$.]  In this case, $\Coinc^3_G(\Gamma_0(N);\Sigma_\infty)_{\eta_3}\to\eta_3$ is a separable double cover of $\eta_3\isom \Spec F(t_1,t_2,t_3)$.  For our values of $N$ and $\Sigma_\infty$, we have that $\Coinc^3_{\circ}(\Gamma_0(N);\set{1})_{\eta_3}$ is isomorphic to the 0-dimensional subvariety of $\mathbf{A}^3$ in variables $a,b,d$ over $\eta_3=\Spec F(t_1,t_2,t_3)$, which classifies matrices 
\[ f(T)=\begin{pmatrix}
(T-a) & b T \\
1 & (T-1)(T-d)
\end{pmatrix}
\]
which are singular at $t_1,t_2,t_3$.  This translates into the equations
\[ (t_i-1)(t_i-a)(t_i-d)-b t_i = 0,\;i=1,2,3. \]
We can eliminate $b$ to obtain two equations for $a,d$, revealing that these are the roots of the irreducible polynomial
\begin{equation}
\label{EqMinPolyForCoinc3}
s^2 - (t_1+t_2+t_3-1)s + t_1t_2t_3.
\end{equation}
Then \eqref{EqMinPolyForCoinc3} is the equation for the double cover $\Coinc^3_G(\Gamma_0(N);\Sigma_\infty)_{\eta_3}\to\eta_3$ and its discriminant is
$(t_1+t_2+t_3-1)^2 - 4t_1t_2t_3$.  Considered as a quadratic polynomial in $t_3$, this discriminant is not a square for any specializations of $t_1$ and $t_2$ other than at $0,1$.  (Indeed, when considered this way, the discriminant of the discriminant is $16t_1t_2(t_1-1)(t_2-1)$, which is nonzero away from such values.) 
Therefore the double cover can be spread out into a family of double covers $\P^1_s\times U^2\to \P^1_{t_3} \times U^2$ defined by the equation \eqref{EqMinPolyForCoinc3}, which has the property that it is nonsplit over every point of $U^2$.  
\end{ex}

\begin{ex}\label{ex:coinc-4}[The case $r=4$.] $\Coinc^4_G(\Gamma_0(N))_{\eta_4}$ is isomorphic to the 1-dimensional subvariety of $\mathbf{A}^5_{\eta_4}$ in variables $a_0,a_1,b,d_0,d_1$ over the base $\eta_4\isom F(t_1,t_2,t_3,t_4)$, which classifies matrices
\[ f(T) = \begin{pmatrix} 
T^2 + a_1T+a_0 & bT(T-1) \\
1 & T^2+d_1T+d_0 
\end{pmatrix}
\]
which are singular at $t_1,t_2,t_3,t_4$.   The closure of $\Coinc^4_G(\Gamma_0(N))_{\eta_4}$ in $\mathbf{A}^5_{\eta_4}$ contains 6 obvious rational points $P_{ij}$, one for each partition of $\set{1,2,3,4}$ into ${i,j}$ and ${i',j'}$, defined by 
\[ f_{ij}(T)=\begin{pmatrix} (T-t_i)(T-t_j) & 0 \\ 1 & (T-t_{i'})(T-t_{j'}) \end{pmatrix} \]
Each of these points is nonsingular.   There is also an obvious involution of $\Coinc^4_G(\Gamma_0(N))_{\eta_4}$ which exchanges $(a_0,a_1)$ with $(d_0,d_1)$. 

The completion of this curve, call it $\caC$, has genus 1.  (One way to see this is that the projective closure in $\mathbf{P}^5$ has a unique singular point, and projection from this point to $\P^4$ gives a nonsingular curve of genus 1.  Another way to see this is to observe that projection onto $\A^2$ via the final coordinates $d_0,d_1$ is an isomorphism of $\Coinc^4_G(\Gamma_0(N))_{\eta_4}$ onto an affine cubic whose projective closure is nonsingular.)  A Weierstrass equation for $\caC$ is
\begin{equation}
\label{EqWeierstrassC}
 y^2 +e_1xy = x^3 + (-e_2+e_3-2e_4)x^2 + (1-e_1+e_2-e_3+e_4)e_4x,
 \end{equation}
where $e_1,\dots,e_4$ are the elementary symmetric polynomials in
$t_1,\dots,t_4$.   The points $P_{ij}$ generate the Mordell-Weil group of $\caC$, which is isomorphic to $\Z^3\oplus \Z/2\Z$.  The points $P_{ij}$ have coordinates 
\[ P_{ij}:\;\;(t_i(t_i-1)t_j(t_j-1),t_i(t_i-1)t_j(t_j-1)(-t_it_j+t_{i'}t_{j'}-t_{i'}-t_{j'}))\]
in the model \eqref{EqWeierstrassC};  we remark that $P_{12},P_{13},P_{14}$ are independent.  The involution noted above is translation by the 2-torsion point, $T$ which has coordinates $(0,0)$  in \eqref{EqWeierstrassC}.

It turns out that $\caC$ is the universal elliptic curve with this data, in the following sense.

\begin{prop}\label{prop:univ-ec-mbar14-order2} Let $\caM_{1,4}'$ be the moduli space of curves over $k$ of genus 1 together with 4 marked points and also an element of order 2 in $\Pic^0$.  There is an isomorphism of $\eta=\Spec k(t_1,t_2,t_3,t_4)$ with the generic point of $\caM_{1,4}'$, such that the pullback of the universal object is $(\caC;O,P_{12},P_{13},P_{14},T)$. 
\end{prop}


\begin{proof} We assume $\car k\neq 2$ for convenience.  Let $V$ be the group $\{(\pm 1, \pm 1)\}$ under multiplication, and let $V$ act on $\Mbar'_{1,4}$ by
$(\sigma_1,\sigma_2)(E;P_1,P_2,P_3,P_4;t) = (E;P_1,P_2,\sigma_1(P_3),\sigma_2(P_4);t)$, where negation is with respect to the origin $P_1$.  We claim that
$\Mbar'_{1,4}/V$ is birational to $\mathbf{A}^4_k$.  Indeed,  an elliptic curve with a point of
order $2$ has a model $y^2 = x(x^2+rx+s)$ which is unique up to replacing $(r,s)$ with $(\lambda^2 r,\lambda^4 s)$.   Since negation is an automorphism,  $P_2$ is only defined up to sign.  Our birational map $\Mbar'_{1,4}/V\dashrightarrow \mathbf{A}^4_k$ sends $(E;P_1,P_2,P_3,P_4;t)$ to 
$(s/r^2,x(P_2),x(P_3),x(P_4)) \in \A_k^4$. 

On the other hand, the data $(\caC,O,P_{12},P_{13},P_{14},T)$ defines a map $\eta\to \Mbar'_{1,4}$.  One calculates \cite{code} that the composition $\eta\to \Mbar'_{1,4}\to \Mbar'_{1,4}/V\dashrightarrow \A_k^4$ is dominant of degree 4.   We conclude that $\eta\to \Mbar'_{1,4}$ is an isomorphism from $\eta$ onto the generic point of $\Mbar'_{1,4}$.  
\end{proof}
\end{ex}

We can only speculate on the geometry of $\Coinc^r_G(\Gamma_0(N);\Sigma_\infty)_{\eta_r}$ for general $r$.   Let $n_1,n_2$ be the bounds on the degree of $c(T),d(T)$ in the description we gave above:  thus $n_1=r/2-2,(r-3)/2$ and $n_2=r/2,(r+1)/2$ as $r$ is even or odd, respectively.  If $r$ is odd, we have an
additional condition that $d(1) = 0$, so let $n'_2 = r/2$ if $r$ is even
and $(r-1)/2$ if $r$ is odd.  In both cases we have $n_1+n_2=r-2$.  Projection onto the pair $(c(T),d(T))$ or $(c(T),d(T)/(T-1))$, depending on the parity
of $r$, gives a birational equivalence between $\Coinc^r_G(\Gamma_0(N);\Sigma_\infty)_{\eta_r}$ and a hypersurface $C$ in $\P^{n_1}\times\P^{n'_2}$ of degree $(n_1+1,n'_2+1)$.  
If such a hypersurface has canonical singularities, it has a resolution
which is a (nonsingular) Calabi-Yau variety.  For instance, this holds in the case $r=5$, in which the variety $C$ is birational to an elliptic K3 surface.
We guess that it is true in general.

\section{Isogenies between K3 surfaces}

\subsection{Motivation: $2$-modularity of extremal rational elliptic fibrations}
Let $E/\F_q(t)$ be a nonisotrivial elliptic curve with degree 4 conductor $N$, with corresponding rational elliptic fibration $\E\to \P^1$.  Let $\Sigma$ be the support of $N$, and let $U=\P^1\smallsetminus \Sigma$.    In the introduction we sketched a plan to prove that $E$ is $2$-modular.  In the first phase of the plan, we found a dominant rational map
\[ \Sht^2_G(\Gamma_0(N)) \dashrightarrow \caZ^2(\E), \]
where $\caZ^2(\E)$ is defined as a certain cartesian product:
\begin{equation}
 \xymatrix{
\caZ^2(\E) \ar[r] \ar[d] & \E\times_{\F_q} U^2 \ar[d] \\
 \P^1\times U^2 \ar[r]_{2:1} & \P^1\times U^2
} 
 \end{equation}
The lower horizontal arrow is a certain family of nonsplit double covers of $\P^1$ parametrized by $U$, derived from the moduli space of 3-legged coincidences.  Now, the base change of a rational elliptic surface $\mathcal{A}\to \P^1$ by a nonsplit double cover $\P^1\to \P^1$ is a K3 surface.  Thus $\caZ^2(\E)\to U^2$ is a family of K3 surfaces.  As explained in the introduction, in order to prove that $E$ is 2-modular, it suffices to find a finite morphism $\caZ^2(\E) \to \Km(\E^2)$ commuting with the maps to $U^2$.  This is what we accomplish in this section and the next:

\begin{thm}  \label{ThmIsogenyFromZtoKm} There exists a finite morphism $\caZ^2(\E)\to \Km(\E^2)$ commuting with the maps to $U^2$.
\end{thm}

\begin{rmk} Theorem \ref{ThmIsogenyFromZtoKm} can be stated in any characteristic, since the definition of the space of coincidences is independent of characteristic.  In the cases where $\E\to \P^1$ is unstable our formulas confirm Theorem \ref{ThmIsogenyFromZtoKm} in any characteristic.  
\end{rmk}
%
%

Generally, it is an interesting problem to give sufficient conditions for two K3 surfaces to be isogenous,  see \cite{Mukai}, \cite{Buskin}, \cite{Yang}.   A necessary condition is some relation between the Picard lattices of the two surfaces.  Naturally one wants to know if this is sufficient:  

\begin{question}\label{QuestionIsogeny} Let $S$ be a K3 surface over an algebraically closed field.  Let $\Lambda$ be a lattice such that $\Pic S\otimes \Q$ and $\Lambda\otimes\Q$ are isometric.  Does there exist a K3 surface $S'$ in correspondence with $S$ such that $\Pic S'\isom \Lambda$?
\end{question}

For many values of $\Lambda$, operations on elliptic fibrations can be used to construct the correspondence.  The main theorem of this section answers a special case of Question \ref{QuestionIsogeny}.

\begin{thm} \label{ThmIsogenyFromK3ToKummer} Let $S$ be a K3 surface over an algebraically closed field.  Suppose there is an isometry $\Pic S \otimes \Q\isom \Pic K\otimes \Q$, where $K=\Km(E_1\times E_2)$ for two non-isogenous elliptic curves $E_1$ and $E_2$.  Then there is a morphism of finite degree from $S$ to a Kummer surface of this form.
\end{thm}

\subsection{Elliptic fibrations on K3 surfaces}

We begin with a few remarks on elliptic fibrations and especially elliptic K3 surfaces.  Some standard references are \cite{SchuettShioda} for elliptic fibrations and \cite{huybrechts} for K3 surfaces.  Unless stated otherwise, by ``surface'' we mean a smooth projective surface over an algebraically closed field.

\subsubsection{Generalities on elliptic surfaces}  Let $S$ be a surface. A {\em genus 1 fibration} on $S$ is a morphism $S\to X$ to a curve whose geometric generic fiber is a smooth integral curve of genus 1.  If $S\to X$ admits a section, we call it an {\em elliptic fibration}.   In any case it always has a {\em multisection}, meaning a curve $C\subset S$ such that $C\to X$ is finite;  the degree of the multisection is the degree of $C\to X$.  An {\em elliptic surface} is a surface admitting an elliptic fibration, such that no fiber contains an exceptional curve (meaning a smooth rational curve of self-intersection $-1$).
In this paper all elliptic surfaces are assumed to be nonisotrivial; in other
words, they are not products of two curves.


For a surface $S$, the {\em N\'eron-Severi group} $\NS(S)$ is the group of divisors modulo algebraic equivalence.  It is a finitely generated abelian group endowed with an intersection pairing, which we write as $(v,w)$.  This pairing is nondegenerate on the torsion-free part of $\NS(S)$.  We write $\rho(S)$ for the rank of the torsion-free part of $\NS(S)$.  If $S\to X$ is an elliptic fibration with general fiber $F$ and section $O$, the {\em trivial lattice} $T\subset \NS(S)$ is generated by $O$ and all irreducible components of fibers of $S\to X$.  Suppose the reducible fibers occur at $R\subset X$.  For $v\in R$, let $x_0,\dots,x_{m_v-1}\in \NS(S)$ be the irreducible components of the fiber $S_v$, where $x_0$ is the unique such component which meets $O$.  As all fibers are equivalent to $F$, the trivial lattice $T$ is the orthogonal direct sum of $\class{O,F}$ and $\class{x_1,\dots,x_{m_v-1}}$.  

The isomorphism classes of singular fibers of an elliptic surface have a well-known classification by Kodaira symbol.  For each such fiber, the corresponding summand of $T$ is a root lattice of type $A_n$, $D_n$ $(n\geq 4)$, $E_6$, $E_7$, or $E_8$.  Here our convention is that our root lattices are {\em negative definite}.  We will make use of the standard facts about each Kodaira symbol without comment, e.g., the corresponding Dynkin diagram, multiplicity of components, discriminant, etc.

There are two useful formulas relating the geometry of an elliptic surface $S$ to the singular fibers of an elliptic fibration on $S$.  One is the Shioda-Tate formula, and the other is a formula for the Euler number of $S$.

The {Mordell-Weil group} of an elliptic fibration $S\to X$ is the group of sections $S(X)$. Equivalently it is the set of rational points of the generic fiber of $S\to X$, which is an elliptic curve. There is an isomorphism 
\begin{equation}
\label{EqNSFormula}
NS(S)/T \isomto S(X)
\end{equation}
 sending a divisor to its generic fiber.  In particular we have the Shioda-Tate formula \cite[Proposition 6.6]{SchuettShioda}:
 \begin{equation}
 \label{EqShiodaTate}
 \rho(S) = 2+\sum_{v\in R} (m_v-1) + \rank S(X)/(\text{torsion})
\end{equation}

The smooth locus in each singular fiber $S_v$ of an elliptic fibration has a group structure, with neutral component either $\mathbf{G}_m$ (type $I_n$) or $\mathbf{G}_a$ (all other types).  We call $S_v$ a multiplicative or additive fiber, respectively.  
For an additive fiber $v$, let $\delta_v$ be the index of wild ramification, defined by the formula:
\[ \delta_v = v(\Delta) - 1 - m_v \]
where $\Delta$ is the discriminant of a Weierstrass equation for $S$ over the local ring $\OO_{X,v}$.  Then $\delta_v\neq 0$ only if the ground field has characteristic 2 or 3.  If $\delta_v=0$ for all $v$, we say that $S\to X$ is tame.  

For a projective variety $Y$, write $e(Y)$ for the Euler number (= topological Euler characteristic = alternating sum of Betti numbers of $Y$).  For a singular fiber $S_v$ of an elliptic fibration $S\to X$, we have 
\begin{equation}
\label{EqEulerNumberFormula}
 e(S_v)=\begin{cases} m_v, & \text{if $S_v$ is multiplicative} \\ m_v+1, & \text{if $S_v$ is additive.} \end{cases} 
\end{equation}
We have the following relations:
\begin{equation}
\label{EqEulerNumberChi}
\chi(S,\OO_S) = \frac{1}{12}e(S) 
\end{equation}
(where $\chi$ = Euler characteristic), and 
\begin{equation}
e(S) = \sum_v (e(S_v)+\delta_v),
\end{equation}
where the sum is over the singular fibers of $S\to X$.

\subsubsection{K3 elliptic surfaces} A smooth projective surface $S$ is {\em K3} if it has trivial canonical bundle and $H^1(S,\OO_S)=0$.  

Let $S$ be a K3 surface.  Then on $S$, the notions of linear, algebraic, and numerical equivalence of divisors agree, and $\Pic S\isom \NS(S)$ is torsion-free.   We refer to this group as the Picard lattice. It has signature $(1,\rho(S)-1)$.  By Riemann-Roch, an irreducible curve $C\subset S$ of arithmetic genus $g$ has $(C,C)=2g-2$.  This implies that $\Pic S$ is an even lattice.  

In the case that $k=\mathbf{C}$ is the field of complex numbers, then $\Pic S$ is a primitive sublattice of $H^2(S,\Z)$.  For its part, $H^2(S,\Z)$ is isomorphic to the {\em K3 lattice} $E_8^2\oplus U^3$, where $U$ is the hyperbolic plane.  By the Lefschetz principle, this fact about $\Pic S$ extends to any field of characteristic 0.  In positive characteristic, we can use the following lemma.

\begin{prop}[{\cite[Proposition 5.8, Chapter 9]{huybrechts}}] \label{PropLifting} Let $S$ be a K3 surface of finite height
  over a perfect field $k$ of characteristic $p$.  Then $S$ can be lifted
  to $W(k)$ in such a way that the specialization map from the N\'eron-Severi
  group of the general to the special fibre is an isomorphism.
\end{prop}

We are interested in those K3 surfaces which admit genus 1 fibrations. If $S$ is K3 and $S\to X$ is a genus 1 fibration, then necessarily $X\isom \P^1$.  An interesting fact about elliptic K3 surfaces $S$ is, if their Picard rank is not too small, they typically admit several genus 1 fibrations which are distinct modulo the action of $\Aut S$.  It will be important for us to recognize the genus 1 fibrations on $S$ purely by examining the lattice $\Pic S$.

\begin{prop}[{\cite[\S3, Theorem 1]{PSS}}]  \label{PropExistenceOfEllFib}
Assume that the characteristic of $k$ is not 2 or 3.  Let $S$ be a K3 surface of finite height.  
\begin{enumerate}
\item Suppose $F$ is a primitive class in $\Pic S$ with $(F,F)=0$.  Then there exists a genus 1 fibration $S\to\mathbf{P}^1$ whose fiber class is $\OO(\Pic S)$-equivalent to $F$.  
\item Continuing, let $d$ be the greatest common divisor of $(F,D)$ for all $D\in \Pic S$.   Then $S\to\mathbf{P}^1$ admits a multisection of degree $d$.  
\item Continuing further, suppose $d=1$, so that $S\to\mathbf{P}^1$ is an elliptic fibration with fiber $F$ and section $O$.  Then $\class{F,O}$ is the hyperbolic plane.   The reducible fibers of $S\to\mathbf{P}^1$ correspond exactly to the root lattice summands of the root sublattice of $\class{F,O}^{\perp}$.  (The {\em root sublattice} is the sublattice spanned by vectors $v$ with $(v,v)=-2$.)   
\end{enumerate}
\end{prop}

\begin{proof} (1) One can apply standard results on $\Pic S$ \cite[Corollary 8.2.9]{huybrechts} to find a sequence of smooth rational curves $C_1,\dots,C_r$ on $S$ such that $F'=\pm \rho_{C_r}\circ\cdots\circ \rho_{C_1} (F)$ is 
nef, where $\rho_C(v)=v-(2(v,C)/(C,C))C$ is the reflection in the hyperplane orthogonal to a Picard class $C \in \Pic S$ satisfying $C^2 = -2$.  Since $\rho_C$ preserves the intersection pairing, we have $(F',F')=0$.  By \cite[Proposition 2.3.10]{huybrechts}, the linear system $\abs{F'}$ has no base points, and the associated morphism to projective space factors through a genus 1 fibration $S\to \P^1$, with fiber equivalent to $F$.

(2) It is clear that $F'\equiv F\pmod{d}$ in $\Pic S$, so that there exists $D\in \Pic S$ with $(F',D)=d$.  After replacing $D$ with $D+nF'$ for $n\gg 0$, we may assume that $D$ is effective.  Then $D$ is a multisection of the fibration of degree $d$.

(3) 
Let ${\mathcal C} \subset \class{F,O}^\perp$ be the set of $C$ that satisfy
$(C,C) = -2$.  By Riemann-Roch either $C$ or $-C$ is effective, so let
${\mathcal C}^+$ be the set of effective elements.  We claim that every 
element of ${\mathcal C}^+$ is a sum of classes of curves contained in fibers.

Indeed, write $C \in {\mathcal C}^+ = \sum a_i [C_i]$, where the $a_i$ are
nonnegative integers and the $C_i$ are classes of irreducible curves.
Since $[F]$ moves with empty base locus, it has nonnegative intersection
with all curves, and positive intersection with all curves not contained in
any member of the linear series $|F|$.  In particular $O$ is not among the
$C_i$, because $[F]\cdot[O] = 1$.
Therefore $O \cdot C_i \ge 0$ for all $i$, and so $F$ is not one of
the $C_i$ either.  On the other hand it is clear that if $C_i$ is a
curve in a fibre that does not meet $O$, then $[C_i] \in {\mathcal C}$.

Let $R$ be the root sublattice of $\class{F,O}^\perp$.  From the above we see
that $R$ is freely generated by the classes of curves in fibers that do not
meet $O$.  For each reducible fibre we obtain a root lattice as in
\cite[\S11.13]{SchuettShioda}, and distinct fibers give orthogonal
sublattices of $R$.
\end{proof}

\begin{cor} Let $S$ be a K3 surface of finite height.   If the Picard rank $\rho(S)$ is at least 5, then $S$ admits a genus 1 fibration.
\end{cor}
\begin{proof} The quadratic form of $\Pic S$ is indefinite and represents zero over $\Q_p$ for all primes $p$.  Therefore by Hasse-Minkowski it represents zero rationally; i.e., there exists $F\in \Pic S$ with $(F,F)=0$.  By Proposition \ref{PropExistenceOfEllFib}, the surface $S$ admits a genus 1 fibration.
\end{proof}

\begin{rmk}\label{rem:char-p-bad} If the characteristic of $k$ is $2$ or $3$, it is possible that a nef divisor with self-intersection $0$ does not actually define a genus 1 fibration in the sense we have defined it; it can happen that the generic fibre has a cusp.  See \cite[proof of Proposition 2.3.10]{huybrechts}.  
\end{rmk}

In the situation of Proposition \ref{PropExistenceOfEllFib},  we have a genus 1 fibration on $S$ admitting a multisection $D$ of degree $d$.  We may pass to the Jacobian of this fibration \cite[Chapter 11, \S 4]{huybrechts}:  this is another K3 surface $S'$ admitting a Jacobian fibration.  There is a finite morphism $S\to S'$ of degree $d^2$, sending a section $P$ to $dP-D$.   The following lemma identifies the Picard lattice of $S'$.

\begin{lm}[{\cite[Lemma 2.1]{Keum}}]\label{lem:keum} Let $S$ be a K3 surface admitting a genus 1 fibration of multisection index $d$ with fiber $F\in \Pic S$, and let $S'$ be the Jacobian of that fibration.   Then  $\Pic S'$ is isomorphic to the overlattice $(\Pic S)[F/d]$ of $\Pic S$.  
\end{lm}

It is natural to ask about a sort of converse to Lemma \ref{lem:keum}.  Suppose $S$ is a K3 surface with Picard lattice $L$, and suppose $L'\supset L$ is an overlattice with $L'/L$ cyclic of degree $d$.   Does there exist a genus 1 fibration on $S$ whose Jacobian has Picard lattice $L'$?  This is true if the rank of $L$ is at least 13, as we will see in Proposition \ref{prop:genus-1-exists}.

\begin{lm}\label{lem:split-quad} Let $R$ be a principal ideal domain, let $M$ be a free $R$-module of finite rank, and let $Q$ be a quadratic form on $M$.  Let $N\subset M$ be a submodule such that the discriminant of $Q|_N$ is a unit in $R$.  Then
  $M = N \oplus N^{\perp}$.
\end{lm}

\begin{proof}
  First note that $N \cap N^\perp = 0$, as otherwise $Q|_N$ would
  have discriminant $0$.  
  
  Let $F$ be the fraction field of $R$, and let $N^\vee\subset M\otimes_R F$ be the dual lattice, i.e., the $R$-submodule of $M\otimes_R F$ consisting of those $m$ for which $(m,N)\subset R$.  Then $\Hom(N,R)\isom N^\vee$.   By hypothesis, we have $N = N^\vee$.  Let $m\in M$.   Then we can find $n\in N$ such that $(m,x)=(n,x)$ for all $x\in N$.  Then $m-n\in N^\perp$, and we conclude $N+N^\perp=M$.
\end{proof}

In the context of Lemma \ref{lem:split-quad}, we say that $Q$ {\em primitively represents} $a\in R$ if there exists a primitive element $x\in M$ with $Q(x)=a$.

\begin{lm}\label{lem:represent-zp} Let $p$ be prime
  and let $Q$ be a quadratic form over
  $\Z_p$ of rank $n\geq 3$ and
  unit discriminant.  If $p = 2$, assume further that $Q$ is even.
  Then $Q$ represents all elements of $\Z_p$
  primitively.
\end{lm}

\begin{proof}  Suppose $p$ is odd.  Then $Q$ may be diagonalized \cite[p. 369, Theorem 2]{SPLAG}:  $Q=\sum_{i=1}^n a_ix_i^2$. Since $Q$ has unit discriminant, each $a_i\in \Z_p^\times$.  A standard counting argument shows that $a_1x_1^2+a_2x_2^2+a_3x_3^2$ primitively represents all elements of $\F_p$.  By Hensel's lemma, $Q$ primitively represents all of $\Z_p$.

If $p = 2$, then (loc.~cit.) the form $Q$ is isomorphic to the direct
sum of a diagonal form $\sum_{i=1}^m a_ix_i^2$ with forms of the shape
$2^rQ_{a,b,c}$, where $r\geq 1$ and $Q_{a,b,c}(x,y)=ax^2+bxy+cy^2$ is a
form of unit discriminant
with $a,c\in\Z_2$ and $b\in\Z_2^\times$.  Note that $\disc 2^rQ_{a,b,c}$
has 2-adic valuation $2r$.  Since $Q$ has unit discriminant, all
$a_i\in\Z_2^\times$, and all the $r$ are $0$.  But then since $Q$ is even,
$m=0$.

Thus $Q$ is the direct sum of at least two forms $Q_{a,b,c}$.  For each of
these, we consider two cases.
If $2|ac$ then the discriminant $b^2-4ac$ is congruent to
$1 \bmod 8$.  It is therefore the square of a unit, and so $Q$ is a
product of two linear factors $(rx+sy),(r'x+s'y)$ that generate the space
of linear forms in two variables over $\Z_2$.
So $Q$ represents every element of $\Z_2$ primitively in this case.
Alternatively, if $a,c \in \Z_2^\times$, then
$Q_{a,b,c}$ defines a nonsingular conic over $\Z_2$, so by Hensel's
lemma $Q_{a,b,c}$ primitively represents all of $\Z_2^\times$.
Since every element of $\Z_2$ is either a unit or the sum of two units,
it follows that $Q$ primitively represents all of $\Z_2$ if there are two
factors of this type.
%
\end{proof}

\begin{prop}\label{prop:genus-1-exists} Let $S$ be a 
  K3 surface with Picard number at least $13$.  Assume that the ground field has characteristic 0, or else that $S$ has finite height.  Let $p$ be prime, and let $D$ be a divisor of $S$ which is not divisible by $p$ in $L=\Pic S$.  Assume that $p\vert (D,x)$ for all $x\in L$, and also that $p^2\vert (D,D)$.   Equivalently, $L'=L[D/p]\subset \Pic S\otimes\Q$ is an overlattice of $L$ of index $p$.   If $p=2$, assume that $L'$ is even.  (Since $L$ is even, this condition is equivalent to $8\vert (D,D)$.)  
  
Then there exists a divisor class $D'\in D+pL$ such that the $p$-part of $\gcd_{x\in L}(D',x)$ is $p$, and a genus 1 fibration $S\to \PP^1$ with fiber $D'$.  In particular, if $S'\to \PP^1$ is the $p$-Jacobian $J^p$ (\cite[Remark 11.4.1]{huybrechts})) of the fibration, then $\Pic S'\isom L'$.  
\end{prop}


\begin{proof} Using Proposition \ref{PropLifting} we assume that the ground field is $\mathbf{C}$.

Let $M\subset \Pic S\otimes \Z_p$ be a $\Z_p$-sublattice of maximal rank among those which have unit discriminant.  Write $Q$ for the quadratic form on $\Pic S\otimes\Z_p$.  By Lemma \ref{lem:split-quad}, we have $\Pic S\otimes\Z_p=M\oplus N$, with $N=M^{\perp}$.  Then $Q\vert_{N}$ is divisible by $p$.  Indeed, if $x\in N$ satisfies $Q(x)\in\Z_p^\times$, then $M\oplus \class{x}$ would be a larger submodule with unit discriminant.  In fact we claim that the intersection form on $N$ is divisible by $p$, i.e., that $N\subset pN^\vee$.  This follows automatically if $p$ is odd.  If $p=2$, we also need to exclude the possibility of $x,x'\in N$ with $(x,x')\in\Z_2^\times$.  But since we already know that $(x,x),(x',x')\in 2\Z_2$, we have $(x,x)(x',x')-(x,x')^2\in\Z_2^\times$, and again we have constructed a larger submodule $M\oplus\class{x,x'}$ with unit discriminant.

Recall that the discriminant group of a lattice $\Lambda$ (over whatever PID base) is $D(\Lambda)=\Lambda^\vee/\Lambda$.   Then $D(L)\otimes\Z_p\isom D(L\otimes \Z_p)$.  On the other hand $L\otimes\Z_p=M\oplus N$ as above, and we have $D(M)=0$, whereas 
\[ D(N)\otimes \F_p=N^\vee/(N,pN^\vee)=N^\vee/pN^\vee =N^\vee \otimes\F_p \]
since $N\subset pN^\vee$ as we observed above.  We conclude from this that $\dim_{\F_p} D(L)\otimes\F_p=\rank N$. 

Now we use that fact that $L=\Pic S=\NS(S)$ embeds primitively into the K3 lattice $H=H^2(X,\Z)$.   (This is because for any complex projective surface $S$, we have $\NS(S)\isom H^{1,1}(S,\mathbf{C})\cap H^2(S,\Z)$.)   Crucially, $H$ is unimodular.  Let $L^\perp$ be the orthogonal complement of $L$ in $H$.  The intersection pairing on $H$ induces a pairing $D(L)\otimes H/(L\oplus L^\perp)\to \Q/\Z$. 

This pairing puts $D(H)$ and $H/(L\oplus L^\perp)$ into Pontrjagin duality. 
Proof: if $h\in H$ satisfies $(L^\vee,h)\subset\Z$, write $h=\ell+\ell^\perp$, where $\ell\in L\otimes\Q$, $\ell^\perp\in L^\perp \otimes\Q$. Then $(L^\vee,\ell)\subset\Z$ implies $\ell\in L^{\vee\vee}=L$ and therefore $\ell^\perp\in H\cap (L^\perp\otimes\Q)=L^\perp$ and $h\in L\oplus L^\perp$.  Conversely if $\ell^\vee\in L^\vee$ satisfies $(\ell^\vee, H)\in\Z$, then $\ell^\vee\in H^\vee\cap L^\vee = H\cap L^\vee \subset H\cap (L\otimes\Q)=L$ because $L$ is embedded primitively in $H$.

In particular $\dim_{\F_p} D(L)\otimes\F_p = \dim (H/(L\oplus L^\perp))\otimes\F_p\leq \rank L^\perp=22-\rank L$, as one can see by extending a $\Z$-basis of $L$ to $H$.  Thus $\rank N\leq 22-\rank L$.

We have assumed $\rho(S)=\rank L\geq 13$, and so
\[ \rank M=\rank L -\rank N \geq 2\rank L - 22 \geq 4.\]

Our divisor $D$ was assumed to satisfy $p\vert (D,x)$ for all $x\in L$.   If we decompose $D$ as $D=m+n$ with $m\in M$, $n\in N=M^\perp$, then $p\vert (m,x)$ for all $x\in M$, which is to say, $m\in pM^\vee$.  But since $M$ is unimodular, we find that $m\in pM$.  We find that $n\equiv D\pmod{L\otimes p\Z_p}$.   We have also assumed that $L[D/p]$ is a lattice, which implies that $p^2\vert (D,D)$;  if $p=2$ we have assumed $8\vert (D,D)$.  The same statements are true when $D$ is replaced with $n$.   By Lemma \ref{lem:represent-zp}, we can find a vector $x\in M$ such that $(x,x)=-(n,n)/p^2$.  Then $D_0=n+px\in L\otimes\Z_p$ satisfies $D_0\equiv D\pmod{p}$ and $(D_0,D_0)=0$.

By weak approximation on quadric hypersurfaces, there exists $D'\in L\otimes\Q$ satisfying $(D',D')=0$ which is $p$-adically close to $D_0$.  After clearing denominators, we may assume $D'\in L$.  Now we may apply Proposition
  \ref{PropExistenceOfEllFib} to obtain the required genus 1  fibration.  
  
\end{proof}

\begin{cor}\label{cor:drop-p}
  Let $S$ be a finite height K3 surface of Picard number $\ge 13$, and let
$p$ be an odd prime.   Let $d=\disc \Pic S$, and let $p^v$ be the largest power of $p$ dividing $d$.  Suppose we have an isometry
  $\Pic S \otimes \Q_p \isom L \otimes \Q_p$, where $L$ is a unimodular $\Z_p$-lattice.  Then there is a map of finite degree from $S$
  to a K3 surface $S'$ for which $\disc \Pic S'=d/p^v$.  
\end{cor}

\begin{proof}  The hypothesis on $\Pic S$ implies that $v$ is even, as the discriminant of a quadratic form is well-defined up to a square.  It also implies that the Hasse invariant of $\Pic S\otimes\Q_p$ is trivial, as this is the case for $L\otimes\Q_p$.  Indeed, we may diagonalize the quadratic form on $L$ as $\sum_{i=1}^n a_ix_i^2$, with $a_i\in\Z_p^\times$, and then the Hasse invariant is $\prod_{i<j}(a_i,a_j)$ (Hilbert symbol).  Each factor in the product is trivial (note that $p$ is odd).  

We proceed by induction on $v$, the case $v=0$ being obvious.  First suppose that $D(\Pic S)$ contains an element of order $p^2$, represented by $D^\vee\in (\Pic S)^\vee$.   Then $D=p^2D^\vee \in \Pic S$ satisfies the hypotheses of Proposition \ref{prop:genus-1-exists}, in which case there is a degree $p^2$ map $S\to S'$ with $\disc \Pic S'=d/p^2$.

Therefore assume that $D(\Pic S)\otimes\Z_p$ is $p$-torsion, in which case its $\F_p$-dimension is $v$.  The quadratic form on $\Pic S\otimes\Z_p$ is equivalent to the diagonal form
\begin{equation}
\label{EqLongQuadraticForm}
 pa_1x_1^2+\cdots+pa_vx_v^2+a_{v+1}x_{v+1}^2+\cdots +a_nx_n^2, 
 \end{equation}
with each $a_i\in\Z_p^\times$.  If $v>2$, there exist $x_1,x_2,x_3\in\Z_p$ such that $a_1x_1^2+a_2x_2^2+a_3x_3^2\equiv 0\pmod{p}$ since every nonsingular conic over $\F_p$ has a rational point.  Let $D\in \Pic S$ be $p$-adically close to $(x_1,x_2,x_3,0,\dots)\in \Pic S\otimes\Z_p$.  Then  $D$ satisfies the hypotheses of Proposition \ref{prop:genus-1-exists}, and again we can remove a factor of $p^2$ from $\disc \Pic S$.

We are now reduced to the case $v=2$.  Standard formulas for the Hilbert symbol reveal that $(pa_1,a_i)=(pa_2,a_i)$ for all $i\geq 3$, so that the Hasse invariant of \eqref{EqLongQuadraticForm} is $(pa_1,pa_2)=(-a_1a_2/p)$ (Legendre symbol).  But by the observation in the first paragraph, the Hasse invariant is trivial; i.e., $-a_1a_2$ is a square modulo $p$.  Thus $a_1x_1^2+a_2x_2^2$ represents 0 modulo $p$, and we proceed as in the previous paragraph.
\end{proof}

\subsection{Kummer surfaces $\Km(E_1\times E_2)$ of Picard rank 18}

Here we recall the basic constructions and properties of the Kummer surface $\Km(A)$ attached to an abelian surface $A$.   For the moment let us suppose we are in characteristic not 2.   Denoting by $\iota$ the involution $x\mapsto -x$ on $A$, the quotient $A/\iota$ has rational double point singularities at each of the 16 fixed points of $\iota$ (namely, the 2-torsion points of $A$).  Let $\Km(A)$ be the minimal resolution of $A/\iota$.  Then $\Km(A)$ is a K3 surface. 

The Picard lattice $\Pic \Km(A)=\NS \Km(A)$ contains both $\NS(A)$ and the classes of the 16 exceptional divisors.  In fact these generate $\Pic \Km(A)$ and we have:
\begin{equation}
\label{EqRhoRelation}
 \rho(\Km(A)) = 16+\rho(A) 
 \end{equation}

We are mostly interested in the case that $A=E_1\times E_2$ is a product of non-isogenous elliptic curves $E_1,E_2$.  In this case, \eqref{EqRhoRelation} gives $\rho(\Km(A))=18$.   There is an elliptic fibration $\Km(E_1\times E_2)\to \mathbf{P}^1$ given by projecting onto $E_2/\iota\isom \mathbf{P}^1$;  the geometric fibers are all isomorphic to $E_1$.  In fact this gives an alternate construction of $\Km(E_1\times E_2)$:  it is the quadratic twist of the constant elliptic fibration $E_1\times \mathbf{P}^1\to \mathbf{P}^1$ by the double cover $E_2\to \mathbf{P}^1$.  (Of course, we could have reversed the roles of $E_1$ and $E_2$ in this construction.)

The bad fibers of $\Km(E_1\times E_2)\to \mathbf{P}^1$ occur at the four branch points of $E_2\to \mathbf{P}^1$, and they are all of type $I_0^\ast = \tilde D_4$.  Therefore the trivial lattice of the fibration is $U\oplus D_4^{\oplus 4}$.  We note here that $\disc D_4=4$.  On the other hand the Mordell-Weil group of this fibration has order 4 (coming from the 2-torsion in $E_1$).  We conclude from \eqref{EqNSFormula} that $\Pic \Km(E_1\times E_2)$ has discriminant $-16$. The lattice $\Pic\Km(E_1\times E_2)$ does not depend on the elliptic curves $E_1$ or $E_2$, so long as they are nonisogenous.   In the complex setting, the transcendental lattice of $\Km(E_1\times E_2)$ is $T_{E_1\times E_2}(2)\isom (H^1(E_1,\Z)\otimes H^1(E_2,\Z))(2)\isom U^{\oplus 2}(2)$.

We are interested in the question of whether a given K3 surface $S$ of rank 18 is isogenous to a Kummer surface of the form $\Km(E_1\times E_2)$, and if so, how to find the elliptic curves $E_1$ and $E_2$.    The following result relies on a deep theorem of Mukai.

\begin{prop} \label{pro:Mukai} Let $S$ be a complex K3 surface such that $\Pic S\otimes \Q$ is isometric to $\Pic \Km(E_1\times E_2)\otimes \Q$ for some nonisogenous elliptic curves $E_1,E_2$.  Then there exists an isogeny between $S$ and a Kummer surface of that form.  (An isogeny between K3 surfaces $S$ and $S'$ is an algebraic cycle on $S\times S'$ inducing an isometry $H^2(S,\Q)\isomto H^2(S',\Q)$.) 
\end{prop}

\begin{proof} (Sketch.)  Let $T_S$ be the transcendental lattice of $S$.  Then $T_S$ has signature $(2,2)$, and is the complement of $\Pic \Km(E_1\times E_2)$ in $H^2(S,\Z)$.  A calculation involving Hasse-Minkowski invariants shows that $T_S\otimes \Q$ is isometric to $U^{\oplus 2} \otimes \Q$.   Via this isometry, the Hodge structure on $T_S$ now determines a Hodge structure on $U^{\oplus 2}$, which is to say, a morphism of real groups $h\from \mathbf{S}\to O(2,2)$, where $\mathbf{S}$ is the Deligne torus.  Now observe that there is an isomorphism $(g_1,g_2)\mapsto g_1\otimes g_2$ from $(\Sp(2)\times \Sp(2))/\set{\pm 1}$ onto the neutral component of $O(2,2)$.  Therefore $h$ factors through a morphism $\mathbf{S}\to (\Sp(2)\times \Sp(2))/\set{\pm 1}$.  For $i=1,2$, let $h_i$ be the projection of this morphism onto the $i$th copy of $\Sp(2)/\set{\pm 1}\isom \PGL_2$;  then $h_i$ determines an elliptic curve $E_i$.   Thus we have an isometry of rational Hodge structures:
\[ T_S\otimes\Q\isom H^1(E_1,\Q)\otimes H^1(E_2,\Q) \isom T_{\Km(E_1\times E_2)}\otimes \Q \]

We now invoke a theorem of Mukai \cite[Corollary 1.10]{Mukai}:  an isometry between rational Hodge structures of two K3 surfaces of rank $\geq 11$ is always induced from an isogeny.  
\end{proof}

In this section we prove an effective version of Proposition \ref{pro:Mukai}, whereby the elliptic curves $E_1$ and $E_2$ can in theory be computed from $S$, and where the isogeny is simply a finite rational map (which can also be computed).   The idea is to leverage operations on elliptic fibrations.  We build up the result in stages.

  
\begin{prop} \label{PropRecognizeKummer} Let $S$ be a K3 surface whose Picard lattice is isometric to $\Km(E_1\times E_2)$ for some nonisogenous elliptic curves $E_1,E_2$.  Then $S$ is isomorphic to a Kummer surface of that form.
\end{prop}

\begin{proof}  Proposition \ref{PropExistenceOfEllFib} produces an elliptic fibration $S\to \mathbf{P}^1$ with $I_0^\ast$ fibers at four points of $\mathbf{P}^1$.  Let $E_2\to \mathbf{P}^1$ be the elliptic curve branched at exactly these four points.  Then the quadratic twist of $S\to \mathbf{P}^1$ by $E_2\to\mathbf{P}^1$ has $I_0$ reduction at these points; i.e., it has good reduction everywhere and therefore must be a constant elliptic curve $E_1\times \mathbf{P}^1$.  Thus $S$ is a quadratic twist of $E_1\times \mathbf{P}^1$ by $E_2\to \mathbf{P}^1$, and this is exactly $\Km(E_1\times E_2)$.  
\end{proof}

\begin{rmk}\label{rem:shioda-inose-exists}
  Recall that a {\em Shioda-Inose structure}
  (\cite[Definition 6.1]{morr}) on a complex K3 surface $S$ is a rational map
  $S \to S'$ of degree $2$ such that $S'$ is a Kummer surface and
  $T_{S'} \isom 2T_{S}$, meaning that that the Gram
  matrix of $T_{S'}$ is twice that of $T_S$.   Morrison proves \cite[Theorem 6.3]{morr} that $S$ has
  a Shioda-Inose structure if and only if $\Pic S \otimes \Q$ is isometric
  to $\Pic K \otimes \Q$ for some Kummer surface $K=\Km(A)$, if and only if
  $T_A$ embeds primitively in $U^3$.
\end{rmk}

\begin{prop}\label{prop:to-kum}
  Assume that $\car k\neq 2$.  Let $S$ be a K3 surface whose Picard lattice has rank $18$ and
  discriminant $-1$.  Then there is a finite map from $S$ to a K3 surface
  of the form $\Km(E_1 \times E_2)$, where
  $E_1,E_2$ are nonisogenous elliptic curves.
\end{prop}

\begin{proof} The transcendental lattice of $S$ is an even unimodular lattice
  of rank $4$ and signature $(2,2)$, so it is isometric to $U^2$.  By
  Remark~\ref{rem:shioda-inose-exists}, there is a Shioda-Inose structure on
  $S$.  The codomain of this isogeny has transcendental lattice $U^2(2)$,
  which as we have seen is that of the Picard lattice of $\Km(E_1 \times E_2)$
  for nonisogenous $E_1, E_2$.  The result now follows from
  Proposition~\ref{PropRecognizeKummer}.
\end{proof}

We need a further lattice-theoretic result.

\begin{prop}\label{prop:always-contained}
  Every even lattice $L$ of discriminant $-4^i$ and signature $(1,17)$ is
  contained in a unimodular lattice, necessarily $II_{1,17}=U \oplus E_8^{\oplus 2}$.  
\end{prop}


\begin{proof} We will show that if $i\geq 1$ then $L$ is contained in a lattice of discriminant $-4^{i-1}$, and then the result follows by induction. The idea is that if $x\in D(L)=L^\vee/L$ is an element of order 2 satisfying $(x,x)\in 2\Z$ (i.e., $x$ is isotropic in $D(L)[2]$), then $L'=L+\class{x}$ is again an even lattice, and $\disc L'=(\disc L)/4$. 

We need the fact that (for general lattices $L$) the $\F_2$-dimension of $D(L)[2]$ has the same parity as the rank of $L$, which in our case is even.  (Proof:  Use the classification of lattices over $\Z_2$ to reduce to the case that the quadratic form is either $x^2$ or $2^r(ax^2+bxy+cy^2)$ with unit discriminant where $a,c\in\Z_2$ and $b\in\Z_2^\times$.)   As a result the following four cases are exhaustive.

  \begin{itemize}
  \item{Case 1: $D(L)$ contains an element $a$ of order $2^i$, where $i>2$.}  We have $(a,2^i a)\in \Z$, and therefore $(a,a)\in 2^{-i}\Z$.  Now let $x=2^{i-1}a$;  then $x$ has order $2$ in $D(L)$ and $(x,x)=2^{2i-2}(a,a)\in 2^{i-2}\Z\subset 2\Z$ as desired.

  \item{Case 2: $D(L)$ contains two independent elements $a,b$ of
    order $4$.}  Let these be $a, b$.  We then have $4(a,a),4(b,b),4(a,b)\in \Z$.  
    If $4(a,a)$ and $4(b,b)$ are both odd, then
    $(2a+2b,2a+2b) = 4(a,a) + 8(a,b) + 4(b,b)$ is even.
    Thus at least one of $2a, 2a+2b, 2b$ is isotropic in $D(L)[2]$.
    
  \item{Case 3: $D(L)[2]$ contains $(\Z/2\Z)^{\oplus 4}$.}  Let $S$ be the subset of $D(L)[2]$
    consisting of $x$ with $(x,x) \in \Z$;  it is easy to see this is a subgroup.  If $x,y\in D(L)[2]$ do not lie in $S$, each of $(x,x),(y,y)\in \Z+\half$, and then $(x+y,x+y)=(x,x)+2(x,y)+(y,y)\in \Z$.  This shows that $S$ has index at most 2 in $D(L)[2]$, so that $S$ contains $(\Z/2\Z)^{\oplus 3}$.   Let $s_1, s_2, s_3\in S$ be independent elements.
    If any $s_i$ satisfies $(s_i,s_i) \in 2\Z$, it is isotropic.
    If $(s_i,s_j) \in \Z$ for any $i \ne j$, then $(s_i+s_j,s_i+s_j) \in 2\Z$.
    If not, then $(s_i,s_i) \in 1+2\Z$ for all $i$ and $(s_i,s_j) \in 1/2+\Z$
    for all $i \ne j$, so $(s_1+s_2+s_3,s_1+s_2+s_3) \in 2\Z$.  
  
  \item{Case 4: $D(L)\isom (\Z/2\Z)^2$.}  Let $a_1,a_2,a_3$ be the nonzero elements of $D(L)$.  Assume that none of these are isotropic.
      By \cite[Theorem 2.8]{morr} we can embed $L$ into the K3 lattice
      $U^3 + E_8^2$, and the complement is a lattice of signature $(2,2)$ and
      discriminant $4$.  By \cite[Theorem 2.8]{morr} again we can embed
      this into $U^3$, with complement of signature $(1,1)$ and discriminant
      $-4$.  This lattice has the same discriminant form as $L$, by
      applying \cite[Lemma 2.4]{morr} twice.  On the other hand, there are
      only two such lattices $L_1, L_2$, with Gram matrices
      $M_1 = \begin{pmatrix} 0&2\\2&2 \end{pmatrix}$ and
      $M_2 = \begin{pmatrix} 0&2\\2&0 \end{pmatrix}$
      (proof: let $x, y$ be a basis.
      If $(x,x)(y,y) < 0$ then we must have $(x,x) = 2, (y,y) = -2, (x,y) = 0$,
      or the same with $x, y$ switched, and that is the first case.  If not,
      we can take $0 \le (x,x) \le (y,y)$ and $(x,y) \ge 0$, and either we
      are in one of the cases above or we can decrease $(x,x) + (y,y)$ by
      replacing $y$ by $y \pm x$). 
       So $L$, being determined
      by its invariants (\cite[Theorem 2.2]{morr}), is the direct sum
      $M_i \oplus  E_8^{\oplus 2}$ for $i = 1$ or $2$.  In both cases there is a vector in
      $L$ that can be divided by $2$, namely $(x,0)$.
  \end{itemize}
\end{proof}

We have reached the main theorem of this section.

\begin{thm}\label{thm:always-map} Let $\Lambda$ be the Picard lattice of (any) Kummer surface of the form $\Km(E_1\times E_2)$, where $E_1,E_2$ are nonisogenous elliptic curves.  Let $S$ be a K3 surface.  Assume there is an isometry $\Pic S\otimes\Q\isom  \Lambda\otimes\Q$.  Then there exists a finite morphism $S\to\Km(E_1\times E_2)$ for nonisogenous elliptic curves $E_1,E_2$.
\end{thm}

\begin{proof} Since $\Lambda$ has discriminant $-16$, the hypothesis implies that $\abs{\disc \Pic S}$ is a square.  We apply
  Corollary \ref{cor:drop-p} to all odd primes dividing $\disc \Pic S$ to obtain a finite morphism from $S\to S'$, where $S'$ is a K3 surface with $\disc \Pic S'=-4^n$.  By Proposition \ref{prop:always-contained}, there is an embedding of $L'=\Pic S'$ into $L''$, where $L''=II_{1,17}$ is the even unimodular lattice of signature $(1,17)$.  We can factor this embedding as $L'=L_0\subset L_1\subset\cdots\subset L_m=L''$, with $L_i/L_{i+1}\isom \Z/2\Z$ for each $i$.  Successive applications of  \ref{prop:genus-1-exists} give a finite morphism $S'\to S''$, where $\Pic S''\isom L''$.  Finally, by Proposition \ref{prop:to-kum} there is a finite morphism from $S''$ to a Kummer surface of the form $\Km(E_1\times E_2)$.
\end{proof}

We conclude this section with some remarks on potential extensions of Theorem \ref{thm:always-map} to the case of more general Kummer surfaces.  Recall that if $A$ is an abelian surface, the Picard rank of the Kummer surface $\Km(A)$ equals $16+\rank \NS(A)$.  In characteristic 0, the rank of $\Km(A)$ takes one of the values 17, 18, 19, 20.   The cases of rank 19, 20 are easy to deal with.

\begin{rmk}
  If $\rank \Pic S> 18$, then a Shioda-Inose structure on $S$ always exists
  \cite[Theorem 6.3, Corollary 6.4]{morr},
  so there is a map of degree $2$ from $S$ to a Kummer surface.
  Thus the analogue of Theorem \ref{thm:always-map} is true for such~$S$.

\end{rmk}

We now turn to the case of rank 18.  We pose the question of whether a K3 surface $S$ of rank 18 should admit a finite morphism to a Kummer surface $\Km(A)$.  Here, $A$ would have to be an abelian surface admitting endomorphisms by an order $\OO$ in a real (and possibly split) quadratic extension of $\Q$.   Then $\disc \Pic \Km(A)=-16\disc \OO$.  We have already treated the split case $\OO=\Z\times \Z$, which corresponds to the case that $A$ is a product of elliptic curves.   Our methods can also treat the case of $\OO=\Z[\sqrt{2}]$.  

\begin{prop}  Let $S$ be a K3 surface.  Let $\Lambda$ be the Picard lattice of (any) Kummer surface of the form $\Km(A)$, where $A$ is an abelian surface with $\End A = \Z[\sqrt{2}]$.   Assume there is an isometry $\Pic S\otimes\Q\isom \Lambda\otimes\Q$.  Then there exists a finite morphism $S\to\Km(A)$ for an abelian surface $A$ of this sort.
\end{prop}

\begin{proof} (Sketch.)  As in Corollary \ref{cor:drop-p}, we may remove all odd primes from the discriminant of $\Pic S$.  Therefore assume that $\disc \Pic S=-2^{2i+1}$ for some $i\geq 0$.  An even lattice of rank $>1$ may not have discriminant $\pm 2$, so in fact we may assume $i\geq 1$.  

We claim that $\Pic S$ is contained in a lattice of discriminant $-8$.   We argue as in the proof of Proposition \ref{prop:always-contained}:  referring to that proof, if $i\geq 2$ then one of the first three cases always holds and we can always find an overlattice of $\Pic S$ of index 2.   Therefore let us assume that $i=1$, so that $\disc \Pic S = -8$.   Using the results of \cite[Chapter 15]{SPLAG}, we can confirm that there is only one isomorphism class of even lattices with discriminant $-8$ and signature $(1,17)$, namely $D_9\oplus E_7\oplus U$.  We can now assume that $\Pic S$ is isomorphic to this lattice.

Invoking a computation with Hasse-Minkowski invariants, this forces the transcendental lattice of $S$ to be isomorphic to $U \oplus \langle -2 \rangle \oplus \langle 4 \rangle$.  (This is also a consequence of \cite[Proposition 7]{ElkiesKumar}.) Thus by \cite[Corollary 6.2]{morr} there is a Shioda-Inose structure $S\to S'$,  where $S'=\Km(A)$ is the Kummer surface of an $A$ satisfying $T_A\isom U\oplus \langle -2 \rangle \oplus \langle 4 \rangle$.   The N\'eron-Severi lattice  $\NS(A)$ of $A$, being the complement of $T_A$ in $H^2(A,\Z)\isom U^{\oplus 3}$, is isomorphic to $\langle 2 \rangle \oplus \langle -4 \rangle$.  The vector of length 2 in $\NS(A)$ is a principal polarization on $A$, and $\End A=\NS(A)$ is the quadratic order of discriminant $-8$, namely $\Z[\sqrt{-2}]$.


\end{proof}

\begin{rmk}\label{rem:change-signs}
  It is easy to embed $D_9 \oplus E_7 \oplus U$ into the K3 lattice
  $E_8^{\oplus 2} \oplus U^{\oplus 3}$ with complement either
  $U \oplus \langle -2 \rangle \oplus \langle 4 \rangle$ or
  $U \oplus \langle 2 \rangle \oplus \langle -4 \rangle$.  However,
  these need not be distinguished.  Indeed,
  let $L$ be the lattice $\langle 2 \rangle \oplus \langle -4 \rangle$
  and let $x,y$ be the given basis.  Then $L \isom -L$, as one sees by
  changing to the basis $(x+y, 2x+y)$.  (This is a manifestation of the
  fact that $\Z[\sqrt{2}]$ has a unit of norm $-1$.)  
\end{rmk}


Finally we turn to the case of Picard rank 17.

\begin{prop} Let $S$ be a K3 surface.  Let $\Lambda$ be the Picard lattice of (any) Kummer surface of the form $\Km(A)$, where $A$ is an abelian surface with $\End A = \Z$.   Assume there is an isometry $n\Pic S\otimes\Q\isom \Lambda\otimes\Q$, where $n \in \{1,2\}$.  Then there exists a finite morphism $S\to\Km(A)$ for an abelian surface $A$ of this sort.
\end{prop}

\begin{proof} (Sketch.)  Again we apply the techniques of Corollary \ref{cor:drop-p} and \ref{prop:always-contained} to assume that $\disc \Pic S$ is either $-2$ or $-4$.  In these two cases we must have $n = 1$ or $2$ respectively.
  In both cases, the even lattice $\Pic S$ is uniquely determined by its discriminant and signature.

In the case of discriminant $-2$, we have $\Pic S\isom E_8\oplus E_7\oplus U$.   This lattice is rationally isometric to $\Lambda$. The orthogonal complement of $\Pic S$ in $H^2(S,\Z)$ is isomorphic to $A_1\oplus U^{\oplus 2}$.  The existence of a Shioda-Inose structure
  now follows from \cite[Corollary 6.4 (iii)]{morr}.
  
  In the case of discriminant $-4$, $\Pic S\isom L\oplus U$, where $L$ is a lattice containing its root sublattice $L_0$ with index 2, and $L_0\isom A_3\oplus D_{12}$.  Therefore by Proposition \ref{PropExistenceOfEllFib}, there exists an elliptic fibration on $S$ with 2-torsion section.  Let $S\to S'$ be the quotient by the 2-torsion section.  The induced map $T_{S'}\to T_S$ is not generally a rational isometry, but rather $T_{S'}$ is rationally isometric to $2T_{S}$:  see \cite[\S 2.4]{BSV}.  Since the rank of $T_S$ is odd,  
the discriminants of $T_S$ and $T_{S'}$ differ by twice the square of a rational number.   Since $\Pic S$ is the complement of $T_S$ in a unimodular lattice, and similarly for $\Pic S'$, we find that $\disc \Pic S'=-2n^2$ for some integer $n\geq 1$.  Also, we now have a rational
  isometry $\Pic S' \otimes \Q \isom \Lambda \otimes \Q$.  After once again removing squares from the discriminant we are in the case of the previous paragraph.
\end{proof}
\section{Verification of 2-modularity for extremal rational elliptic fibrations}
\label{sec:verification}

The goal of this section is to complete the proof of Theorem \ref{ThmMain2Modularity}:  every nonisotrivial tame extremal rational elliptic fibration $\E\to\P^1$ over a finite field is 2-modular.  In \S\ref{sec:extremalrational} we present an overview and classification of such fibrations, and show that when $\E\to \P^1$ is base changed along a double cover of $\P^1$, a K3 surface arises which is isogenous to a Kummer.  In \S\ref{sec:modular-fibrations} we consider the family of K3 surfaces $\caZ^2(\E)\to U^2$ that arose in the introduction, and prove the existence of an isogeny from $\caZ^2(\E)$ onto $\Km(\E'\times \E')$, where $\E'\to U$ is a 2-modular elliptic fibration of the same conductor as $\E$.  If $\E$ is semistable, this is enough to show that $\E$ and $\E'$ are isogenous, so that $\E$ is 2-modular as well.  The remaining sections feature case-by-case calculations which show that $\E'$ and $\E$ are isogenous in the unstable cases as well.  For those calculations, the theory of Shioda-Inose structures is indispensable.

\subsection{Extremal rational elliptic fibrations and associated K3 surfaces}
\label{sec:extremalrational}


A rational elliptic surface $\E$ over an algebraically closed field $F$  is isomorphic to $\P^2$ blown up at 9 points (possibly infinitely near), so that $\rho(\E)=10$ and $e(\E)=12$. If in addition $\E\to\P^1$ is extremal, and if we also assume that the singular fibers are tame, then by the Shioda-Tate formula \eqref{EqShiodaTate} we must have $\sum_v (m_v-1)=8$, where $m_v$ is the number of irreducible components in the fiber $\E_v$.  On the other hand by the Euler number formula \eqref{EqEulerNumberFormula} we have $\sum_v e(\E_v)=12$, where $e(E_v)$ is $m_v$ or $m_v+1$ as $v$ is multiplicative or additive.   Therefore the singular fibers of $\E\to\P^1$ fall into one of the following three possibilities:
\begin{enumerate}
\item Four multiplicative fibers,
\item Two multiplicative fibers and one additive fiber,
\item Two additive fibers.
\end{enumerate}
The case of two additive fibers can only occur if $\E\to \P^1$ is isotrivial (i.e., has constant $j$-invariant).  We discard this case, and refer to the (1) as the semistable case and (2) as the unstable case.

\begin{rmk}  There do exist nonconstant extremal rational elliptic fibrations with two singular fibers, for example the curve with Weierstrass equation
\[ y^2+txy=x^3-t^5 \]
in characteristic 2 has $j$-invariant $t$ and singular fibers exactly at $t=0,\infty$.  The fiber at $0$ is wild.
\end{rmk}

We present here the classification of semistable extremal rational elliptic fibrations $\E\to \P^1$.  These have exactly four fibers of multiplicative type.  This classification is due to Beauville \cite{Beauville}, to whom we refer the reader for Weierstrass equations and the connection to universal elliptic curves with level structure.   See also \cite{Ito}.

\begin{prop}\label{PropClassificationSemistable}  Let $k$ be an algebraically closed field, and let $\E\to \P^1_k$ be a nonisotrivial semistable extremal rational elliptic fibration.   Then the fibration $\E\to\P^1_k$ is determined up to isomorphism by its configuration of singular fibers.  Up to an automorphism of $\P^1_k$, those configurations appear in the table (grouped by isogeny class):

\begin{center}
\begin{tabular}{|l|l|l|}
\hline
Singular fibers (locations) & Mordell-Weil group & notes \\
\hline \hline
$\I_3,\I_3,\I_3,\I_3$ & $(\Z/3\Z)^2$ & $\car k\neq 3$\\

$\I_1,\I_1,\I_1,\I_9$ & $\Z/3\Z$ & \\
\hline
$(1,\omega,\omega',\infty)$ & & $\omega,\omega'$ roots of $x^2+x+1$ \\
\hline \hline
$\I_2,\I_2,\I_4,\I_4$ & $\Z/4\Z\times\Z/2\Z$ & $\car k\neq 2$\\

$\I_1,\I_1,\I_2,\I_8$ & $\Z/4\Z$ & \\
\hline
$(-1,1,0,\infty)$ & & \\
\hline \hline
$\I_1,\I_2,\I_3,\I_6$ & $\Z/6\Z$ & $\car k \neq 2,3$\\
\hline
$(4,-1/2,0,\infty)$ & &  \\
\hline \hline
$\I_1,\I_1,\I_5,\I_5$ & $\Z/5\Z$ & $\car k \neq 5$\\
\hline
$(\phi,\phi',1,\infty)$ & & $\phi,\phi'$ roots of $x^2-x-1$\\
\hline
\end{tabular}
\end{center}
\end{prop}

The following proposition classifies the tame extremal rational elliptic fibrations in the unstable case.   If such a fibration is nonisotrivial, then it has exactly three singular fibers, two additive and one multiplicative.  As $\Aut\P^1$ acts triply transitively, it is no longer necessary to keep track of the locations of the singular fibers.  We derive the following table from \cite{MirandaPersson} and \cite{Ito} (in characteristics 2,3).

\begin{prop} \label{PropClassificationUnstable} Let $k$ be an algebraically closed field, and let $\E\to \P^1_k$ be an unstable nonisotrivial tame extremal rational elliptic fibration.  Then the fibration is determined up to isomorphism by its configuration of singular fibers.  The possible configurations are listed below, grouped by isogeny class:

\begin{center}
\begin{tabular}{|l|l|l|}
\hline
Singular fibers & Mordell-Weil group & notes \\
\hline \hline
$\I_2^*,\I_2,\I_2$ & $(\Z/2\Z)^2$ & $\car k\neq 2$ \\
$\I_4^*,\I_1,\I_1$ & $\Z/2\Z$ & \\
$\I_1^*,\I_1,\I_4$ & $\Z/4\Z$ & \\
\hline\hline
$\II^*,\I_1,\I_1$ & $0$  & $\car k\neq 2,3$\\
\hline \hline
$\III^*,\I_1,\I_2$ & $\Z/2\Z$ & $\car k\neq 2$\\
$\III,\I_3,\I_6$ & $\Z/6\Z$ & $\car k=3$\\
\hline\hline
$\IV^*,\I_1,\I_3$ & $0$ & $\car k \neq 2,3$\\
\hline\hline
$\IV,\I_2,\I_6$ & $\Z/6\Z$& $\car k = 2$\\
$\IV^*,\I_1,\I_3$ & $\Z/3\Z$ & \\
\hline\hline
$\II,\I_5,\I_5$ & $\Z/5\Z$ &$\car k=5$\\
\hline
\end{tabular}
\end{center}
\end{prop}

\begin{rmk} In characteristic 2, the $\IV,\I_2,\I_6$ fibration is the specialization of the $\I_1,\I_2,\I_3,\I_6$ fibration, and in characteristic 5, the $\II,\I_5,\I_5$ fibration is the specialization of the $\I_1,\I_1,\I_5,\I_5$ fibration.  
\end{rmk}

Fix a nonisotrivial extremal rational elliptic fibration $\E\to \P^1_F$.  
Let $C\to\P^1_F$ be a separable double cover, with $C$ a rational curve; thus $C\to\P^1_F$ is ramified at two points. Assume these points are disjoint from the singular locus of $\E\to\P^1_F$.  Consider the base change
\[ S = C\times_{\P^1} \E.\]
Then $S$ is a K3 surface.  Indeed, $S$ is a double cover of the rational surface $\E$ branched along the sextic described by the union of two cubics (namely, the fibers of $\E$ over the branch points of $C\to \P^1$).  

 Under our hypothesis that $\E\to\P^1$ is rational and extremal, we must have $\rho(\E)=10=2+8$, where the 2 is from the identity $O$ and fiber $F$, and the 8 is from irreducible components of fibers which do not cross $O$.  Considering the elliptic fibration $S\to C$, the contribution to $\rho(S)$ from irreducible components of fibers is $2\cdot 8 = 16$, and so $\rho(S)\geq 2+16=18$.  

We now have a rational map of moduli spaces: (double covers of $\P^1$ branched at 2 points) $\to$ (K3 surfaces of Picard rank $\geq 18$).  Both spaces are 2-dimensional, so we expect the generic double cover $C\to\P^1$ to produce a K3 surface $S$ of rank 18, with Picard lattice not depending on $C$.  

\begin{prop}  \label{PropRationalIsometry}  Assume that $\Pic S$ has rank 18.  
There exists a rational isometry $\Pic S\otimes \Q\isom \Lambda\otimes \Q$, where $\Lambda$ is the Picard lattice of the Kummer surface associated to the product of two nonisogenous elliptic curves. 
\end{prop}

\begin{proof}   We have $\Pic \E\isom E_8\oplus U$.   Let $L\subset \Pic \E$ be the sublattice generated by components of reducible fibers not meeting the identity section, so that $L\oplus U$ embeds into $E_8\oplus U$ with finite cokernel.   Observe therefore that we have embeddings of lattices:  $L^{\oplus 2}\oplus U \subset E_8\oplus L \oplus U \subset E_8^{\oplus 2} \oplus U$.

Now $\Pic S$ contains a finite-index sublattice isomorphic to $L^{\oplus 2}\oplus U$, generated by reducible fibers and the identity section.   By the observation above, $\Pic S\otimes \Q\isom (E_8^{\oplus 2}\oplus U)\otimes \Q$.    A calculation involving Hasse-Minkowski invariants shows that the latter is isomorphic to $\Lambda \otimes \Q$.  
\end{proof}

\subsection{On the construction of 2-modular elliptic fibrations}\label{sec:modular-fibrations}

With this background, we now turn to the question of 2-modularity for a tame extremal rational elliptic fibration $\E\to \P^1$ of conductor $N$ over a finite field.  We have a family of K3 surfaces $\caZ^2(\E)\to U^2$, obtained via base changing $\E\to \P^1$ by a family of double covers $\P^1\to \P^1$ parametrized by $U^2$.  Here is what we
know so far about $\caZ^2(\E)$:
\begin{enumerate}
        \item There exists an isogeny between families of K3 surfaces $\caZ^2(\E)\to \Km(\mathcal{A})$ over $U^2$.  Here $\mathcal{A}\to U^2$ is a family of abelian varieties which, possibly after passing to a double cover of $U^2$, splits as a product of nonisogenous elliptic curves with transcendental $j$-invariant.  
    \item There exists a dominant rational map $\Sht_G(\Gamma_0(N))\dashrightarrow \caZ^2(\E)$ over $U^2$, see \eqref{EqMapFromSht2ToZ2}. 
\end{enumerate}

\begin{rmk} The first statement is obtained by applying \ref{ThmIsogenyBetweenK3s}
to $\caZ^2(\E)$.  Strictly speaking, that theorem only applies over algebraically closed fields, but it can be made to work over the base $U^2$, with the proviso that the splitting of $\mathcal{A}$ into a product of elliptic curves may only happen over an \'etale double cover of $U^2$.  The argument runs this way:  consider the ring scheme $\End \mathcal{A}$ over $U^2$ which classifies endomorphisms of $\mathcal{A}$.  Since 
$\mathcal{A}_{\overline{\eta}_2}$ splits as a product of nonisogenous, transcendental elliptic curves, the $\overline{\eta}_2$-points of $\End \mathcal{A}$ are $\Z\times \Z$.  On the other hand, an appeal to the rigidity lemma shows that the union of those components of $\End \mathcal{A}$ which dominate $U^2$ together form an \'etale group scheme;  thus one can think of this union as a representation of $\pi_1(U^2,\overline{\eta}_2)$ on the ring $\Z\times \Z$.  There are two possibilities.    If the representation is trivial, in which case $\mathcal{A}$ is the product of two families of elliptic curves.  If the representation is nontrivial,   $\pi_1(U^2,\overline{\eta}_2)$ permutes the factors of the $\Z\times\Z$;  in this case $\mathcal{A}$ is the restriction of scalars of a family of elliptic curves over an \'etale double cover of $U^2$.
\end{rmk}

The main goal of the subsection is to prove the following theorem.

\begin{thm} 
\label{ThmASeparates} 
Let $\mathcal{A}\to U^2$ be a family of abelian surfaces. Assume:
\begin{enumerate}
    \item \'Etale-locally on $U^2$, the abelian surface $\mathcal{A}$ is isomorphic to a product of non-isogenous elliptic curves, each of which has transcendental $j$-invariant.  
    \item There exists a dominant rational map from $\Sht_G^2(\Gamma_0(N))$ to $\Km(\mathcal{A})$ lying over $U^2$. 
\end{enumerate}
Then there exists a non-isotrivial family of elliptic curves $\E'\to U$ of conductor bounded by $N$, and an isogeny $\Km(\mathcal{A})\to\Km(\E'\times_{\F_q}
\E')$ over $U^2$.  This $\E'$ is $2$-modular.
\end{thm}

The main players in the proof are the \'etale fundamental groups and their representations.  Let us notate our \'etale fundamental groups as $\pi_1(U)$, $\pi_1(U^2)$, etc., the base point being understood to be $\overline{\eta}$ or $\overline{\eta}_2$.  The K\"unneth theorem fails in characteristic $p$, in the sense that the natural map $\pi_1(U^2)\to \pi_1(U)^2$ is not an isomorphism.  In fact it is neither injective nor surjective;  see the discussion of Drinfeld's Lemma in \cite[\S1.1]{ScholzeWeinstein}.  Write $\overline{U}$ (resp., $\overline{U}^2$) for the base change of $U$ (resp., $\overline{U}^2$) from $\F_q$ to $\overline{\F}_q$.  
The various fundamental groups fit into a diagram with exact rows and columns (here $\Delta$ = diagonal map):
\[
\xymatrix{
& 1 \ar[d] & 1 \ar[d] & & \\
& \mathcal{K} \ar[r]^{=} \ar[d] & \mathcal{K}\ar[d] \ar[r] &1\ar[d] & \\
1 \ar[r] & \pi_1(\overline{U}^2) \ar[d] \ar[r] & \pi_1(U^2) \ar[r] \ar[d] & \Gal_{\F_q} \ar[r] \ar[d]^{\Delta} & 1 \\
1 \ar[r] & \pi_1(\overline{U})^2 \ar[d] \ar[r] & \pi_1(U)^2 \ar[r] \ar[d] & \Gal_{\F_q}^2 \ar[r] \ar[d] & 1 \\
& 1 \ar[r] & \Gal_{\F_q} \ar[r]_{=} \ar[d] & \Gal_{\F_q} \ar[d] \\
&& 1 & 1& 
}
\]

Now consider our family of abelian surfaces $f\from \mathcal{A}\to U^2$.  Let $H^2(\mathcal{A},\Q_\ell)$ denote the fiber of $R^2f_*\Q_\ell$ at the geometric generic point $\overline{\eta}_2$:  this is a representation of $\pi_1(U^2)$.  Within this there is the 4-dimensional transcendental subspace $H^2_{\text{trans}}(\mathcal{A},\Q_\ell)$, meaning the subspace orthogonal to all algebraic classes.

\begin{lemma} \label{LemmaNoFixedVector} The Tate twist $H^2_{\text{trans}}(\mathcal{A},\overline{\Q}_\ell)(1)$ contains no nonzero vector fixed by an open subgroup of $\pi_1(U^2)$.  
\end{lemma}

\begin{proof} Suppose otherwise, and that the nonzero vector is fixed by the fundamental group of a connected finite cover of $U^2$ with fraction field $L$.  After a possible further extension of $L$, the base change $\mathcal{A}_L$ is isomorphic to $E_1\times_L E_2$ for elliptic curves $E_1,E_2$ over $L$.  Then $H^2_{\text{trans}}(\mathcal{A}_L,\Q_\ell)(1)$ is the $\Q_\ell$-linear dual of the tensor product $V_\ell(E_1)\otimes V_\ell(E_2)(-1)$, where $V_\ell$ means the rational Tate module.  This is  in turn isomorphic via the Weil pairing to $\Hom(V_\ell(E_1),V_\ell(E_2))$.  

The situation now is that $\Hom_{\Gal_L}(V_\ell(E_1),V_\ell(E_2))\otimes\overline{\Q}_\ell$ is nonzero.  We now appeal to the isogeny theorem over fields finitely generated over $\F_q$ \cite[Theorem 1.4]{Zarhin} to conclude that $E_1$ and $E_2$ are isogenous over $L$, contrary to our hypothesis about $\mathcal{A}$.
\end{proof}

We have assumed the existence of a dominant rational map from $\Sht_G^2(\Gamma_0(N))$ to $\Km(\mathcal{A})$.  This induces a
map on cohomology
\begin{equation}
\label{EqInducedMapOnCohomology}
{H^2_c(\Sht_G^2(\Gamma_0(N))_{\overline{\eta}_2}, {\mathbf{Q}}_\ell)} 
\to {H^2_{\text{trans}}({{\caA}},
\Q_\ell)}
\end{equation}
which is equivariant for the action of $\pi_1(U^2)$ on either side.   The map in \eqref{EqInducedMapOnCohomology} is surjective:  this reduces to a general fact about dominant rational maps between quasi-projective varieties, see \cite[Proposition 1.2.4]{Kleiman}.

\begin{lemma} The map in \eqref{EqInducedMapOnCohomology} is nonzero when restricted to the cuspidal subspace.
\end{lemma}

\begin{proof}  The cuspidal subspace is the kernel of the ``constant term map'' 
\[ H^2_c(\Sht_G^2(\Gamma_0(N))_{\overline{\eta}_2},\Q_\ell)(1)\to H^0_c(\Sht_M^2(\Gamma_0(N))_{\overline{\eta}_2},\Q_\ell)\]
towards the compactly supported cohomology of a space of shtukas relative to $M\subset G$, where $M\isom \Gm$ is the Levi subgroup. (For the construction of the constant term map, see \cite{Xue}.  The Tate twist (1) appears to compensate for the fact that we have used constant coefficients rather than intersection cohomology on the $G$-shtuka side.)  The cohomology of $M$-shtukas is known by class field theory:  after base extension to $\overline{\Q}_\ell$,  the image of the constant term map decomposes as a $\pi_1(U^2)$-module as direct sum of characters $\chi\boxtimes \chi^{-1}$, where $\chi$ runs over finite-order characters of $\pi_1(U)$ of conductor bounded by $N$.    Note that all such characters become trivial when restricted to an open subgroup.  For the map  \eqref{EqInducedMapOnCohomology} to be trivial on the cuspidal subspace, it would mean that such a direct sum surjects onto $H^2_{\text{trans}}(\mathcal{A},\overline{\Q}_\ell)(1)$, which violates Lemma \ref{LemmaNoFixedVector}.  
\end{proof}

Proposition \ref{PropCohoOfShtukaSpaceIntro} states that the action of $\pi_1(U^2)$ on the cuspidal subspace of $H^2_c(\Sht_G^2(\Gamma_0(N)),\overline{\Q}_\ell)$ extends along the homomorphism $\pi_1(U^2)\to \pi(U)^2$, and decomposes into a direct sum of representations of $\pi(U)^2$ of the form $\sigma^{\boxtimes 2}$, where $\sigma$ runs over irreducible representations of $\pi_1(U)$ with determinant $\overline{\Q}_\ell(-1)$ of conductor bounded by $N$.  Since each $\sigma^{\boxtimes 2}$ is irreducible of dimension 4, we find that there is a particular $\sigma$ for which 
\begin{equation}
    \label{EqIsomPi1}
 H^2_{\text{trans}}(\mathcal{A},\overline{\Q}_\ell) \isom \sigma^{\boxtimes 2}\vert_{\pi_1(U^2)}
 \end{equation}
as representations of $\pi_1(U^2)$.

\begin{lemma} \label{LemmaASplits} The family of abelian surfaces $\mathcal{A}\to U^2$ is isomorphic to $\E_1\times_{U^2} \E_2$, where $\E_1$ and $\E_2$ are families of elliptic curves over $U^2$.
\end{lemma}

\begin{proof} Recall we have assumed that $\mathcal{A}$ splits into such a product over an \'etale cover.  Assume this splitting doesn't occur globally over $U^2$.  Then $\mathcal{A}$ is the restriction of scalars of a family of elliptic curves $\E\to V$ along a connected double cover $V\to U^2$.  Consequently $H^2_{\text{trans}}(\mathcal{A})$ is isomorphic to the {\em {tensor-induced}} representation of $\rho:=H^1(\E_{\overline{\eta}_2},\overline{\Q}_\ell)$ from $\pi_1(V)$ to $\pi(U^2)$.  This can be modeled as the extension of $\rho^{\otimes 2}$ to $\pi_1(U^2)$ determined by the rule $s(v_1\otimes v_2)=v_2\otimes s^2 v_1$, where $s$ is any representative for the nontrivial coset of $\pi_1(V)$ in $\pi_1(U^2)$.  Such an element $s$ always acts nontrivially on $\rho^{\otimes 2}$:  to see this, choose $v_1$ to be an eigenvector of $s^2$, and choose $v_2$ to be linearly independent from $v_1$.  

Recall that $\mathcal{K}$ is the kernel of $\pi_1(U^2)\to \pi_1(U)^2$.  
The isomorphism \eqref{EqIsomPi1} shows that $\mathcal{K}$ acts trivially on $H^2_{\text{trans}}(\mathcal{A},\Q_\ell)$.  By the above observation about the nontrivial action of $s$, it follows that $\mathcal{K}\subset \pi_1(V)$.  Equivalently, there exists a finite connected cover $W\to U$ such that $V$ is intermediate to $W^2\to U^2$.  

Let us now base change the entire story from $U$ to $W$.  We have a family of elliptic curves $\E\to W^2$, such that if $\rho=H^1(\E_{\overline{\eta}_2},\overline{\Q}_\ell)$, then $\rho^{\otimes 2}$ extends along the homomorphism $\pi_1(W^2)\to \pi_1(W)^2$ to a representation of the form $\sigma^{\boxtimes 2}$.  

Let $\mathcal{K}_W$ be the kernel of $\pi_1(W^2)\to \pi_1(W)^2$.  An element of $\mathcal{K}_W$ acts trivially on $\rho^{\otimes 2}$, which implies that it must act as a scalar $\pm 1$ on $\rho$.  Therefore $\mathcal{K}_W$ acts trivially on the projective representation $P\rho$.

We find that the restriction of $P\rho$ to $\pi_1(\overline{W}^2)$ factors through a projective representation 
\[ P\rho\from \pi_1(\overline{W})^2\to \PGL_2(\overline{\Q}_\ell),\]
which is tantamount to two homomorphisms from $\pi_1(\overline{W})$ whose images commute with one another. 
By considering Zariski closures, there are two possibilities, each of which leads to a contradiction:
\begin{itemize}
    \item $P\rho$ is trivial when restricted to one of the factors of $\pi_1(\overline{W})^2$. But then the same would be true of $P(\sigma^{\boxtimes 2})$, which is false.
\item The image of each copy of $\pi_1(\overline{W})$ under $P\rho$ lies in the same torus in $\PGL_2$.  This would imply that the image of $\pi_1(W^2)$ in $\GL_2(\overline{\Q}_\ell)$ lies in the normalizer of a torus.  By the isogeny theorem, this is only possible if $\End\E$ is larger than $\Z$.  This contradicts our hypothesis on $\mathcal{A}$:  its elliptic curve factors have transcendental $j$-invariant.  
\end{itemize}
\end{proof}

By Lemma \ref{LemmaASplits}, our abelian surface $\mathcal{A}$ is isomorphic to a product $\E_1\times_{U^2} \E_2$, where each $\E_i\to U^2$ is a family of elliptic curves.  Let $\rho_i$ be the representation of $\pi(U^2)$ on 
 $H^1(\E_{i,\overline{\eta}_2},\overline{\Q}_\ell)$.  The isomorphism in \eqref{EqIsomPi1} becomes:
 \begin{equation}
     \label{EqIsomPi2}
 \rho_1\otimes \rho_2\isom \sigma^{\boxtimes 2}\vert_{\pi_1(U^2)}.
  \end{equation}

Let $\text{pr}_1,\text{pr}_2$ refer to the projection maps $U^2\to U$.

\begin{lemma} \label{LemmaProjectiveRep}
After possibly relabeling, the projective representation $P\rho_i$ is isomorphic to $\text{pr}_i^*P\sigma$, meaning the pullback of $P\sigma$ along the map $\pi_1(U^2)\to \pi(U)$ induced by $\text{pr}_i$.  
\end{lemma}

\begin{proof}
Since $\mathcal{K}$ acts trivially on $\rho_1\otimes \rho_2$, it must act by scalars on each $\rho_i$, in fact by $\pm 1$ due to the Weil pairing on each $\rho_i$.  
Therefore the projective representations $P\rho_i$ are trivial on $\mathcal{K}$, and so $P\rho_i\vert_{\pi_1(\overline{U}^2)}$ factors through a projective representation of $\pi_1(\overline{U})^2$.  As such, $P\rho_i$ has to be trivial on one of the factors of $\pi_1(\overline{U})^2$.  (The argument is similar to that  given in the proof of Lemma \ref{LemmaASplits}.) 

Considering \eqref{EqIsomPi2}, the only possibility (after possibly relabeling) is that $P\rho_i\vert_{\pi_1(\overline{U}^2)}\isom \text{pr}_i^*P\sigma\vert_{\pi_1(\overline{U})}$. 
\end{proof}

\begin{lemma} \label{LemmaDescentToU} There are families of elliptic curves $\E_1',\E_2'\to U$ and characters 
$\chi_1,\chi_2\from \pi_1(U^2)\to \set{\pm{1}}$ such that $\E_i\isom \text{pr}_i^*\E_i'\otimes \chi_i$ for $i=1,2$.  
\end{lemma}

\begin{proof} We run the argument for $\E_1$ only.  For the proof it will be convenient to distinguish the two copies of $U$:  $U^2=U_1\times U_2$.  We want to show that $\E_1$ is isotrivial relative to $U_2$.  

Let $\overline{\eta}_1\to U_1$ be a geometric generic point.  The base change of $\E_1$ to $\overline{\eta}_1\times U_2$ induces an action of the monodromy group 
$\pi_1(\overline{\eta}_1\times U_2)$ on $\rho_1$;  this action factors through the map $\pi_1(\overline{\eta}_1\times U_2)\to\pi_1(U_1\times U_2)$.  Lemma \ref{LemmaProjectiveRep} shows that this monodromy acts as a scalar, since the composition 
\[ \pi_1(\overline{\eta}_1\times U_2)\to \pi_1(U_1\times U_2)\to \pi_1(U_1) \]
is the identity.  This is enough to show that $\E_1$ is isotrivial relative to $U_2$:  its $j$-invariant lies in the coordinate ring of $\overline{\eta}_1$.  But the $j$-invariant already lies in the coordinate ring of $U_1\times U_2$, so it must lie in the coordinate ring of $U_1$.  

 Let $\E_1'\to U_1$ be a family of elliptic curves with the same $j$-invariant;  there exists a (possibly trivial) character $\chi_i\from \pi_1(U^2)\to \set{\pm 1}$ such that $\E_1$ is isomorphic to the twist $\E_1'\otimes \chi_i$.
\end{proof}

The following lemma completes the proof of Theorem \ref{ThmASeparates}:

\begin{lemma}  There is a family of elliptic curves $\E'\to U$ of conductor bounded by $N$ and an isogeny $\Km(\E_1\times_{U^2} \E_2)\to \Km(\E'\times_{\F_q} \E')$.  
\end{lemma}

\begin{proof} Let $\rho_i'$ be the representation of $\pi_1(U_i)$ on $H^1(\E_i',\overline{\Q}_\ell)$.  According to Lemma \ref{LemmaDescentToU}, $\rho_i\isom \text{pr}_i^*\rho_i'\otimes \chi_i$ as a representation of $\pi_1(U^2)$.  
The isomorphism \eqref{EqIsomPi2} now reads
\[ (\rho_1'\boxtimes \rho_2')\otimes \chi_1\chi_2 \isom \sigma^{\boxtimes 2}\vert_{\pi_1(U^2)} \]    

Since $\mathcal{K}$ acts trivially on the external tensor products $\rho_1\boxtimes \rho_2$ and $\sigma^{\boxtimes 2}$, it lies in the kernel of $\chi_1\chi_2$.  Therefore $\chi_1\chi_2$ factors through the image of $\pi_1(U^2)\to \pi_1(U)^2$.  It is possible to extend $\chi_1\chi_2$ to a character $\chi_1'\boxtimes\chi_2'$ of $\pi_1(U)^2$ valued in $\set{\pm 1}$ such that
\[ (\rho_1'\boxtimes \rho_2')\otimes (\chi_1'\boxtimes\chi_2') \isom \sigma\boxtimes \sigma, \]
in which case $\rho_1'\otimes \chi_1'\isom \rho_2'\otimes \chi_2'\isom\sigma$.  By the isogeny theorem, $\E_1'\otimes\chi_1'$ and $\E_2'\otimes\chi_2'$ are isogenous to the same family of elliptic curves $\E'\to U$, which has conductor bounded by $N$ (since $\sigma$ does).  

Since $\chi_1'\boxtimes \chi_2'$ 
extends $\chi_1\chi_2$, 
the products $\chi_1\pr_1^*\chi_1'$ 
and $\chi_2\pr_2^*\chi_2'$ 
represent the same character on $\pi_1(U^2)$;  
call this common character $\chi$.  
Now note that the formation of the Kummer surface of a product of elliptic curves is insensitive to twisting both curves by a common quadratic character: 
\[
\Km(\E_1\times_{U^2} \E_2)\isom \Km((\E_1\otimes \chi)\times_{U^2} (E_2\otimes\chi))\isom 
\Km((\E_1'\otimes\chi_1')\times_{U^2} (\E_2'\otimes\chi_2')). \]
The latter is isogenous to $\Km(\E'\times_{\F_q} \E')$.  
\end{proof}

Return now to the situation of Theorem \ref{ThmMain2Modularity}.  Let $\E\to \P^1$ be a tame extremal rational elliptic fibration over a finite field.  Then on the one hand $\mathcal{Z}^2(\E)\to U^2$ is a family of K3 surfaces admitting a dominant rational map from $\Sht_G^2(\Gamma_0(N))$, and on the other, there is a finite map $\mathcal{Z}^2(\E)\to \Km(\mathcal{A})$ for a family of abelian surfaces $\mathcal{A}\to U^2$ which splits \'etale-locally as a product of nonisogenous elliptic curves with transcendental $j$-invariant.  Theorem \ref{ThmASeparates} applies, and we find a 2-modular elliptic fibration $\E'\to U$ of conductor bounded by $N$.  Since $N$ has degree 4 and $\E'$ is nontrivial, the conductor of $\E'$ must be exactly $N$.   

We claim that $\E$ and $\E'$ are isogenous over $U$, which would imply that $\E$ is 2-modular as well.  In the semistable cases, there is only one isogeny class of 
conductor $N$ for $p$ large enough, so the isogeny is automatic for those primes.  For the remaining primes, we argue this way:  The objects $\E$ and  $\mathcal{Z}^2(\E)$ can be defined in characteristic 0;  we have found an isogeny between $\mathcal{Z}^2(\E)$ and $\Km(\E'\times\E')$, where $\E'$ is an elliptic fibration in characteristic 0.  After replacing $\E'$ with an isogenous fibration, we have shown that $\E$ and $\E'$ are isomorphic modulo almost all primes;  this shows immediately that they are isomorphic at all primes of good reduction.  

In the unstable cases, there are multiple elliptic fibrations with the same conductor, and we could not find a theoretical argument for why $\E$ and $\E'$ should be isogenous.  In all those cases however we were able to find an explicit isogeny.  

The remaning material in this section consists of calculations performed on each isogeny class of tame extremal rational elliptic fibrations, indicating how to find the required isogeny $\caZ(\E^2)\to \Km(\E\times\E)$.

\subsection{The $\I_2^\ast\I_2\I_2$ (Legendre) fibration, with Mordell-Weil group $(\Z/2\Z)^2$}

The simplest case of Theorem \ref{ThmIsogenyFromZtoKm} concerns the Legendre fibration $\E\to\P^1_t$, with equation
\begin{equation}
\label{EqLegendre}
 y^2=x(x-1)(x-t) 
\end{equation}
over a field $k$ of characteristic $\neq 2$.  Then $\E\to \P^1_t$ has singular fibers $\I_2^*,\I_2,\I_2$ at $t=\infty,0,1$, with split multiplicative reduction at $t=1$.  The conductor is $N=(0)+(1)+2(\infty)$, and the Mordell-Weil group is $(\Z/2\Z)^2$.  Let $U=\P^1\smallsetminus \set{0,1,\infty}$.   We write $t_1,t_2$ for the coordinates on $U^2$.

Let $\Sigma_\infty=\set{1}$ (this choice is unimportant);  we computed in Example \ref{ex:coinc-3} that the space of coincidences $\Coinc^3(\Gamma_0(N);\Sigma_\infty)\isom \P^1_s\times U^2$ is a double cover of $\P^1_t\times U^2$, with equation
\begin{equation}
\label{EqTInTermsOfS}
 t = \frac{s(s-t_1-t_2+1)}{s-t_1t_2} 
 \end{equation}
 (obtained by solving for $t=t_3$ in \eqref{EqMinPolyForCoinc3}).  

The elliptic fibration $\caZ^2(\E)\to \P^1_s\times U^2$ may be defined by substituting \eqref{EqTInTermsOfS} into \eqref{EqLegendre}.  For each pair $(t_1,t_2)\in U^2$, it has singular fibers of type $\I_2^*,\I_2^*,\I_2,\I_2,\I_2,\I_2$ at $s=\infty,t_1t_2,0,t_1,t_2$, and $\delta:=t_1+t_2-1$, respectively.  Generically, the Mordell-Weil group is again $(\Z/2\Z)^2$.  From this we conclude that the Picard lattice of the generic fiber of $\caZ^2(\E)\to U^2$ has rank 18 and determinant $-16$, which agrees with the corresponding data for the lattice of $\Km(A)$, where $A$ is the product of two nonisogenous elliptic curves. In fact the two lattices are isomorphic, suggesting that we can find an isomorphism between $\caZ^2(\E)$ and such a Kummer surface.  This is in fact the case:  there is an {\em isomorphism} (not just a finite rational map) of K3 surfaces over $U^2$:
\[ \caZ^2(\E) \isomto \Km(\E^2)  \]

To find the isomorphism, we should look for an elliptic fibration on $\caZ^2(\E)$ with four $\I_0^*$ fibers.  One $\I_0^*$ configuration can be found within the union of the identity section $\sigma_0$, together with two of the $\I_2^*$s and two of the $\I_2$s, shown here as $2\sigma_0+a_1+a_2+b_1+b_2$:


\begin{center}
\includegraphics[scale=0.75]{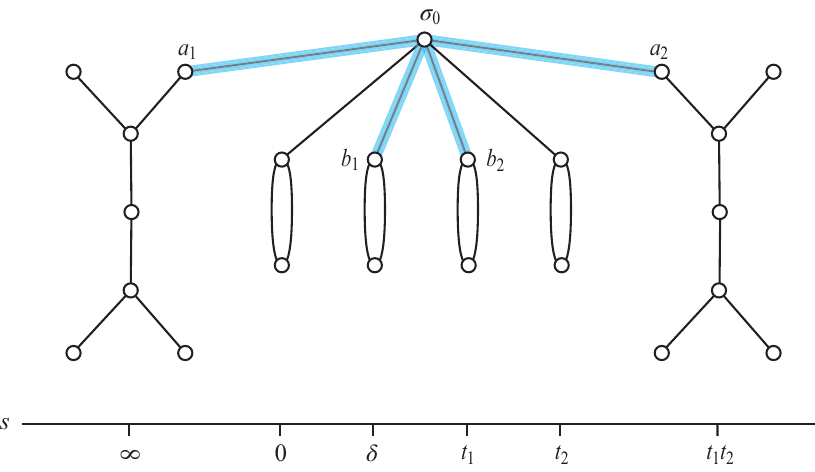}
\end{center}

By Proposition \ref{PropExistenceOfEllFib}(3), there exists an elliptic fibration $\caZ^2(\E)\to \P^1$ containing this $I_0^*$ as a fiber.  
We write it down explicitly:  Introduce a function $w\from \caZ^2(\E)\to \P^1$ by the substitutions (all are equivalent):
\begin{eqnarray*}
x&=&-\frac{(s-\delta)(s-t_1w)}{(s-t_1t_2)(w-1)} \\
x-1&=& -\frac{(s-t_1)(s-t_2-(t_1-1)w)}{(s-t_1t_2)(w-1)}\\
x-t&=&-\frac{(s-\delta)(s-t_1)w}{(s-t_1t_2)(w-1)} 
\end{eqnarray*}

Substituting the above into $y^2=x(x-1)(x-t)$ yields (after absorbing square factors into $y$):
\[ y^2 =-w(w-1)(s-t_1t_2)(s-t_1w)(s-t_2-(t_1-1)w), \]
and then $s=-(w-t_2)u+t_1w$ brings this into the form
\[ y^2=w(w-1)(w-t_2)u(u-1)(u-t_1),\]
which is the Kummer surface $\Km(\E^2)$.  

We used software to find rather ungainly proofs of Theorem \ref{ThmIsogenyFromZtoKm} for some other elliptic fibrations.  But then Masato Kuwata recognized that in all those cases $\caZ^2(\E)$ is related to the {\em Inose surface} of $\E^2$.  He graciously explained to us the connection in the following interlude.

\subsection{Interlude by Masato Kuwata:  Kummer surfaces and Inose surfaces}
\label{sec:MasatoAppendix}

Throughout this note, the base field $k$ is a field of characteristic different from~$2$.

Let $E_{1}$ and $E_{2}$ be two elliptic curves defined over~$k$.
Let $\iota$ be the inversion map $(P,Q)\mapsto (-P,-Q)$ on $E_{1}\times E_{2}$, and let $\bar S=E_{1}\times E_{2}/\class{\iota}$ be the quotient by $\iota$.  The surface $\bar S$ has sixteen double points corresponding to the $2$-torsion points of $E_{1}\times E_{2}$. 
The Kummer surface associated with the product $E_{1}\times E_{2}$, denoted by $\Km(E_{1}\times E_{2})$, is defined as the smooth surface obtained by blowing up these double points. 

There are two obvious elliptic fibrations on $\Km(E_{1}\times E_{2})$ corresponding to the projections $E_{1}\times E_{2}\to E_{i}$ ($i=1,2$).  
\[
\begin{tikzcd}[column sep = small]
             & \hbox to 0pt{\hss $\bar S = $} E_{1}\times E_{2}\hbox to 0pt {$/\class{\iota}$\hss} \arrow[ld, "\pi_{1}"'] \arrow[rd, "\pi_{2}"] &              \\
{\hbox to 0pt{\hss$E_{1}/\{\pm 1\}\simeq$}\P^{1}} &                                                                        & {\P^{1}\hbox to 0pt{$\simeq E_{2}/\{\pm 1\}$\hss}}
\end{tikzcd}
\]

It is known (Oguiso \cite{Oguiso}) that there are eleven different types of elliptic fibrations on $\Km(E_{1}\times E_{2})$ if $k$ is algebraically closed.  Generally, most of these fibrations (excepting the two above) will not be defined over $k$.

``Inose's pencil'' \cite{Kuwata-Shioda} is a genus 1 fibration on $\Km(E_1\times E_2)$ which is always defined over $k$.  To define it, 
choose Weierstrass equations of $E_{1}$ and $E_{2}$:
\begin{align*}
E_{1}: y^2 = x^3 + a_{2}x^2 + a_{4}x + a_{6},\\
E_{2}: y^2 = x^3 + a'_{2}x^2 + a'_{4}x + a'_{6},
\end{align*}
Then an affine model of $\Km(E_{1}\times E_{2})$ is given by
\begin{equation}\label{eq:affine}
(z^3 + a_{2}z^2 + a_{4}z + a_{6})t^{2}
= x^3 + a'_{2}x^2 + a'_{4}x + a'_{6}.
\end{equation}
This equation can be viewed as a cubic curve in $x$, $z$ over the function field $k(t)$.  
The map $\Km(E_{1}\times E_{2})\to \P^{1}$ given by $(x,z,t)\mapsto t$ defines a genus 1 fibration.  This is Inose's pencil.  It does not admit a section unless $E_1$ or $E_2$ has a $k$-rational 2-torsion point.  

Let $J\to\P^1$ be the Jacobian of Inose's pencil \eqref{eq:affine}.  Using the formula in \cite{Artin-et-al:Jacobian}, we obtain the Weierstrass equation for $J$ as follows:
\begin{multline}\label{eq:Weierstrass}
Y^2 = X^3 + 4a_{2} a'_{2} X^2
+ 16(a_{2}^2 a'_{4} - 3 a_{4}a'_{4} + a_{4} a'_{2}{}^2) X
\\
+ \bigl(\Delta_{E_{1}} t^{2} -
(c_{6} a'_{6} - 32 a_{2} a_{4} a'_{2} a'_{4} + 864 a_{6} a'_{6}
+ c'_{6} a_{6} )
+ \Delta_{E_{2}}t^{-2}\bigr),
\end{multline}
where
\begin{align*}
&\Delta_{E_{1}} = -16(4 a_{2}^3 a_{6} - a_{2}^2 a_{4}^2 - 18 a_{2} a_{4} a_{6} + 4 a_{4}^3 + 27 a_{6}^2),\\
&\Delta_{E_{2}} = -16(4 a'_{2}{}^3 a'_{6} - a'_{2}{}^2 a'_{4}{}^2 - 18 a'_{2} a'_{4} a'_{6} + 4 a'_{4}{}^3 + 27 a'_{6}{}^2),
\\
&c_{6} =-32(2 a_{2}^3 - 9 a_{2} a_{4} + 27 a_{6}),\quad
c'_{6} = -32 (2 a'_{2}{}^3 - 9 a'_{2} a'_{4} + 27 a'_{6}).
\end{align*}

An alternate way to obtain \eqref{eq:Weierstrass} is to substitute $t=(t')^3$ into \eqref{eq:affine}, and then use the rational point $(1:(t')^2:0)$ to convert \eqref{eq:affine} into Weierstrass form (see \cite[\S2.1]{Kumar-Kuwata:Rk18}).

Following \cite{Kumar-Kuwata:Rk18}, the {\em Inose surface} $\Ino(E_{1}\times E_{2})$ may be defined as the quotient of \eqref{eq:Weierstrass} by the involution $t\mapsto -t$.  It has equation:
\begin{multline}\label{eq:Inose}
\Ino(E_{1}\times E_{2}): Y^2 = X^3 + 4a_{2} a'_{2}  X^2
+ 16(a_{2}^2 a'_{4} - 3 a_{4}a'_{4} + a_{4} a'_{2}{}^2)  X
\\
+ \bigl(\Delta_{E_{1}} T -
(c_{6} a'_{6} + 32 a_{2} a_{4} a'_{2} a'_{4} - 864 a_{6} a'_{6}
+ c'_{6} a_{6} )
+ \Delta_{E_{2}}T^{-1}\bigr),
\end{multline}
where $T=t^{2}$ is the parameter.  
It has two $\II^{*}$ fibers, at $T=0$ and $\infty$;  all other singular fibers are irreducible.  It is known that, if $E_{1}$ and $E_{2}$ are not isogenous to each other, the transcendental lattice of the Inose surface is isomorphic to $U\oplus U$, where $U=\left(\begin{smallmatrix}0 & 1 \\ 1 & 0\end{smallmatrix}\right)$ is the hyperbolic plane, whereas the transcendental lattice of the Kummer surface is isomorphic to $U(2)\oplus U(2)$.

There is an involution on  $\Ino(E_{1}\times E_{2})$ given by $i:(X,Y,T)\mapsto (X,-Y,\Delta_{E_{2}}/\Delta_{E_{1}}T)$.  The quotient by this involution has equation:
\begin{multline}\label{eq:InoseKummer}
\Ino(E_{1}\times E_{2})/\class{i}: 
Y^2 = X^3 + 4a_{2} a'_{2} (S^2 - 4\Delta_{E_{1}}\Delta_{E_{2}}) X^2
\\
+ 16(a_{2}^2 a'_{4} - 3 a_{4}a'_{4} + a_{4} a'_{2}{}^2)  
(S^2 - 4\Delta_{E_{1}}\Delta_{E_{2}})^{2} X
\\
-
(16 S + c_{6} a'_{6} - 32 a_{2} a_{4} a'_{2} a'_{4} + 864 a_{6} a'_{6}
+ c'_{6} a_{6}  )(S^2 - 4\Delta_{E_{1}}\Delta_{E_{2}})^3.
\end{multline}
(Note the sign $-Y$ in the definition of the involution.  Without it, the resulting quotient would be a rational surface.)
In fact $\Ino(E_1\times E_2)/\class{i}$ is isomorphic to $\Km(E_1\times E_2)$, and the quotient map is nothing but the rational map $\pi_{2}$ Shioda and Inose \cite{Shioda-Inose} used to construct the so-called Shioda-Inose structure:
\[
\begin{tikzcd}[column sep = tiny]
\Ino(E_{1}\times E_{2}) \arrow[rd, "\pi_2"' near start, dashed] &          & E_{1}\times E_{2} \arrow[ld, "\pi_1" near start, dashed] \\
                               &\Km(E_{1}\times E_{2}) &                                 
\end{tikzcd}
\]
Here $\pi_1$ and $\pi_2$ are each dominant rational maps of degree 2.  
The isomorphism between the quotient $\Ino(E_{1}\times E_{2})/\class{i}$ and $\Km(E_{1}\times E_{2})$ is defined only over an extension of $k$ containing all the $2$-torsion points of both $E_{1}$ and $E_{2}$.

In general, there is always a $k$-rational morphism $\Km(E_1\times E_2)\to \Ino(E_1\times E_2)$ of degree 8, using the degree 2 multisection of Inose's pencil to get a degree 4 morphism $\Km(E_1\times E_2)\to J$, followed by the degree 2 morphism $J\to \Ino(E_1\times E_2)$.  

\subsection{The remaining unstable fibrations}

We consider one from each isogeny class.

\subsubsection{The $\II^*\I_1\I_1$ fibration, with Mordell-Weil group $0$}
\label{subsec:II*}
Assume that $\car k\neq 2,3$.  Consider the elliptic fibration $\E\to \P^1_k$ defined by the Weierstrass equation:
\begin{equation}
\label{EqEII}
y^{2}=x^3 - 3x - 2(2t - 1)
\end{equation}
The singular fibers of $\E\to\P^1_k$ are of type $\I_1,\I_1,\II^\ast$ at $0,1,\infty$, respectively.

\begin{prop} There is an isomorphism $\caZ^2(\E)\isom \Ino(\E^2)$ of varieties over $U^2$.
\end{prop}

\begin{proof}  For convenience we work over the generic fiber $\eta_2=\Spec k(t_1,t_2)$ of $U^2$.  We have $\E^2_{\eta_2}=E_1\times E_2$, where for $i=1,2$,  $E_i/k(t_1,t_2)$ is the elliptic curve
\[ E_i: y^2=x^3-3x-2(2t_i-1). \]
The Inose surface $\Ino(E_{1}\times E_{2})$ has equation
\[
\Ino(E_{1}\times E_{2}):
Y^2 = X^3 - 3X + 2 \Bigl( 2 t_{1} (t_{1} - 1) T 
- (2 t_{1} - 1) (2 t_{2} - 1) 
+ \frac{2 t_{2} (t_{2} - 1)}{T}\Bigr),
\]
with $\II^\ast$ fibers at $T=0,\infty$.
On the other hand the surface $\caZ^2(\E)$ is obtained by substituting 
%
%
\begin{equation}\label{eq:base-change-3}
t = \frac{s(s - t_{1} - t_{2} + 1)}{s - t_{1} t_{2}}
\end{equation}
into \eqref{EqEII}, giving
\begin{equation}
\label{eq:tau}
Y^{2} = X^3 - 3(s - t_{1} t_{2})^4 X 
-2\bigl(2s^2 - (2 t_{1} + 2 t_{2} - 1) s + t_{1} t_{2}\bigr) (s - t_{1} t_{2})^5
\end{equation}
Like $\Ino(E_1\times E_2)$, the fibration $\caZ^2(\E)$ also has two $\II^\ast$ fibers, located at $s=\infty,t_1t_2$. 
It is now easy to see that the linear transformation
\[
s =\frac{t_{2} (T t_{1} - t_{2} + 1)}{T}
\]
transforms \eqref{eq:tau} to $\Ino(E_{1}\times E_{2})$.
\end{proof}

\subsubsection{The $\III^{*}\I_2\I_1$ fibration, with Mordell-Weil group $\Z/2\Z$}\label{subsec:III*}
Assume that $\car k\neq 2$.  
Let $\E\to \P^1_k$ be the elliptic fibration with Weierstrass equation:
\begin{equation}
\label{EqEIII}
y^2 = x(x^2-2x+t) 
\end{equation}
The singular fibers of $\E\to\P^1_k$ are of type $\I_2,\I_1,\III^\ast$ at $t=0,1,\infty$, respectively.  The Mordell-Weil group is $\Z/2\Z$.

\begin{prop} There is a finite morphism $\caZ^2(\E)\to \Ino(\E^2)$ of degree 4 of varieties over $U^2$.
\end{prop}

\begin{proof}
The surface $\caZ^2(\E)$ obtained by substituting \eqref{EqTInTermsOfS} into \eqref{EqEIII} is an elliptic fibration 
with singular fibers of type $\III^*,\III^*,\I_2,\I_2,\I_1,\I_1$ at $s=\infty,t_1t_2,0,\delta=t_1+t_2-1,t_1,t_2$, respectively, with Mordell-Weil group again $\Z/2\Z$.  
Let $\sigma_0$ and $\sigma_1$ be the identity and 2-torsion section, respectively.  We display here the $\III^*$ and $\I_2$ fibers along with $\sigma_0$ and $\sigma_1$:

\begin{center}
\includegraphics[scale=.75]{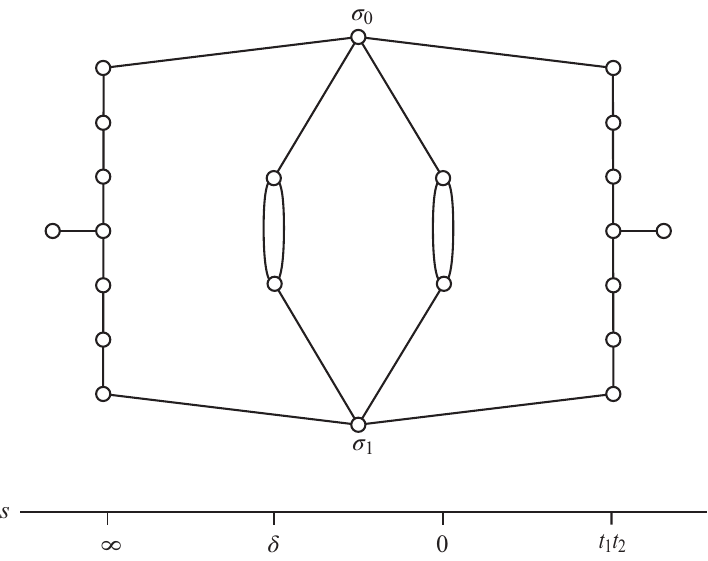}
\end{center}

This configuration contains two $\II^*$ fibers, highlighted below:

 \begin{center}
 \includegraphics[scale=.75]{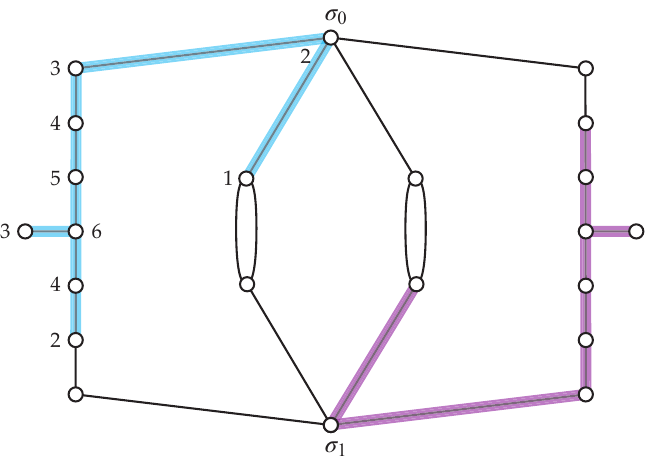}
\end{center}

By Proposition \ref{PropExistenceOfEllFib}, there exists an elliptic fibration $\caZ^2(\E)\to \P^1$ with two $\II^\ast$ fibers.  
To compute it explicitly, we use the so-called ``$2$-neighbor step'' developed by Noam Elkies.  For comprehensible accounts, see \cite{Kumar:2014}, \cite{Sengupta:2017}, \cite{Utsumi:2012}. 
Since the divisor of the function $x$ equals $2(\sigma_{1})-2(\sigma_{0})$, the pole of the function
\[ 
 w = \frac{x}{s-\delta}
\]
coincides with the divisor colored in blue. Thus, $w$
defines a genus 1 fibration on $\caZ^2(\E)$ with these $\II^\ast$ fibers at $w=0$ and $w=\infty$, with equation
\begin{equation}
\label{EqHyperellipticQuartic}
y^{2}=w s^4 - 3 t_{1}t_{2} w s^3 + (3 t_{1}^2 t_{2}^2 w - 2 w^{2}) s^2 
- (t_{1}^3 t_{2}^3 w - 4 t_{1} t_{2} w^{2} - w^3) s 
-  2 t_{1}^2 t_{2}^2 w^2 - (t_{1} + t_{2} - 1) w^{3}
\end{equation}
This fibration has no section, but it does have a multisection $D$ of degree 2.  Indeed, any of the uncolored vertices in the figure represents a curve which meets each fiber with multiplicity 2.   Therefore there is a degree 4 morphism from $\caZ^2(\E)$ onto the Jacobian $J$ of the fibration.   Standard formulas supply the Weierstrass equation for $J$ in terms of the coefficients of the quartic in \eqref{EqHyperellipticQuartic}.  After a further substitution $w=-t_2^2(t_2-1) T$, this Weierstrass equation becomes
\begin{equation}
\label{EqInoOfIII}
Y^2=X^3+4X^2+(4t_1+4t_2-3t_1t_2)X \\
+(-t_1^2(t_1-1)T+2t_1t_2-t_2^2(t_2-1)T^{-1}),
\end{equation}
which we recognize as the equation for $\Ino(E_1\times E_2)$.  
\end{proof}

In this case, each $E_i$ has a 2-torsion point $(0,0)$, so the Inose pencil \eqref{eq:affine} on $\Km(E_1\times E_2)$ admits a section, and there is a double cover $\Km(E_1\times E_2)\to \Ino(E_1\times E_2)$. 

\subsubsection{The $\IV^*\I_3\I_1$ fibration, with Mordell-Weil group $\Z/3\Z$}\label{subsec:IV*}
Assume that $\car k\neq 2,3$.   Let $\E\to \P^1_k$ be the elliptic fibration with Weierstrass equation
\begin{equation}
\label{EqEIV}
 y^2 = x^3 + 9 x^2 + 24 t x + 16 t^2. 
 \end{equation}
The singular fibers of $\E\to\P^1_k$ are of type $\I_3,\I_1,\IV^{\ast}$ at $t=0,1,\infty$, respectively.  The Mordell-Weil group is $\Z/3\Z$, generated by $(0,4t)$.  

\begin{prop} There is a morphism $\caZ^2(\E)\to \Ino(\E^2)$ of degree 9 of varieties over $U^2$.
\end{prop}

\begin{proof} 
The surface $\caZ^2(\E)$ is obtained by substituting $t=s(s-t_1-t_2+1)/(s-t_1t_2)$ into \eqref{EqEIV}, giving
\begin{equation}
\label{eq:base-ch-IV}
y^{2} = x^{3} +9  (s -t_{1} t_{2})^2 x^{2} 
+ 24 s (s - \delta) (s - t_{1} t_{2})^3 x+16 s^2 (s - \delta)^2 (s - t_{1} t_{2})^4,
\end{equation}
It has singular fibers of type $\IV^\ast,\IV^\ast,\I_3,\I_3,\I_1,\I_1$ at $s=\infty,t_1t_2,0,t_1+t_2-1,t_1,t_2$, respectively, with Mordell-Weil group again $\Z/3\Z$.  Let $\sigma_0$ be the identity section and let $\sigma_1$ be one of the 3-torsion sections.   We display here the $\IV^\ast$ and $\I_3$ fibers along with $\sigma_0$ and $\sigma_1$:

\begin{center}
\includegraphics[scale=.75]{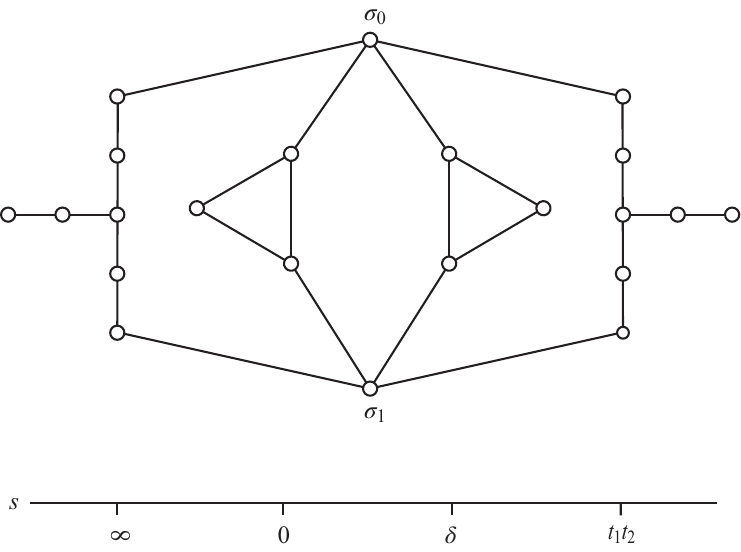}
\end{center}

This configuration contains two $\II^\ast$ fibers, highlighted below:
\begin{center}
\includegraphics[scale=.75]{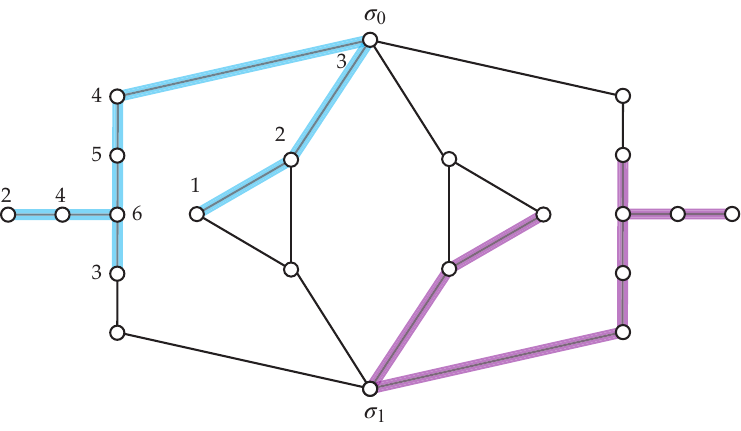}
\end{center}

Therefore there exists a genus 1 fibration $\caZ^2(\E)\to \P^1\times U^2$ with two $\II^{*}$ fibers.  
To perform a $3$-neighbor step, we first find the tangent line at each point of the $3$-torsion section $\sigma_{1}=\bigl(0,4s(s - t_{1} - t_{2} + 1)(s - t_{1}t_{2})^2\bigr)$:
\[
y=3(s-t_{1}t_{2})x + 4s(s - t_{1} - t_{2} + 1)(s-t_{1}t_{2})^2.
\]
Using this, we find an elliptic parameter 
\begin{equation}
\label{paramIV}
w = -\frac{t_{2}^3(t_{2} - 1)}{8}\cdot
\frac{y - 3(s-t_{1}t_{2})x - 4s(s - t_{1} - t_{2} + 1)(s-t_{1}t_{2})^2}{s^2(s-t_{1}t_{2})^4},
\end{equation}
which leads to a cubic curve in $x'$ and $s$ with parameter~$w$:
\begin{equation}
\label{EqIVCubic}
t_{2}^6 (t_{2} - 1)^2 {x'}^3 + 12 t_{2}^3 (t_{2} - 1) w (s -t_{1} t_{2}) x'  + 8 t_{2}^3 (t_{2} - 1) w (s - t_{1} - t_{2} + 1) - 8 w^2 s (s - t_{1} t_{2})^2,
\end{equation}
where $x'=2s(s-t_{1}t_{2})^2 x$.

This fibration lacks a section, but it does have a multisection of degree~$3$.  Indeed, any of the uncolored vertices in the figure represents a curve which meets each fiber with multiplicity~$3$.  Therefore there exists a degree~$9$ map from $\caZ^2(\E)$ to the Jacobian $J$ of the fibration.  The Weierstrass equation for $J$ is\[
Y^2 = X^3 - 3  (8 t_{2} - 9) (8 t_{1} - 9)X 
+ 2 (32 t_{1}^3 (t_{1} - 1) w - (8 t_{2}^2 - 36 t_{2} + 27) (8 t_{1}^2 - 36 t_{1} + 27) + 32 t_{2}^3 (t_{2} - 1) w^{-1}),
\]
which coincides with the twist of $\Ino(\E^2)$ by~$-3$.%

\end{proof}

\subsection{The semistable fibrations}

The verification of Theorem \ref{ThmIsogenyFromZtoKm} for the semistable extremal rational elliptic fibrations is more difficult, since now one must keep track of the locations of the singular fibers of $\E\to\P^1$.  Each time, we found a genus 1 fibration on $\caZ^2(\E)$ with two $\II^\ast$ fibers, which means we can apply the techniques of \S\ref{sec:MasatoAppendix}, and conclude that there is a product of elliptic curves $\E_1\times \E_2\to U^2$ and a finite morphism $\caZ^2(\E)\to \Km(\E_1\times \E_2)$ commuting with the maps to $U^2$.  We demonstrate this fact with the figures that follow.

However, we were not able to directly compute $\E_1\times\E_2$ in all cases.  
The trouble is that in two of the cases, the genus 1 fibration on
$\caZ^2(\E)$ with two $\II^\ast$ fibers has a multisection of degree 5
(resp., 6).  There is currently no explicit formula for the Weierstrass
equation of the Jacobian of a genus 1 curve with such a multisection.
(Compare with \cite{Artin-et-al:Jacobian}, which has formulas
for the Jacobian of a plane cubic, and with \cite{akmmmp}, which shows
how to proceed for the intersection of two quadrics in $\P^3$.)
Fisher \cite{Fis18} gives a method for finding the equation in all degrees,
and it has
been implemented in Magma for degree~$5$.  However, the implementation
did not conclude
within a reasonable time in the example that arose.  For degree $6$, even
this resource is not available.

\begin{rmk}\label{rem:elliptic-param}
  In all four cases, there is an elliptic parameter for the $\II^*$ fibration
  whose divisor is of the form $d(\sigma_1 - \sigma_0) + F$, where $d$ is
  the torsion order of the semistable fibration and $\sigma_1-\sigma_0$
  generates the torsion group, while $F$ is a fibral divisor for this
  fibration.  We do not have a unified explanation for this fact.  If we
  were also to consider the curves $E$ with torsion subgroup
  $\Z/3\Z \oplus \Z/3\Z$ and $\Z/4\Z \oplus \Z/2\Z$, isogenous to the first
  two considered below, we would not find fibrations with two $\II^*$ fibres
  on $\mathcal{Z}^2(\E)$.  The reason for this is that the discriminant
  of the Picard lattice is $t^2$, where $t$ is the torsion order, and so
  the multisection degree would have to be $t$.  However, the discriminant
  group has no elements of order $t$, and therefore no such fibration exists
  by Lemma~\ref{lem:keum}.
\end{rmk}

\subsubsection{The $\I_9\I_1\I_1\I_1$ fibration, with Mordell-Weil group $\Z/3\Z$}

The reducible fibers of $\caZ^2(\E)\to\P^1\times U^2$ are of type $\I_9,\I_9$, and these together with the trivial section and a section of order 3 contain two $\II^\ast$ configurations as shown below:

\begin{center}
\includegraphics[scale=.75]{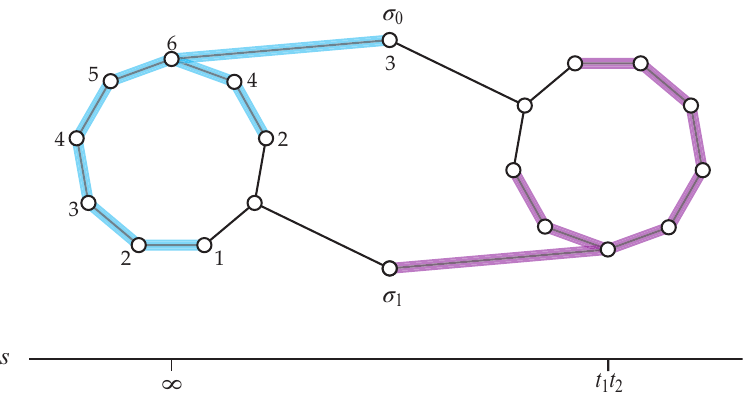}
\end{center}

The genus 1 fibration on $\caZ^2(\E)$ with these two $\II_2^\ast$ fibers has a degree 3 multisection, represented by any of the uncolored vertices (or the other torsion section which is not shown).   Therefore $\caZ^2(\E)$ admits a degree 9 map to the Jacobian of the fibration, which is an Inose surface. 

In this case, we were able to make all of the equivalences explicit.  The
first step is to write down the elliptic parameter for the fibration with the
two $\II^*$ fibers by exhibiting a function whose divisor is described by
the figure.  Computationally, we approached this by first writing down a
projective model for $\caZ^2(\E)$ in $\P^6$ with five $A_1$ singularities
and two $A_4$ singularities such that each of the two $I_9$ fibres consists
of three curves of degree $1$, two $A_1$ points, and one $A_4$ point.
The other $A_1$ singularity is the zero section of the fibration, while the
$3$-torsion sections are curves of degree $2$ in this model.  (It may be
surprising that such a nice model exists for a K3 surface of such large
Picard rank and small discriminant.)  In this model it is straightforward
to exhibit, not only the elliptic parameter $t$, but also three
functions $f_0, f_1, f_2 = 1$ whose restrictions to a smooth fiber generate
the Riemann-Roch space of $\OO(D)$, where $D$ is the multisection of degree
$3$.

Once this was done, we wanted to find the image of the map
$(f_0:f_1:f_2),(t:1)$ from $\caZ^2(\E)$ to $\P^2 \times \P^1$, since the
general fibre of the image of the map to $\P^1$ would be the cubic whose
Jacobian has the two $\II^*$ fibers.  This proved to be computationally
difficult and we resorted to interpolation; however, the final result is
rigorous, because we could verify the equations of the map and its codomain
in Magma once we had found them.  At that point it was a routine matter to
use the formulas of \S\ref{sec:MasatoAppendix}, in particular
\ref{eq:Inose}, to verify that up to a change of coordinates and twist
the Inose surface is isomorphic to $\Km(E_1 \times E_2)$, where the $E_i$
are obtained by substituting $t_i$ for $t$ in the equation defining an
elliptic curve over $\Q(t)$ with four bad fibres of type $\I_3$ and no others.
(This is isogenous to the curve with one $\I_9$ and three $\I_1$ fibres.)
See \cite{code} for details.  

\subsubsection{The $\I_8\I_2\I_1\I_1$ fibration, with Mordell-Weil group $\Z/4\Z$} The reducible fibers of $\caZ^2(\E)\to\P^1\times U^2$ are of type $\I_8,\I_8,\I_2,\I_2$, and these together with the trivial section and a section of order 4 contain two $\II^\ast$ configurations as shown below:

\begin{center}
\includegraphics[scale=.75]{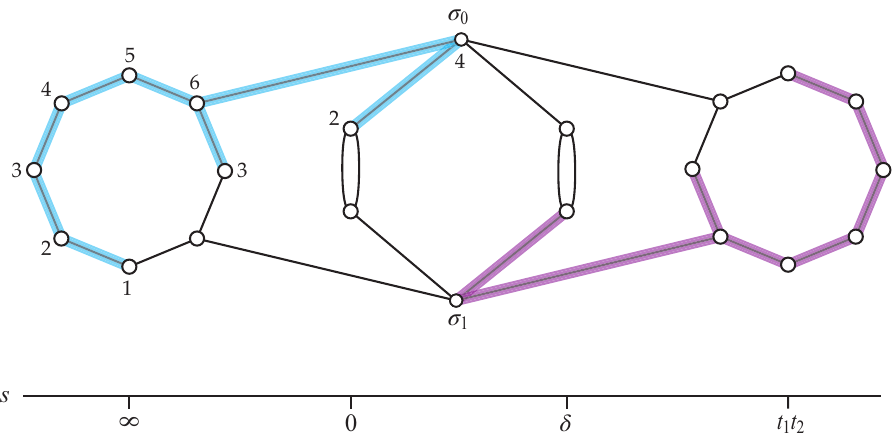}
\end{center}

The genus 1 fibration on $\caZ^2(\E)$ with these two $\II_2^\ast$ fibers has a degree 4 multisection, represented by any of the uncolored vertices (or the other two torsion sections which are not shown).  Therefore $\caZ^2(\E)$ admits a degree 16 map to the Jacobian of the fibration, which is an Inose surface.

\subsubsection{The $\I_5\I_5\I_1\I_1$ fibration, with Mordell-Weil group $\Z/5\Z$}

Assume that $\car k\neq 2$.  
Let $\E\to \P^1_k$ be the elliptic fibration with Weierstrass equation:
\begin{equation}
\label{EqEI5}
y^2 = x^3 + (t^2 + 1) x^2 - 4 t (t^2 + t - 1) x + 4 t^2 (t^2 + 1)
\end{equation}
The reducible fibers of $\E\to \P^1_k$ are of type $\I_5,\I_5$ at $t=0,\infty$. The Mordell-Weil group is $\Z/5\Z$, generated by $\sigma_{1}=(2t, 4t)$. 

The reducible fibers of $\caZ^2(\E)\to\P^1\times U^2$ are of type $\I_5,\I_5,\I_5,\I_5$, and these together with the $5$-torsion section $\sigma_{1}$ and the $0$-section $\sigma_{0}$ contain two $\II^\ast$ configurations as shown below:

\begin{center}
\includegraphics[scale=.75]{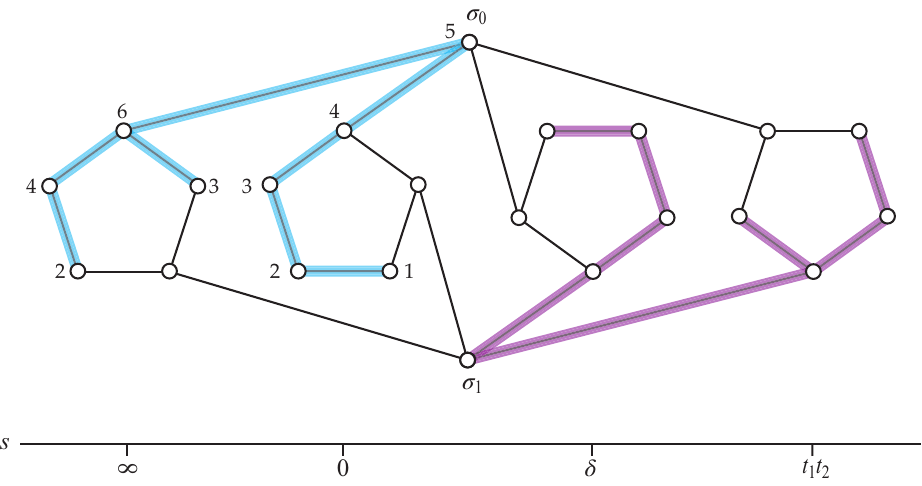}
\end{center}

The genus 1 fibration on $\caZ^2(\E)$ with these two $\II^\ast$ fibers has a degree 5 multisection, represented by any of the uncolored vertices (or the other four torsion sections which are not shown).  Therefore $\caZ^2(\E)$ admits a degree 25 map to the Jacobian of the fibration, which is an Inose surface.

\subsubsection{The $\I_6\I_3\I_2\I_1$ fibration, with Mordell-Weil group $\Z/6\Z$} 

The reducible fibers of $\caZ^2(\E)\to\P^1\times U^2$ are of type $\I_6,\I_6,\I_3,\I_3,\I_2,\I_2$, and these together with the trivial section and a section of order 6 contain two $\II^\ast$ configurations as shown below:

\begin{center}
\includegraphics[scale=.75]{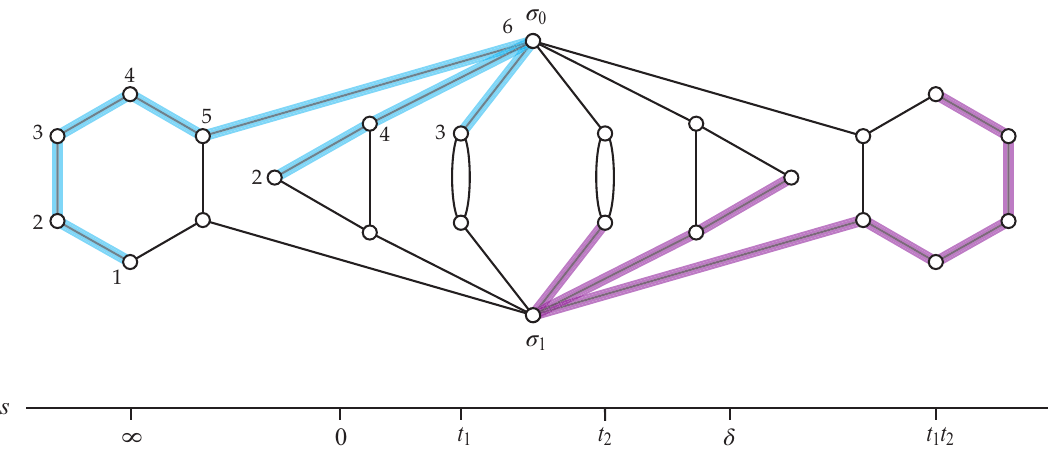}
\end{center}

The genus 1 fibration on $\caZ^2(\E)$ with these two $\II_2^\ast$ fibers has a degree 6 multisection, represented by any of the uncolored vertices (or the other four torsion sections which are not shown).  Therefore $\caZ^2(\E)$ admits a degree 36 map to the Jacobian of the fibration, which is an Inose surface.

\section{Calabi-Yau threefolds and 3-modularity}

We discuss here the problem of proving that an elliptic fibration is 3-modular, at least for the unstable extremal rational elliptic surfaces of Proposition \ref{PropClassificationUnstable}.  Following the technique of proof for the case of 2-modularity, we consider a coincidence variety
\[ \Coinc_G^4(\Gamma_0(N);\Sigma_\infty)\to (\P^1_F)^4 \]
for an arbitrary algebraically closed field $F$.
As we saw in Example \ref{ex:coinc-4}, the generic fiber of this morphism is an affine open in an elliptic curve.  Let $\eta_3\isom \Spec F(t_1,t_2,t_3)$ be the generic point of $(\P^1_F)^3$, and let $\Coinc_G^4(\Gamma_0(N);\Sigma_\infty)_{\eta_3}$ be the fiber over $\eta_3$ in the projection onto the first three coordinates.  Then projection onto the remaining coordinate $t=t_4$ defines an elliptic fibration with affine part $\Coinc_G^4(\Gamma_0(N);\Sigma_\infty)_{\eta_3}$, which we shall call $\caC$.

The Weierstrass equation for $\caC$ is  
\[ y^2 +e_1xy = x^3 + (-e_2+e_3-2e_4)x^2 + (1-e_1+e_2-e_3+e_4)e_4x,\]
where $e_1,\dots,e_4$ are the elementary symmetric polynomials in
$t_1,t_2,t_3,t_4$, and we take $t=t_4$ as the parameter on the base. 
As polynomials in $t$, the coefficients $a_i(t)$ of the fibration satisfy $\deg a_i\leq i$, so that $\caC$ is (at least over an algebraic closure of $\eta_3$) a rational elliptic surface.  We compute that $\caC\to\P^1_{\eta_3}$ has one singular fiber of type $\I_4$ at $t=\infty$, two of type $\I_2$ at $t=0,1$, and four of type $\I_1$ at other points.  Its Mordell-Weil group is $\Z^{\oplus 3} \oplus \Z/2\Z$.

\begin{rmk}  The fibration $\caC\to\P^1_{\eta_3}$ is the universal fibration of type No. 21 in Oguiso-Shioda's tables \cite{Oguiso-Shioda:RES} classifying all rational elliptic fibrations.
\end{rmk}

Now let $\E\to \P^1_F$ be a rational elliptic fibration with multiplicative fibers at $0,1$ and additive fiber at $\infty$.  Define a projective variety $\caZ^3(\E)$ over $\eta_3$ as the fiber product:
\begin{equation}
\label{DefnZrE}
\xymatrix{
\caZ^3(\E) \ar[r] \ar[d] & \E \times\eta_3\ar[d] \\
\caC \ar[r]  & \P^1_{\eta_3} 
}
\end{equation}

The technique of proof of Theorem \ref{ThmMain2Modularity} reduces the 3-modularity of such elliptic fibrations to the following conjecture.

\begin{conj} \label{ConjZ3}
There exists an algebraic correspondence between $\caZ^3(\E)$ and $\E^3_{\eta_3}$, such that the induced map $H^3(\caZ^3(\E))\to H^3(\E^3_{\eta_3})$ is surjective over the transcendental part of $H^3(\E^3_{\eta_3})$.   Here the $H^3$ can refer to $\ell$-adic cohomology as a Galois representation, or (if $F=\mathbf{C}$) Betti cohomology as a family of Hodge structures.
\end{conj}

We were able to prove Conjecture \ref{ConjZ3} in the case that $\E$ is the Legendre fibration, by way of a study of Calabi-Yau threefolds.

\subsection{Fiber products of two elliptic fibrations} 

The fiber product of two rational elliptic fibrations $S_1,S_2\to\P^1$ over a common base was studied in \cite{Schoen}, as a means of constructing interesting Calabi-Yau threefolds.  To see why such a threefold might result, let us model $S_i$ as an equation of bidegree $(3,1)$ in $\P^2\times \P^1$.  Then the fiber product $S=S_1\times_{\P^1} S_2$ is defined by equations of tridegree $(3,0,1),(0,3,1)$ in $\P=\P^2\times\P^2\times\P^1$, and so, by the adjunction formula, its canonical divisor is $K_S=(K_{\P})\vert_S \otimes \OO(3,3,2)=\OO_S$. 

The fiber product $S$ generally has singular points; to construct (nonsingular) Calabi-Yau threefolds, we must determine whether the canonical divisor is trivial on a desingularization of $S$.  Recall that a {\em crepant resolution} $f\from \tilde{S}\to S$ is a desingularization which does not affect the canonical class:  $f^*K_S=K_{\tilde{S}}$.  Thus if our fiber product $S=S_1\times_{\P^1} S_2$ admits a crepant resolution, then its desingularization is a Calabi-Yau threefold.

It is not hard to see that the only possible singularities of $S$ occur over those points $v\in \P^1$ where both $S_v$ and $S'_v$ are singular.  It is shown in \cite{Schoen} that as long as $S_v$ and $S_v'$ are multiplicative, then all singular points of $S$ are ordinary double points, and consequently $S$ admits a crepant resolution.  

However, in the fiber product defining $\caZ^3(\E)$, the fibers of the two fibrations $\caC,\E\to\P^1$ at $\infty$ are of multiplicative and additive type, respectively, and so we are not in the situation considered in \cite{Schoen}.  Nonetheless, the fiber product still has a crepant resolution:

\begin{prop}  Let $S_1,S_2\to X$ be two elliptic fibrations.  Suppose there is no point of $X$ at which the fibers of $S_1$ and $S_2$ are both additive.  Then $S=S_1\times_X S_2$ admits a projective crepant resolution.
\end{prop} 

\begin{proof} (Sketch.) Consider a singular point $P=(P_1,P_2)$ of $S$.  Then $P_i$ is a singular point of a fiber $(S_i)_v$ of $S_i$ for $i=1,2$.  In light of \cite{Schoen} it suffices to assume that $(S_1)_v$ is additive and $(S_2)_v$ is multiplicative.  We will only discuss the case that $(S_1)_v$ has an $\I_n^*$ fiber (indeed this is the only case we will use).  Then $P_1$ falls into one of the following three cases: it belongs to a single nonreduced component of multiplicity 2, it is the intersection of such a component with a reduced component, or else it is the intersection of two nonreduced components.  

The local equations of $P$ for these three types are 
$$x_1^2 - x_3 x_4 = 0, \quad x_1^2 x_2 - x_3 x_4 = 0, \quad
x_1^2 x_2^2 - x_3 x_4 = 0,$$
as hypersurfaces in $\mathbf{A}^4$.  
(In the first case, for instance, let $t$ be a local coordinate for $X$ at $v$.  Then the local equations for $P_1,P_2$ are $x_1^2=t$ and $x_3x_4=t$, respectively, and so the fiber product has equation $x_1^2=x_3x_4$.)  We analyze each case in turn.

First consider $x_1^2 - x_3 x_4 = 0$ as a hypersurface in $\mathbf{A}^4$.  This is a product of $\mathbf{A}^1$ with a surface with an $A_1$ singularity.  This hypersurface has canonical singularities in the sense of \cite{Reid};  we confirm that it has a crepant resolution. 

Let $Q\subset\P^4$ be the projective closure of the hypersurface.  Then $Q$ is the cone over the cone over a smooth conic.  Let $\pi: \tilde{Q}\subset \P^4\times\P^2 \to Q$ be the blowup of $Q$ along its singular locus $S$, which is a line in $\P^4$.  A simple calculation shows that $\tilde{Q}$ is projective and smooth, and that the exceptional divisor $E$ is isomorphic to $(\P^1)^2$. We claim also that $\pi: \tilde{Q} \to Q$ is a crepant resolution of $Q$.

Let $\pi_1, \pi_2$ be the two projections $(\P^1)^2 \to \P^1$.  Identifying
$E$ with $(\P^1)^2$ and $S$ with $\P^1$, we identify $\pi_1$ with $\pi\vert_E$.
Let $p$ be a point of the singular locus of $Q$ and let $F = \pi^{-1}(p)$.
Let $G$ be a fibre of $\pi_2$, viewed as a curve on $E$.
The canonical divisor $K_Q$ of the quadric hypersurface $Q\subset \P^4$ is $K_Q=\OO(-3)$, so we'd like
  to know that the canonical divisor $K_{\tilde{Q}}$ of $\tilde Q\subset \P^4\times\P^2$ is $\pi^*K_Q=\OO(-3,0)$.  We have
  $h^2(\tilde Q) = 2$, which forces $K_{\tilde{Q}}=\pi^*K_Q + cE$. 
  To determine $c$, we use the adjunction formula.  We have
  $(K_{\tilde Q}+E) \cdot E = K_E = -2F - 2G$.  Clearly $\OO(1,0)$ misses $F$
  (a general hyperplane doesn't pass through a given point) and hits
  $G$ once (in the fibre above the point of intersection in $\P^4$), so the
  intersection is $F$.
  On the other hand, if we take a section $S$ of $\OO(1,0)$ containing the
  exceptional divisor (i.e., the strict transform of a hyperplane containing
  the singular locus of the quadric), then we can compute that
  $F \sim S \cdot E = (R+E)\cdot E = 2G + E^2$, so $E^2 = F - 2G$.  It
  follows that $K_{\tilde Q} \cdot E = -3F$, and so $c = 0$ as desired.

We now consider the second case where the local equation is
$x_1^2 x_2 - x_3 x_4 = 0$.
The singular subscheme of the affine scheme defined
by this equation has two components: a line $x_1 = x_3 = x_4 = 0$, and an
embedded component $x_1^2 = x_2 = x_3 = x_4 = 0$.  If we blow up the first
of these (after projectivizing, to make the calculation easier) we obtain
a variety whose singular subscheme misses the locus
$x_1 = x_2 = x_3 = x_4 = 0$.  In other words, the single blowup has resolved
the singularity.  In codimension $1$ this is the same as the previous example,
and it follows that this is likewise a crepant resolution.

Finally, in the third case, the local equation is $x_1^2 x_2^2 - x_3 x_4 = 0$.
There are three components of the singular subscheme: two of the form
$x_i = x_3 = x_4 = 0$ for $i = 1, 2$, and the embedded component
$x_1^2 = x_2^2 = x_3 = x_4$.  Blowing up the first component takes care of
the embedded component and leaves a line of $A_1$ singularities, which we
have already seen to have a crepant resolution.  Note also that in this
case we have converted two rational curves into divisors, which should
increase $h^2$ and $h^4$ by $2$ each, whereas in the previous examples
the contribution was only $1$.  Further, since the resolution method is an
actual blowup, rather than a small resolution, projectivity is automatic.
\end{proof}

\begin{rmk} Similar considerations show a stronger result:  The fiber product $S_1\times_X S_2$ admits a crepant resolution if and only if, when two additive fibers come together at the same point of $X$, the fiber types belong to the following list:  $$(\I_n^*,\II), (\I_0^*,\III), (\IV^*,\II), (\IV,\II), (\IV,\III), (\III,\III), (\III,\II),
(\II,\II).$$
Most of the other cases can be excluded by the observation that
if the fibers both contain a nonreduced component,
then there cannot be a crepant resolution.  The local equation is
$x_2^i - x_4^j = 0$ for $i, j > 1$, and this is singular in dimension $2$,
i.e., along a divisor. 
\end{rmk}

\begin{cor} \label{CorZ3isCY} The fiber product $\caZ^3(\E)$ is birational to a nonsingular projective Calabi-Yau threefold.
\end{cor}


\subsection{The Kummer threefold}

On the other hand, there is a separate Calabi-Yau threefold that we can associate to three elliptic curves $E_1,E_2,E_3$;  this is a generalization of the Kummer surface associated to the product of two elliptic curves.

\begin{defn}
  Let $Q=(E_1\times E_2\times E_3)/V$, where nontrivial elements of the group $V=(\Z/2\Z)^{\oplus 2}$ act by negating two of the factors $E_1,E_2,E_3$ at a time.  
\end{defn}

\begin{prop} The threefold $Q$ admits a crepant resolution $\tilde{Q}\to Q$, where $\tilde{Q}$ is a Calabi-Yau threefold.  
\end{prop}

We refer to any such resolution as a {\em Kummer threefold} and denote it
by $\Km(E_1\times E_2\times E_3)$.

\begin{proof}  (Assuming $\car k\neq 2$.)  First we consider how to embed $Q$ into a toric 
variety.  We can view each $E_i$ as a hyperelliptic curve $y_i^2=f_i(x_i,z_i)$ in weighted projective space $\P(1,2,1)$ , where $f$ is homogenous of degree 4.  With respect to these coordinates, negation on $E_i$ is $(x_i:y_i:z_i)\mapsto (x_i:-y_i:-z_i)$.  

The product $E_1\times E_2\times E_3$ lives in $\P(1,2,1)^3$.  The group $V$ acts on $\P(1,2,1)^3$ in the evident matter.   The quotient $T=\P(1,2,1)^3/V$ has coordinates $y=y_1y_2y_3,x_1,z_1,x_2,z_2,x_3,z_3$;  it is the toric variety obtained by taking the quotient of $\mathbf{A}^7$ by $\Gm^3$ acting by the weights $(2,1,1,0,0,0,0)$, $(2,0,0,1,1,0,0)$, and $(2,0,0,0,0,1,1)$.  Finally, $Q$ is the hypersurface in $T$ with equation:
\[ y^2 = \prod_{i=1}^3 f_i(x_i,z_i).\]

Identify $\Pic T$ with $\Z^3$ by these three weight vectors.  
The canonical divisor of a toric variety is the negative of the sum of the toric divisors \cite[Proposition, section 4.3]{Fulton},
so $K_T=\OO(-4,-4,-4)$.  Since the defining equation for $Q$ has degree 4 with respect to each copy of $\Gm$, we have by the adjunction formula $K_Q = (K_T \otimes \OO(Q))\vert_Q = 0$.

It remains to show that $Q$ has a crepant resolution.  The subschemes of $E_1 \times E_2 \times E_3$ where a nontrivial element of $V$ has a fixed point are precisely those where two or three of the
  coordinates are points of order $1$ or $2$.
  Let $V$ act on $\A^3$ by negating any two of the coordinates.  This
  is a linear action by a subgroup of $\SL_3$, so by a theorem of Roan,
  Ito, and Markushevich \cite[Theorem 1]{Roan} there is a crepant resolution.
  The action of $V$ on the tangent spaces of fixed points of
  $E_1 \times E_2 \times E_3$ is the same as for $\A^3$, so the result
  carries over to our case.
\end{proof}

\begin{rmk} The crepant resolution of $Q$ is obtained by blowing up 48 curves.  These are the images of those curves in $E_1\times E_2\times E_3$ of the form $\set{P_1}\times \set{P_2}\times E_3,\set{P_1}\times E_2\times\set{P_3},E_1\times\set{P_2}\times\set{P_3}$, where $P_1,P_2,P_3$ are 2-torsion points.  The order of blowing up is significant: at each of the 64 points $(P_1,P_2,P_3)$, there are 6 possible orders to choose from, and exchanging two adjacent ones amounts to a flop.  However, for our purposes the choice makes no difference.
\end{rmk}

\begin{rmk} This construction generalizes to give a Calabi-Yau manifold $\Km(E_1\times\cdots\times E_d)$ for $d$ elliptic curves $E_1,\dots,E_d$.  This is the desingularization of the quotient $(E_1\times \cdots E_d)/V$, where now $V\subset (\Z/2\Z)^{\oplus d}$ is the subgroup where the sum of the coordinates is 0.
\end{rmk} 

\begin{prop} The Hodge diamond of $\Km(E_1\times E_2\times E_3)$ is 
\begin{center}
\begin{tabular}{ccccccc}
&&&$1$&&&\\
&&$0$&&$0$&&\\
&$0$&&$51$&&$0$&\\
$1$&&$3$&&$3$&&$1$\\
&$0$&&$51$&&$0$&\\
&&$0$&&$0$&&\\
&&&$1$&&&
\end{tabular}
\end{center}
\end{prop}

\begin{proof} We assume $k=\mathbf{C}$ to give a simple proof in terms of differential forms.
Let $H^{1,0}$ and $H^{0,1}$ of $E_i$ be spanned by $z_i, \bar z_i$ respectively.
Then $H^1(E_1 \times E_2 \times E_3)$ is spanned by all the $z_i, \bar z_i$
(abusively using the same notation for forms on the $E_i$ and their pullbacks
to the product) and $H^n(E_1 \times E_2 \times E_3)$ is identified with
$\bigwedge^n H^1$, as usual for an abelian variety.  We can write
$H^n = \oplus_{j=0}^n H^{j,n-j}$, where $H^{j,n-j}$ is spanned by the products
of $j$ holomorphic and $n-j$ antiholomorphic forms.

The negation map on $E_i$ negates $z_i, \bar z_i$ while fixing
the others.  Thus $V$ acts on $H^1, H^2, H^3$ with the following fixed
subspaces:
\begin{enumerate}
\item the fixed subspace on $H^1$ is trivial;
\item the fixed subspace on $H^2$ is spanned by the $z_i \wedge \bar z_i$;
\item the fixed subspace on $H^3$ is spanned by products of one form with
  each subscript.
\end{enumerate}
In order to obtain the Hodge diamond of $\Km(E_1\times E_2\times E_3)$, we must consider the
effect of the blowup.  Each blowup of a rational curve replaces a subvariety
with $h^{0,0} = h^{1,1} = 1$ by one with $h^{0,0} = h^{2,2} = 1$ and $h^{1,1} = 2$,
so it increases $h^{1,1}$ and $h^{2,2}$ by $1$.  Since there are $48$ such
curves, we obtain the values claimed in the statement of the proposition.
\end{proof}

\subsection{A birational map between a Kummer threefold and a fiber product of elliptic fibrations}

Assume that $\car k\neq 2$.  Let $\E\to\P^1_k$ be the Legendre fibration, with Weierstrass equation
\[ y^2 = x(x-1)(x-t). \]
Let $\eta=\Spec k(t_1,t_2,t_3)$ be the generic point of $(\P^1_k)^3$, and as usual write $\E^3_{\eta}=E_1\times E_2\times E_3$.  Thus $E_i$ is the elliptic curve over $\eta$ with Weierstrass equation $y^2=x(x-1)(x-t_i)$.  We are interested in two Calabi-Yau threefolds over $\eta$.    On the one hand,  we have the Kummer threefold $\Km(E_1\times E_2\times E_3)$, which is birational to the subvariety of the toric variety $T=\mathbf{A}^7/\mathbf{G}_m^3$ with equation
\begin{equation}
\label{EqForKm3fold}
y^2 = \prod_{i=1}^3 x_iz_i(x_i-t_iz_i)
\end{equation}

On the other hand, we have the fiber product $\caZ^3(\E)=\caC\times_{\P^1_{\eta}} \E_{\eta}$, which is birational to a Calabi-Yau variety by Corollary \ref{CorZ3isCY}.  

\begin{thm} \label{ThmLegendreBirational} The varieties $\caZ^3(\E)$ and $\Km(E_1\times E_2\times E_3)$ are birational over $\eta$.  Therefore Conjecture \ref{ConjZ3} is true for $\E$, and (if $k=\F_q$) then $\E$ is 3-modular.
\end{thm}

\begin{proof}  We found a birational equivalence between $\caZ^3(\E)$ and $\Km(E_1\times E_2\times E_3)$ by exhibiting K3 surface fibrations on each threefold, such that the generic fibers of the fibrations are isomorphic.  The discovery of the birational equivalence was extraordinarily  serendipitous:  we reached for a few of the easiest possible K3 fibrations to construct on either side, and happened upon two pairs that matched.  For the birational equivalence in terms of coordinates, we refer to the code \cite{code}. We only describe here the method we used to find it.

The K3 surface fibration on $\Km(E_1\times E_2\times E_3)$ is easy enough to describe.  Recall that $\Km(E_1\times E_2\times E_3)$ is the desingularization of a subvariety $Q$ of a toric variety $T$, with equation \eqref{EqForKm3fold}.  Consider the rational map $T\dashrightarrow \P^1$ defined by the ratio $(y:x_1x_2x_3z_1z_2z_3)$.  We claim that its restriction to $Q\dashrightarrow \P^1$ has generic fiber a K3 surface.  Indeed, the fiber over $u\in\P^1$ is the subvariety of $(\P^1)^2$ with equation
\begin{equation}
\label{EqK3FibrationOnKummer3Fold}
 u^2\prod_{i=1}^3 x_iz_i = \prod_{i=1}^3 (x_i-z_i)(x_i-t_iz_i), 
 \end{equation}
 which one checks is nonsingular for generic $u$.  A nonsingular surface in $(\P^1)^3$ of degree $(2,2,2)$ is a K3 surface.  This is our K3 fibration $\Km(E_1\times E_2 \times E_3)\dashrightarrow \P^1$.  

Now consider $\caZ^3(\E)$, which is the fiber product of two rational elliptic fibrations over a common base.  This is naturally a subvariety of $\P^2\times \P^2\times \P^1$.  Let $S$ be the set of prime components of the singular subscheme of $\caZ^3(\E)$ of dimension 1, and consider the linear system of $(1,1,1)$-forms vanishing on $S$.  Our computations showed that this linear system gives a biratonal map from $\caZ^3(\E)$ to a quintic threefold $Z\subset \P^4$.  

The quintic $Z$ is singular along $8$ lines and some isolated points.  By considering spans of pairs of these lines, we found 9 planes $H\subset Z$.  If such a plane has equations $\ell_0=\ell_1=0$ for linear forms $\ell_0,\ell_1$ on $\P^4$, then the quintic defining $Z$ can be written $\ell_0f_0+\ell_1 f_1$ for forms $f_0,f_1$ of degree 4.  Then $[\ell_0:\ell_1]$ defines a rational map $Z\dashrightarrow \P^1$, whose generic fiber is a quartic in $\P^3$; i.e., a K3 surface.  

Experimentally, we counted points of fibers of these $Z\dashrightarrow \P^1$ over a finite field (choosing random values for $t_1,t_2,t_3$, until we found one which matched the point counts from fibers of $\Km(E_1\times E_2\times E_3)\dashrightarrow \P^1$.  It was then possible to change the coordinate $u$ on the $\P^1$ so that the point count matched fiber by fiber.  This suggested that the generic fibers of the two fibrations, which are K3 surfaces over $k(t_1,t_2,t_3,u)$, were isomorphic, and indeed they were.  We verified this by finding elliptic fibrations with the same configuration of singular fibers, and then directly observing that those elliptic fibrations are isomorphic.
\end{proof}

\begin{rmk}\label{rem:just-lucky}
  In contrast to the theory of elliptic fibrations on a K3 surface, which
  is well developed and draws on the very rich theory of lattice genera,
  K3 surface fibrations on Calabi-Yau varieties are not well understood.
  Such fibrations are described by certain extremal rays of the ample
  cone.  However, there is no general theory that allows one to construct a
  set of orbit representatives for such extremal rays with respect to the
  action of the automorphism group.  In any case, this would not be sufficient,
  because we might need to consider not only K3 fibrations on one particular
  Calabi-Yau model but on all equivalent ones.  Though there are finitely
  many Calabi-Yau varieties birational to a given one up to isomorphism,
  it is not clear how to determine them all in practice or approximately
  how many there should be for a variety such as $Q$.
  It is known (the relevant facts are nicely summarized in
  \cite[Introduction]{Fryers}) that the ample cones of different Calabi-Yau
  models exhaust the movable cone, but one expects that for a Calabi-Yau
  variety of Picard number $51$ such as $Q$ the geometry of this partition
  would be exceedingly complicated.  For example, we can blow up the $48$
  curves of singularities in any order (though blowups in two disjoint curves
  commute).  The group of automorphisms
  acting on this set of Calabi-Yau models is small,
  and presumably there are many other ways to obtain models as well.

  These problems have discouraged us from attempting to extend the result of
  Theorem~\ref{ThmLegendreBirational} to other examples such as the elliptic
  surface with singular fibers $\III^*, \I_2, \I_1$.  We have verified numerically
  that the $\F_p$-rational point counts match mod $p$ and therefore expect
  a correspondence.  If the analogy with surfaces holds, we might hope that
  there is a K3 surface fibration on each side such that the corresponding
  fibers are isogenous (just as in our proof of $2$-modularity for this
  family we find genus $1$ fibrations on both K3 surfaces whose fibers have
  the same number of points and that are therefore related by an isogeny).
  However, at our present level of understanding there is no possibility of
  finding it except by a lucky stab into a potentially enormous space of
  fibrations.  The other families with one additive and two multiplicative
  fibers are more daunting still, since we would not expect any single
  fibration to be sufficient; we would need a third isogenous Calabi-Yau
  threefold to mediate between the two sides.
\end{rmk}

\bibliographystyle{amsalpha}
\bibliography{bibfile}

\end{document}